\documentclass[leqno,12pt]{amsart}
\usepackage{amssymb}
\usepackage{amsmath}
\usepackage{enumerate}
\usepackage{amsfonts}
\usepackage{hyperref}
\usepackage{mathrsfs}
\usepackage{tikz}

\usepackage[headheight=18pt, top=25mm, bottom=25mm, left=20mm, right=20mm]{geometry}

\hypersetup{
    colorlinks=true,
    linkcolor= blue,
    citecolor =cyan,
    urlcolor = teal,
}

\newtheorem{theorem}{Theorem}[section]
\newtheorem{corollary}[theorem]{Corollary}
\newtheorem{lemma}[theorem]{Lemma}

\newtheorem{remark}[theorem]{Remark}
\newtheorem{definition}[theorem]{Definition}

\numberwithin{equation}{section}

\begin{document}
\title[Lipschitz Spaces in Dunkl--Schrödinger Analysis] {On Lipschitz Spaces, Semigroup Transference, and Heat Kernel Derivative Estimates in Dunkl--Schrödinger Analysis}

\author[Agnieszka Hejna-Łyżwa]{Agnieszka Hejna-Łyżwa}

\subjclass[2020]{{primary: 42B35, 35J10; secondary: 33C52, 47D06, 35K08, 46E35}}

\keywords{Rational Dunkl theory, Schr\"odinger operators, Reverse H\"older classes, Lipschitz spaces.}

\begin{abstract}
In this paper, we study Lipschitz spaces adapted to Dunkl--Schr\"odinger operators $\mathcal{L} = -\Delta_k + V$, where $\Delta_k$ is the Dunkl Laplacian and $V \ge 0$ belongs to the reverse H\"older class ${\rm RH}^q(dw)$ with $q > \max(1, \mathbf{N}/2)$. For $0 < \beta < 2$, we establish the full equivalence between the heat semigroup-based Lipschitz spaces $\widetilde{\Lambda}_{\mathcal{L},k}^{\beta/2}(\mathbb{R}^N)$ and the pointwise weighted Zygmund classes $\Lambda_{\mathcal{L},k}^{\beta}(\mathbb{R}^N)$ defined via the critical radius function $m(\mathbf{x})$. Remarkably, this characterization shows that the spatial smoothness component decouples completely from the underlying reflection group symmetries and root systems. As an intermediate result, we establish new equivalences for inhomogeneous weighted Lipschitz spaces associated with the unperturbed Dunkl Laplacian $\Delta_k$. A central difficulty in this context—arising from the necessity to disentangle the Euclidean distance from the orbit metric $d(\mathbf{x},\mathbf{y})$—is resolved through a detailed geometric analysis. Finally, we provide pointwise Gaussian upper bounds for the time derivatives of the Dunkl--Schr\"odinger heat kernel, which are of independent interest for future work in harmonic analysis on Dunkl structures. Crucially, we do not assume the potential $V$ to be $G$-invariant under the action of the Weyl group, which introduces major geometric difficulties resolved here via a local analysis of the orbit metric. 
\end{abstract}
\address{A. Hejna-Łyżwa, Uniwersytet Wroc\l awski,
Instytut Matematyczny,
Pl. Grunwaldzki 2,
50-384 Wroc\l aw,
Poland}
\email{agnieszka.hejna@math.uni.wroc.pl}

\enlargethispage{0.37cm}
\maketitle

\thispagestyle{empty}

\section{Introduction}

\subsection{Motivation and the state of the art}
Lipschitz spaces play an important role in harmonic analysis and partial differential equations, serving as a natural framework to quantify the precise smoothness  of functions. For $0<\beta<1$, the classical inhomogeneous Lipschitz space $\Lambda^\beta(\mathbb{R}^N)$ on the Euclidean space $\mathbb{R}^N$ is defined as the set of all functions $f: \mathbb{R}^N \to \mathbb{C}$ such that the norm
\begin{equation}\label{eq:Lip_class}
    \|f\|_{\Lambda^\beta(\mathbb{R}^N)} := \|f\|_{L^\infty(\mathbb{R}^N)} + \sup_{\mathbf{x} \neq \mathbf{x}'} \frac{|f(\mathbf{x}) - f(\mathbf{x}')|}{\|\mathbf{x} - \mathbf{x}'\|^\beta}
\end{equation}
is finite.

In classical harmonic analysis (see, for example, Taibleson \cite{Taibleson1} and Stein \cite{Stein_singular}), these spaces can be  characterized dynamically using semigroups. Specifically, instead of the Poisson integral, one can utilize the heat semigroup $\{e^{t\Delta}\}_{t > 0}$ to extend the notion of Lipschitz spaces to all positive parameters $\beta > 0$. For any $\beta > 0$, let $n$ be the smallest integer strictly greater than $\beta/2$. We say that $f \in \Lambda^\beta(\mathbb{R}^N)$ if the quantity
\begin{equation}\label{eq:heat_cond_1}
    \|f\|_{L^\infty(\mathbb{R}^N)} + \sup_{t>0} \, t^{n-\beta/2} \left\| \partial_t^n e^{t\Delta} f \right\|_{L^\infty(\mathbb{R}^N)}
\end{equation}
is finite, and \eqref{eq:heat_cond_1} serves as an equivalent norm for the space $\Lambda^\beta(\mathbb{R}^N)$ (see \cite[Theorem 7]{Taibleson1}).

Furthermore, it is a well-known result (see \cite[Theorem 4]{Taibleson1}) that for $0 < \beta < 2$, the semigroup norm \eqref{eq:heat_cond_1} (with any fixed integer $m \geq 1$ such that $m > \beta/2$) is equivalent to the generalized Zygmund condition involving second-order differences:
\begin{equation}\label{eq:Zygmund_cond}
    \|f\|_{L^\infty(\mathbb{R}^N)} + \sup_{\mathbf{x} \in \mathbb{R}^N} \sup_{0 \neq \mathbf{y} \in \mathbb{R}^N} \frac{|f(\mathbf{x}+\mathbf{y}) + f(\mathbf{x}-\mathbf{y}) - 2f(\mathbf{x})|}{\|\mathbf{y}\|^\beta} < \infty.
\end{equation}
Spaces defined via \eqref{eq:Zygmund_cond} can naturally be viewed as inhomogeneous generalized Zygmund classes.

For higher regularities, the spaces exhibit a clear inductive structure. If $\beta > 1$, then $f \in \Lambda^\beta(\mathbb{R}^N)$ if and only if $f \in L^\infty(\mathbb{R}^N)$ and its first-order partial derivatives satisfy $\partial_j f \in \Lambda^{\beta-1}(\mathbb{R}^N)$ for $j = 1, 2, \dots, N$ (see, e.g., \cite[Chapter V, Proposition 9]{Stein_singular}). In this case, we have the norm equivalence
\begin{equation}\label{eq:step-less}
    \|f\|_{\Lambda^\beta(\mathbb{R}^N)} \sim \|f\|_{L^\infty(\mathbb{R}^N)} + \sum_{j=1}^N \|\partial_j f\|_{\Lambda^{\beta-1}(\mathbb{R}^N)}.
\end{equation}

In recent years, significant progress in the study of Lipschitz spaces has been achieved within the framework of Dunkl theory. Dunkl analysis originates from the differential-difference operators introduced by C. F. Dunkl \cite{Dunkl} in the late 1980s to study special functions and spherical harmonics associated with finite Coxeter groups. On $\mathbb{R}^N$ equipped with a root system $R$ and a nonnegative multiplicity function $k \ge 0$, the Dunkl operators are defined by
\begin{equation}\label{eq:Dunkl_op}
    T_j f(\mathbf{x}) = \frac{\partial f}{\partial x_j}(\mathbf{x}) + \sum_{\alpha \in R} \frac{k(\alpha)}{2} \langle \alpha, e_j \rangle \frac{f(\mathbf{x}) - f(\sigma_\alpha(\mathbf{x}))}{\langle \alpha, \mathbf{x} \rangle}, \quad j = 1, \dots, N,
\end{equation}
and they generate the Dunkl Laplacian $\Delta_k = \sum_{j=1}^N T_j^2$. Replacing classical partial derivatives with Dunkl operators allows for the incorporation of reflection symmetries into differential equations and function space theory. Beyond its rich algebraic structure related to Hecke algebras, Dunkl analysis plays a pivotal role in quantum physics—most notably in proving the quantum integrability of the Calogero--Moser--Sutherland many-body model, as well as in the study of anyons, supersymmetry, non-commutative geometry, and PT-symmetric quantum mechanics (see, e.g., \cite{DV}). Furthermore, it has found notable applications in probability theory, particularly in describing Feller processes with jumps \cite{RV1998}. In this setting, remarkable results were recently established in~\cite{DzHL}, namely:
\begin{enumerate}
    \item[(a)] \textbf{Equivalence with Classical Spaces:} The authors proved the unexpected theorem that the Dunkl-Lipschitz spaces $\Lambda_k^\beta$, defined via the decay of time derivatives of the Dunkl-Poisson or heat semigroups on $L^\infty$, coincide \textit{identically} with the classical Euclidean Lipschitz spaces $\Lambda^\beta(\mathbb{R}^N)$ for all $\beta > 0$, with equivalent norms. 
    \item[(b)] \textbf{Abstract Semigroup Framework:} To overcome the technical obstacles inherent to Dunkl operators, the authors introduced an abstract framework for Lipschitz spaces $\Lambda_{\mathcal{A}}^\beta$ defined on non-reflexive dual Banach spaces $X^*$ (such as $L^\infty = (L^1)^*$) via dual semigroups $\mathcal{T}_t^* = e^{t\mathcal{A}^*}$. This approach provided general operator-theoretic machinery—including holomorphic functional calculus, Bessel-type potential isomorphisms, and Peetre $K$-method interpolation—for establishing structural properties of semigroup-associated Lipschitz spaces in an abstract, unified manner.
\end{enumerate}

A fundamental question naturally arises: \textit{can further progress be achieved for Dunkl operators, and how far can this abstract semigroup framework be extended to encompass perturbed differential operators, such as Dunkl--Schrödinger operators $\mathcal{L} = -\Delta_k + V$?}

Before turning to the Dunkl setting, let us recall what is known in the classical case. Let us denote by $\mathcal{L} = -\Delta + V$ the Schr\"odinger operators, where $V \ge 0$ is a non-negative potential satisfying a reverse H\"older inequality $RH_q$ for some $q > N/2$:
\begin{equation}\label{eq:RH}
    \left( \frac{1}{|B|} \int_B V(\mathbf{y})^q d\mathbf{y} \right)^{1/q} \le \frac{C}{|B|} \int_B V(\mathbf{y}) d\mathbf{y} \quad \text{for all balls } B \subset \mathbb{R}^N.
\end{equation}
Then we have the associated heat semigroup $W_t = e^{-t\mathcal{L}}$. A pivotal geometric quantity in this theory is the critical radius function $\frac{1}{m(\mathbf{x})}$, defined by
\begin{equation}\label{eq:critical_radius}
    \frac{1}{m(\mathbf{x})} = \sup \left\{ r > 0 : \frac{1}{r^{N-2}} \int_{B(\mathbf{x},r)} V(\mathbf{y}) d\mathbf{y} \le 1 \right\}.
\end{equation}
The integral defining the function $\mathbf{m}$  was introduced  by Ch. Fefferman (see \cite[p. 146, the assumption of the main lemma]{Fefferman}). The function  is then  used in the well-known Fefferman--Phong inequality (\cite[p. 146]{Fefferman}, see also Shen ~\cite{Shen2},~\cite[Lemma 1.9]{Shen}).
The function $\frac{1}{m(\mathbf{x})}$ acts as a local threshold: at spatial scales smaller than $\frac{1}{m(\mathbf{x})}$, the operator $\mathcal{L}$ behaves predominantly like the unperturbed Laplacian $\Delta$, whereas at scales larger than $\frac{1}{m(\mathbf{x})}$, the potential $V$ exerts a strong dampening effect.

In a paper~\cite{DeLeonTorrea}, M. de Le\'on-Contreras and J. L. Torrea introduced pointwise Lipschitz spaces $\Lambda^\beta_\mathcal{L}$ adapted to Schr\"odinger operators for $0 < \beta < 2$  via the norm
\begin{equation}\label{eq:Schrod_norm}
    \|f\|_{\Lambda^\beta_\mathcal{L}} := \left\| \rho(\cdot)^{-\beta} f(\cdot) \right\|_{L^\infty(\mathbb{R}^N)} + \sup_{0 \neq \mathbf{z} \in \mathbb{R}^N} \frac{\|f(\cdot+\mathbf{z}) + f(\cdot-\mathbf{z}) - 2f(\cdot)\|_{L^\infty(\mathbb{R}^N)}}{|\mathbf{z}|^\beta}.
\end{equation}

To establish a dynamical characterization, they also defined adapted Lipschitz spaces via time derivatives of operator semigroups. For $\beta > 0$ and $k = [\beta/2] + 1$, a function $f$ belonging to the heat-semigroup class $\Lambda_{\beta/2}^W$ satisfies a heat size condition ($\int_{\mathbb{R}^N} e^{-|\mathbf{x}|^2/t} |f(\mathbf{x})| d\mathbf{x} < \infty$ for all $t > 0$, and $\lim_{t \to \infty} \partial_t^\ell W_t f(\mathbf{x}) = 0$ pointwise for $\ell \ge 0$) alongside the decay estimate
\begin{equation}\label{eq:heat_semigroup_norm}
    \left\| \frac{\partial^k}{\partial t^k} W_t f \right\|_{L^\infty(\mathbb{R}^N)} \le C_\beta t^{-k + \beta/2} \quad \text{for all } t > 0.
\end{equation}
The space $\Lambda_{\beta/2}^W$ is endowed with the norm $\|f\|_{\Lambda_{\beta/2}^W} := S^W_\beta[f] + \|m(\cdot)^{\beta} f(\cdot)\|_{L^\infty(\mathbb{R}^N)}$, where $S^W_\beta[f]$ denotes the infimum of constants $C_\beta$ in \eqref{eq:heat_semigroup_norm}. A central achievement of their work is the equivalence theorem: for $0 < \beta \le 2 - N/q$, these three definitions coincide with equivalent norms:
\begin{equation}\label{eq:space_equivalence}
    \Lambda^\beta_\mathcal{L} = \Lambda_{\beta/2}^W.
\end{equation}
 This reveals that the adapted space $\Lambda^\beta_\mathcal{L}$ depends strictly on the critical radius function $\frac{1}{m(\mathbf{x})}$ and its underlying spatial geometry. No higher-order differential operators tailored to $\mathcal{L}$ are needed to capture the space structure up to order $\beta < 2$.

Connecting this to the results of~\cite{DzHL}, where the authors proved that unperturbed Dunkl--Lipschitz spaces $\Lambda_k^\beta$ collapse back to classical spaces $\Lambda^\beta(\mathbb{R}^N)$---leads directly to fundamental questions at the intersection of Dunkl analysis and potential theory:
\begin{itemize}
    \item \textit{Can Lipschitz spaces associated with Dunkl--Schr\"odinger operators $\mathcal{L}_k = -\Delta_k + V$ also be characterized purely in terms of the Dunkl critical radius function combined with classical H\"older/Zygmund conditions?}
    \item \textit{Can the abstract semigroup framework of~\cite{DzHL} be extended to accommodate perturbed operators $\mathcal{L}_k = -\Delta_k + V$, providing a unified operator-theoretic proof?}
\end{itemize}

This line of research builds upon a rich literature on function spaces adapted to Schr\"odinger operators with reverse H\"older potentials $V \in RH_q$. The foundational study of $BMO_\mathcal{L}$ spaces by Dziuba\'nski et al.~\cite{DGMTZ} laid the groundwork for defining Lipschitz and H\"older classes in the Schr\"odinger setting. For orders $0 < \beta < 1$, Bongioanni, Harboure, and Salinas in~\cite{Bongioanni2011} introduced pointwise classes $\Lambda^\beta_\mathcal{L}$ using first-order differences and the critical-radius weight $m(\mathbf{x})^{\beta}$, establishing dualities with Hardy spaces and equivalences with $BMO^\beta_\mathcal{L}$. Subsequently, Ma, Stinga, Torrea, and Zhang in~\cite{MaStingaTorreaZhang2012} characterized H\"older spaces $C^{0,\beta}_\mathcal{L}$ via Poisson semigroups. More recently, these techniques have been extended beyond standard Euclidean Laplacians: Chen and Zhang in~\cite{ChenZhang2022} extended the heat-semigroup characterization of adapted Lipschitz spaces to high-order Schr\"odinger operators $\mathcal{L} = (-\Delta)^m + V^m$, while the authors of~\cite{HongHouHu} established analogous equivalences on non-commutative stratified structures, specifically the Heisenberg group $\mathbb{H}^N$ equipped with a sub-Laplacian. In this direction, recent work in~\cite{Hejna_AMP},~\cite{Hejna_JFAA} establishing heat kernel bounds, fundamental solution estimates, and adapted Hardy spaces for Dunkl--Schr\"odinger operators with reverse H\"older potentials provides a highly promising framework for extending adapted Lipschitz space characterizations to this setting. It is also worth mentioning other works concerning Dunkl--Schrödinger operators in a more general context~\cite{AH},~\cite{AH2020}.

\subsection{Main results}

To state our results, we briefly recall the framework of Dunkl--Schr\"odinger operators; for complete details on Dunkl operators, measures, and kernel estimates, we refer the reader to Section~\ref{sec:preliminaries}. Let $\Delta_k = \sum_{j=1}^N T_j^2$ denote the Dunkl Laplacian acting on $L^2(dw)$, where $dw(\mathbf{x}) = \prod_{\alpha\in R}|\langle \mathbf{x},\alpha\rangle|^{k(\alpha)}\,d\mathbf{x}$ is the doubling weight measure of homogeneous dimension $\mathbf{N}$. The Dunkl--Schr\"odinger operator is defined by $\mathcal{L} = -\Delta_k + V$, where $V \ge 0$ belongs to the reverse H\"older class ${\rm RH}^q(dw)$ for $q > \max(1, \mathbf{N}/2)$ (see Section~\ref{sec:Schorodinger}). Central to this setting is the critical radius function $m(\mathbf{x})$, defined via
\begin{equation}\label{eq:m_def_intro}
    \frac{1}{m(\mathbf{x})} = \sup \left\{ r > 0 : \frac{r^2}{w(B(\mathbf{x},r))} \int_{B(\mathbf{x},r)} V(\mathbf{y})\,dw(\mathbf{y}) \le 1 \right\},
\end{equation}
 and the semigroup $K_t = e^{-t{L}}$ (see Section~\ref{sec:Schorodinger}).

We now present the pointwise and heat semigroup-based definitions of adapted Lipschitz spaces.

\begin{definition}\label{def:dunkl_lipschitz_diff}
Let $0 < \beta < 2$, and let $m(\mathbf{x})$ denote the critical radius function associated with the potential $V$. The pointwise Lipschitz--Zygmund space $\Lambda_{\mathcal{L},k}^{\beta}(\mathbb{R}^N)$ associated with the Dunkl--Schr\"odinger operator $\mathcal{L}$ is defined as the space of all measurable functions $f$ on $\mathbb{R}^N$ for which the norm
\begin{equation}\label{eq:norm_diff_Dunkl}
    \|f\|_{\Lambda^{\beta}_{\mathcal{L},k}} := \| f m^\beta \|_{L^\infty(\mathbb{R}^N)} + \sup_{\mathbf{z} \neq 0} \frac{\| f(\cdot + \mathbf{z}) + f(\cdot - \mathbf{z}) - 2f(\cdot) \|_{L^\infty(\mathbb{R}^N)}}{|\mathbf{z}|^\beta}
\end{equation}
is finite.
\end{definition}

\begin{definition}\label{def:weighted_lipschitz}
Let $\beta > 0$ and let $m(\mathbf{x})$ denote the critical radius function associated with the potential $V$. The inhomogeneous weighted Dunkl--Lipschitz space $\widetilde{\Lambda}_{\mathcal{L},k}^{\beta/2}(\mathbb{R}^N)$ associated with the Dunkl--Schr\"odinger operator $\mathcal{L}$ is defined as the space of all measurable functions $f$ on $\mathbb{R}^N$ for which the norm
\begin{equation}\label{eq:norm_L_Dunkl}
    \|f\|_{\widetilde{\Lambda}^{\beta/2}_{\mathcal{L},k}} := \| f m^\beta \|_{L^\infty(\mathbb{R}^N)} + \sup_{t>0} \, t^{n - \beta/2} \left\| \partial_t^n K_t f \right\|_{L^\infty(\mathbb{R}^N)}
\end{equation}
is finite, where $n = \lfloor \beta/2 \rfloor + 1$ is the smallest integer strictly greater than $\beta/2$.
\end{definition}

Our first main result establishes the fundamental equivalence between the heat semigroup-based Lipschitz space and the pointwise weighted Zygmund space for $0 < \beta < 2$.

\begin{theorem}\label{thm:dunkl_schrodinger_equivalence}
Let $V \in {\rm{RH}}^{q}(dw)$ with $q > \max\left(1, \frac{\mathbf{N}}{2}\right)$, $V \geq 0$ and $0 < \beta < 2-\frac{\mathbf{N}}{q}$, $\beta \neq 1$. Then $\Lambda_{\mathcal{L},k}^{\beta}(\mathbb{R}^N) = \widetilde{\Lambda}_{\mathcal{L},k}^{\beta/2}(\mathbb{R}^N)$ with equivalent norms:
\begin{equation}\label{eq:main_norm_equivalence}
    C^{-1} \|f\|_{\widetilde{\Lambda}^{\beta/2}_{\mathcal{L},k}} \leq \|f\|_{\Lambda^{\beta}_{\mathcal{L},k}} \leq C \|f\|_{\widetilde{\Lambda}^{\beta/2}_{\mathcal{L},k}}.
\end{equation}
\end{theorem}

Theorem~\ref{thm:dunkl_schrodinger_equivalence} reveals a crucial structural feature of the Dunkl--Schr\"odinger setting. It demonstrates that the adapted Lipschitz spaces $\Lambda_{\mathcal{L},k}^{\beta}(\mathbb{R}^N)$ are entirely characterized in terms of the geometry of the critical radius function $m(\mathbf{x})$ and classical Euclidean second-order Zygmund differences. Notably, the pointwise norm in \eqref{eq:norm_diff_Dunkl} contains no explicit dependence on the multiplicity function $k$ or the underlying root system $R$.

A major novelty and technical highlight of our work is that we completely 
dispense with the assumption of $G$-invariance of the potential $V$ 
(i.e., $V(\sigma(\mathbf{x})) = V(\mathbf{x})$ for all $\sigma \in G$). Consequently, the 
critical radius function $m(\mathbf{x})$ is no longer $G$-invariant, which prevents 
the use of classical global arguments. To overcome this obstacle, we perform 
a highly delicate localization scheme in Section~\ref{sec:proof_1} based on orbit sectors $D_\sigma(\mathbf{x})$ and dynamic choices of group elements realizing the orbit 
metric.

As a crucial intermediate step toward understanding the Dunkl--Schr\"odinger setting, we also establish new equivalence results for inhomogeneous weighted Lipschitz spaces associated with the unperturbed Dunkl Laplacian $\Delta_k$ and its heat semigroup $H_t = e^{t\Delta_k}$. In this context, the role of the potential-driven weight $m(\mathbf{x})^\beta$ is replaced by the classical polynomial growth factor $(1+\|\mathbf{x}\|)^{-\beta}$.

We first state the definition of the weighted Dunkl--Lipschitz space defined dynamically via time-derivatives of the Dunkl heat semigroup $H_t$.

\begin{definition}\label{def:lipschitz_weighted}
Let $\beta>0$. We say that a measurable function $f$ defined on $\mathbb R^N$ belongs to the inhomogeneous weighted Lipschitz space $\widetilde{\Lambda}_k^{\beta/2}$, if 
\begin{equation}\label{eq:norm_L_Dunkl_def_2}
    \| f(\cdot)(1+\|\cdot\|)^{-\beta}\|_{L^\infty(\mathbb{R}^N)} + \sup_{t>0} t^{n-\beta/2} \left\| \partial_t^n H_tf\right\|_{L^\infty(\mathbb{R}^N)} <\infty, 
\end{equation}
where $n$ is the smallest positive integer strictly greater than $\beta/2$. We endow this class with the norm
\begin{equation}\label{eq:norm_L_Dunkl_norm_2}
    \| f\|_{{\widetilde{\Lambda}}^{\beta/2}_k} := \| f(\cdot)(1+\|\cdot\|)^{-\beta}\|_{L^\infty(\mathbb{R}^N)} + \sup_{t>0} t^{n-\beta/2} \left\| \partial_t^n H_tf\right\|_{L^\infty(\mathbb{R}^N)}.
\end{equation}
\end{definition}

Next, we define the corresponding pointwise weighted Lipschitz classes via classical H\"older and Zygmund conditions.

\begin{definition}\label{def:lipschitz_weighted_holder_1}
Let $0<\beta<1$. We say that a measurable function $f$ defined on $\mathbb R^N$ belongs to the inhomogeneous weighted Lipschitz space $\widetilde{\Lambda}_k^{\beta/2,H}$, if 
\begin{equation}\label{eq:cond_L_Dunkl}
    \| f(\cdot)(1+\|\cdot\|)^{-\beta}\|_{L^\infty(\mathbb{R}^N)} + \sup_{\mathbf{x}\ne \mathbf{x}'}\frac{|f(\mathbf{x})-f(\mathbf{x}')|}{\|\mathbf{x}-\mathbf{x}'\|^\beta} <\infty. 
\end{equation}
We endow this class with the norm
\begin{equation}\label{eq:norm_L_Dunkl_3}
    \| f\|_{{\widetilde{\Lambda}}^{\beta/2,H}_k} := \| f(\cdot)(1+\|\cdot\|)^{-\beta}\|_{L^\infty(\mathbb{R}^N)} + \sup_{\mathbf{x}\ne \mathbf{x}'}\frac{|f(\mathbf{x})-f(\mathbf{x}')|}{\|\mathbf{x}-\mathbf{x}'\|^\beta}.
\end{equation}
\end{definition}

\begin{definition}\label{def:lipschitz_weighted_holder_2}
Let $1 < \beta<2$. We say that a measurable function $f$ defined on $\mathbb R^N$ belongs to the inhomogeneous weighted Lipschitz space $\widetilde{\Lambda}_k^{\beta/2,H}$, if 
\begin{equation}\label{eq:cond_L_Dunkl_zygmund}
    \| f(\cdot)(1+\|\cdot\|)^{-\beta}\|_{L^\infty(\mathbb{R}^N)} + \sup_{\mathbf{x}\in\mathbb{R}^N} \sup_{\mathbf{y}\neq \mathbf{0}} \frac{|f(\mathbf{x}+\mathbf{y}) + f(\mathbf{x}-\mathbf{y}) - 2f(\mathbf{x})|}{\|\mathbf{y}\|^\beta} <\infty. 
\end{equation}
We endow this class with the norm
\begin{equation}\label{eq:norm_L_Dunkl_zygmund}
    \| f\|_{{\widetilde{\Lambda}}^{\beta/2,H}_k} := \| f(\cdot)(1+\|\cdot\|)^{-\beta}\|_{L^\infty(\mathbb{R}^N)} + \sup_{\mathbf{x}\in\mathbb{R}^N} \sup_{\mathbf{y}\neq \mathbf{0}} \frac{|f(\mathbf{x}+\mathbf{y}) + f(\mathbf{x}-\mathbf{y}) - 2f(\mathbf{x})|}{\|\mathbf{y}\|^\beta}.
\end{equation}
\end{definition}

Our next main theorem establishes the equivalence between the heat semigroup-based weighted Lipschitz space and the pointwise H\"older/Zygmund spaces in the unperturbed Dunkl setting.

\begin{theorem}\label{teo:equivalence}
Let $0 < \beta < 2$ with $\beta \neq 1$. Then the inhomogeneous weighted Lipschitz space $\widetilde{\Lambda}_k^{\beta/2}$ (defined via the Dunkl heat semigroup $H_t$) coincides with the space $\widetilde{\Lambda}_k^{\beta/2,H}$ (defined via the H\"older and Zygmund conditions). Furthermore, their respective norms are equivalent; that is, there exists a constant $C > 0$ such that for every measurable function $f$:
\begin{equation}\label{eq:equivalence_lipschitz}
    C^{-1} \|f\|_{\widetilde{\Lambda}_k^{\beta/2,H}} \leq \|f\|_{\widetilde{\Lambda}_k^{\beta/2}} \leq C \|f\|_{\widetilde{\Lambda}_k^{\beta/2,H}}.
\end{equation}
\end{theorem}

The proofs of Theorem~\ref{thm:dunkl_schrodinger_equivalence} and Theorem~\ref{teo:equivalence} partially parallel the classical Euclidean approach, yet require fundamentally new ideas. In the classical setting, extracting pointwise Zygmund and H\"older conditions from heat semigroup estimates relies heavily on the standard translational convolution structure of the Gauss--Weierstrass kernel. In the Dunkl framework, however, the heat kernel $h_t(\mathbf{x},\mathbf{y})$ (and its perturbed counterpart $k_t(\mathbf{x},\mathbf{y})$) lacks this translational invariance, preventing a direct application of classical difference-extraction techniques. To circumvent these technical hurdles, our arguments build upon and adapt the abstract semigroup framework of~\cite{DzHL}, providing a unified methodology to bridge operator-theoretic decay bounds and pointwise spatial smoothness.

Another fundamental tool in our analysis is a transference estimate that allows us to transition seamlessly between the unperturbed Dunkl heat semigroup $H_t$ and the Dunkl--Schr\"odinger semigroup $K_t$. This comparison result controls the difference between their time derivatives in terms of the potential-weighted norm $\|f m^\beta\|_{L^\infty(\mathbb{R}^N)}$.

\begin{theorem}\label{thm:main_derivative_difference}
Assume that $V \in {\rm{RH}}^{q}(dw)$ with $q>\max\left(1,\frac{\mathbf{N}}{2}\right)$ and $V \geq 0$. Let $0 < \beta < 2 - \frac{\mathbf{N}}{q}$. Then there exists a constant $C>0$ such that for all measurable functions $f$ satisfying $\|f m^\beta\|_{L^\infty(\mathbb{R}^N)} < \infty$, all $\mathbf{x} \in \mathbb{R}^N$, and $t>0$, the following inequality holds:
\begin{equation}\label{eq:main_derivative_difference}
    |\partial_t H_t f(\mathbf{x}) - \partial_t K_t f(\mathbf{x})| \leq C t^{-1+\beta/2} \|f m^\beta\|_{L^\infty(\mathbb{R}^N)}.
\end{equation}
\end{theorem}

Theorem~\ref{thm:main_derivative_difference} plays a pivotal role in bridging the regularity theory governed by the Dunkl Laplacian $\Delta_k$ and the perturbed operator $\mathcal{L} = -\Delta_k + V$. While the overall proof strategy draws inspiration from the work of De Le\'on-Contreras and Torrea~\cite{DeLeonTorrea} in the classical Schr\"odinger setting, their techniques are inherently insufficient in the presence of reflection symmetries. 

The primary difficulty lies in the fact that heat kernel estimates in Dunkl analysis (cf. Theorem~\ref{teo:heat_new}, see also~\cite{DzH_CalcVar} for upper and lower bounds for Dunkl heat kernel) cannot be expressed purely in terms of the Euclidean metric $\|\mathbf{x}-\mathbf{y}\|$. Instead, they heavily involve two distinct metrics—most prominently the orbit metric $d(\mathbf{x},\mathbf{y}) = \min_{\sigma \in G} \|\sigma(\mathbf{x})-\mathbf{y}\|$, which encodes the geometry of the reflection group $G$. To overcome this obstacle, we restructure the proof by carefully disentangling and isolating the role of the standard Euclidean distance from that of the orbit distance. This geometric decoupling is entirely absent in the classical Euclidean context, where no reflection group actions or multiple metrics exist.

Finally, a key technical ingredient supporting our heat semigroup estimates is a pointwise Gaussian bound for the time derivatives of the Dunkl--Schr\"odinger heat kernel $k_t(\mathbf{x},\mathbf{y})$.

\begin{theorem}\label{thm:time_derivatives}
For every $m \in \mathbb{N}_0$, there exist constants $C_m, c_m > 0$ such that for all $\mathbf{x}, \mathbf{y} \in \mathbb{R}^N$ and $t > 0$, the Dunkl--Schr\"odinger heat kernel satisfies
\begin{equation}\label{eq:time_derivatives_bound}
    |\partial_t^m k_t(\mathbf{x},\mathbf{y})| \leq \frac{C_m}{t^m w(B(\mathbf{x},\sqrt{t}))} \exp\left(- c_m \frac{d(\mathbf{x},\mathbf{y})^2}{t}\right).
\end{equation}
\end{theorem}

Although the proof of Theorem~\ref{thm:time_derivatives} adapts the well-established framework developed by Coulhon and Sikora \cite{CoulhonSikora2008}, a rigorous verification is necessary in our setting. This requirement stems from the fact that the bounds are governed by the orbit metric $d(\mathbf{x},\mathbf{y}) = \min_{\sigma \in G} \|\sigma(\mathbf{x})-\mathbf{y}\|$ associated with the reflection group $G$, rather than the standard Euclidean metric. Due to its standard yet technical nature, the detailed verification is deferred to Section~\ref{sec:appendix_}. We emphasize that Theorem~\ref{thm:time_derivatives} is of independent interest beyond its application in this paper, as these Gaussian estimates for time derivatives are expected to be a valuable tool for future research in Dunkl--Schr\"odinger analysis and related functional settings.

The paper is organized as follows. In Section~\ref{sec:preliminaries}, we provide preliminaries on Dunkl theory and introduce Dunkl--Schr\"odinger operators. Section~\ref{sec:reverse} is devoted to facts concerning the reverse H\"older class and the auxiliary function  $m(\mathbf{x})$. In Section~\ref{sec:lipschitz}, we state basic properties of the Lipschitz spaces defined in Definition~\ref{def:lipschitz_weighted} to build the necessary framework for the proof of Theorem~\ref{teo:equivalence}, which is presented in Section~\ref{sec:equiv}. Section~\ref{sec:aux} is concerned with integral representations for the semigroup kernels and potential operators. In Section~\ref{sec:proof_1}, we prove Theorem~\ref{thm:main_derivative_difference}, and finally, Section~\ref{sec:appendix_} is dedicated to the pointwise estimates from Theorem~\ref{thm:time_derivatives}.

Throughout this paper, $C>0$ denotes a positive constant that may change from line to line, but is independent of the main variables involved.

\section{Preliminaries}\label{sec:preliminaries}

\subsection{Basic definitions of the Dunkl theory}
In this section we present facts concerning the theory of the Dunkl operators.  For details we refer the reader to~\cite{Dunkl0},\cite{Dunkl},\cite{Dunkl3},\cite{Dunkl2},\cite{Roesler2},\cite{Roesle99},\cite{Roesler3},\cite{Roesler-Voit},\cite{RV1998}, and \cite{ThangaveluXu}.

We consider the Euclidean space $\mathbb R^N$ with the scalar product $\langle\mathbf x,\mathbf y\rangle=\sum_{j=1}^N x_jy_j
$, where $\mathbf x=(x_1,...,x_N)$, $\mathbf y=(y_1,...,y_N)$, and the norm $\| \mathbf x\|^2=\langle \mathbf x,\mathbf x\rangle$. For a nonzero vector $\alpha\in\mathbb R^N$,  the reflection $\sigma_\alpha$ with respect to the hyperplane $\alpha^\perp$ orthogonal to $\alpha$ is given by
\begin{align*}
\sigma_\alpha (\mathbf x)=\mathbf x-2\frac{\langle \mathbf x,\alpha\rangle}{\| \alpha\| ^2}\alpha.
\end{align*}
In this paper we fix a normalized root system in $\mathbb R^N$, that is, a finite set  $R\subset \mathbb R^N\setminus\{0\}$ such that $R \cap \alpha \mathbb{R} = \{\pm \alpha\}$,  $\sigma_\alpha (R)=R$, and $\|\alpha\|=\sqrt{2}$ for all $\alpha\in R$. The finite group $G$ generated by the reflections $\sigma_\alpha$, $\alpha \in R$ is called the {\it Weyl group} ({\it reflection group}) of the root system. A~{\textit{multiplicity function}} is a $G$-invariant function $k:R\to\mathbb C$ which will be fixed and $\geq 0$  throughout this paper. 
 Let
\begin{equation}\label{eq:measure_formula}
dw(\mathbf x)=\prod_{\alpha\in R}|\langle \mathbf x,\alpha\rangle|^{k(\alpha)}\, d\mathbf x
\end{equation} 
be  the associated measure in $\mathbb R^N$, where, here and subsequently, $d\mathbf x$ stands for the Lebesgue measure in $\mathbb R^N$.
We denote by 
\begin{equation}\label{eq:homo}
\mathbf{N}=N+\sum_{\alpha \in R} k(\alpha)
\end{equation}
the homogeneous dimension of the system. Clearly, 
\begin{align*} w(B(t\mathbf x, tr))=t^{\mathbf N}w(B(\mathbf x,r)) \ \ \text{\rm for all } \mathbf x\in\mathbb R^N, \ t,r>0,   
\end{align*}
 where $B(\mathbf x, r)=\{\mathbf y\in\mathbb R^N: \|\mathbf y-\mathbf x\|<r\}$. Observe that there is a constant $C>0$ such that 
\begin{equation}\label{eq:balls_asymp} 
C^{-1}w(B(\mathbf x,r))\leq  r^{N}\prod_{\alpha \in R} (|\langle \mathbf x,\alpha\rangle |+r)^{k(\alpha)}\leq C w(B(\mathbf x,r)),
\end{equation}
so $dw(\mathbf x)$ is doubling, that is, there is a constant $C>0$ such that
\begin{equation}\label{eq:doubling} w(B(\mathbf x,2r))\leq C w(B(\mathbf x,r)) \ \ \text{ for all } \mathbf x\in\mathbb R^N, \ r>0.
\end{equation}
Moreover, there exists a constant $C\ge1$ such that,
for every $\mathbf{x}\in\mathbb{R}^N$ and for every $r_2\ge r_1>0$,
\begin{equation}\label{eq:growth}
C^{-1}\Big(\frac{r_2}{r_1}\Big)^{N}\leq\frac{{w}(B(\mathbf{x},r_2))}{{w}(B(\mathbf{x},r_1))}\leq C \Big(\frac{r_2}{r_1}\Big)^{\mathbf{N}}.
\end{equation}

For a measurable subset $A$ of $\mathbb{R}^N$ we define 
\begin{equation}\label{eq:orbit_of_A}
    \mathcal{O}(A)=\{\sigma(\mathbf{x})\,:\, \mathbf{x} \in A, \, \sigma \in G\}.
\end{equation}
and let
$$d(\mathbf x,\mathbf y)=\min_{\sigma\in G}\| \sigma(\mathbf x)-\mathbf y\|$$
be the distance of the orbit of $\mathbf x$ to the orbit of $\mathbf y$. 
Clearly, by~\eqref{eq:balls_asymp}, for all $\mathbf{x} \in \mathbb{R}^N$ and $r>0$ we get
\begin{equation}\label{eq:ball_orbit_compare}
    w(\mathcal{O}(B(\mathbf{x},r))) \leq |G|w(B(\mathbf{x},r)).
\end{equation}

For $\xi \in \mathbb{R}^N$, the {\it Dunkl operators} $T_\xi$  are the following $k$-deformations of the directional derivatives $\partial_\xi$ by a  difference operator:
\begin{equation}\label{eq:T_def}
     T_\xi f(\mathbf x)= \partial_\xi f(\mathbf x) + \sum_{\alpha\in R} \frac{k(\alpha)}{2}\langle\alpha ,\xi\rangle\frac{f(\mathbf x)-f(\sigma_\alpha(\mathbf{x}))}{\langle \alpha,\mathbf x\rangle}.
\end{equation}
The Dunkl operators $T_{\xi}$, which were introduced in~\cite{Dunkl}, commute and are skew-symmetric with respect to the $G$-invariant measure $dw$. For fixed $\mathbf y\in\mathbb R^N$ the {\it Dunkl kernel} $E(\mathbf x,\mathbf y)$ is the unique analytic solution to the system
\begin{equation}\label{eq:Dunkl_kernel_definition}
    T_\xi f=\langle \xi,\mathbf y\rangle f, \ \ f(0)=1.
\end{equation}
The function $E(\mathbf x ,\mathbf y)$, which generalizes the exponential  function $e^{\langle \mathbf x,\mathbf y\rangle}$, has the unique extension to a holomorphic function on $\mathbb C^N\times \mathbb C^N$. Moreover, it satisfies $E(\mathbf{x},\mathbf{y})=E(\mathbf{y},\mathbf{x})$ for all $\mathbf{x},\mathbf{y} \in \mathbb{C}^N$.

Let $\{e_j\}_{1 \leq j \leq N}$ denote the canonical orthonormal basis in $\mathbb R^N$ and let $T_j=T_{e_j}$. 

The \textit{Dunkl transform}
  \begin{align*}\mathcal F f(\xi)=c_k^{-1}\int_{\mathbb R^N} E(-i\xi, \mathbf x)f(\mathbf x)\, dw(\mathbf x),
  \end{align*}
  where
  $$c_k=\int_{\mathbb{R}^N}e^{-\frac{\|\mathbf{x}\|^2}{2}}\,dw(\mathbf{x})>0,$$
   originally defined for $f\in L^1(dw)$, is an isometry on $L^2(dw)$, i.e., $\|f\|_{L^2(dw)}=\|\mathcal{F}f\|_{L^2(dw)} \text{ for all }f \in L^2(dw)$
and preserves the Schwartz class of functions $\mathcal S(\mathbb R^N)$ (see \cite{deJeu}).

\subsection{Dunkl Laplacian and Dunkl heat semigroup}\label{sec:laplacian} The {\it Dunkl Laplacian} associated with $R$ and $k$  is the differential-difference operator $\Delta_k=\sum_{j=1}^N T_{j}^2$, which  acts on $C^2(\mathbb{R}^N)$-functions by

\begin{align*}
    \Delta_k f(\mathbf x)=\Delta f(\mathbf x)+\sum_{\alpha\in R} k(\alpha) \delta_\alpha f(\mathbf x),
\end{align*}
\begin{align*}
    \delta_\alpha f(\mathbf x)=\frac{\partial_\alpha f(\mathbf x)}{\langle \alpha , \mathbf x\rangle} - \frac{\|\alpha\|^2}{2} \frac{f(\mathbf x)-f(\sigma_\alpha \mathbf x)}{\langle \alpha, \mathbf x\rangle^2}.
\end{align*}
Obviously, $\mathcal F(\Delta_k f)(\xi)=-\| \xi\|^2\mathcal Ff(\xi)$. The operator $\Delta_k$ is essentially self-adjoint on $L^2(dw)$ (see for instance \cite[Theorem\;3.1]{AH}) and generates the semigroup $H_t$  of linear self-adjoint contractions on $L^2(dw)$. The semigroup has the form
  \begin{equation}\label{eq:heat}
  H_t f(\mathbf x)=\mathcal F^{-1}(e^{-t\|\xi\|^2}\mathcal Ff(\xi))(\mathbf x)=\int_{\mathbb R^N} h_t(\mathbf x,\mathbf y)f(\mathbf y)\, dw(\mathbf y),
  \end{equation}
  where the two-variable Dunkl heat kernel $h_t(x,y)$ is defined explicitly in terms of the Dunkl kernel $E(x,y)$ via the formula
\begin{equation}\label{eq:heat_kernel_explicit}
h_t(x,y) = c_k^{-1} \mathcal{F}^{-1}\left(e^{-t\|\cdot\|^2} E(\cdot, y)\right)(x) = c_k^{-2} \int_{\mathbb{R}^N} e^{-t\|\xi\|^2} E(i\xi, x) E(-i\xi, y) \, dw(\xi).
\end{equation}
It is well-known that $h_t(x,y)$ is a $C^\infty$-function of all variables $x,y \in \mathbb{R}^N$, $t > 0$, and satisfies \begin{align*} 0<h_t(\mathbf x,\mathbf y)=h_t(\mathbf y,\mathbf x),
  \end{align*}
 \begin{equation} \label{eq:h_integral_1}
    \int_{\mathbb R^N} h_t(\mathbf x,\mathbf y)\, dw(\mathbf y)=1.
 \end{equation}

 \begin{theorem}[{\cite[Theorem 4.7]{RV1998}}]
Let $f \in \mathcal{S}(\mathbb{R}^N)$. Then $u(\mathbf{x},t)=H_{t}f(\mathbf{x})$ solves the initial-value problem
\begin{equation*}
\begin{cases}
    \partial_{t}u(\mathbf{x},t)=\Delta_{k,\mathbf{x}}u(\mathbf{x},t) \text{ on }\mathbb{R}^N \times (0,\infty),\\
    u(\mathbf{x},0)=f(\mathbf{x}) \text{ for all }\mathbf{x} \in \mathbb{R}^N.
\end{cases}
\end{equation*}
Moreover, the following properties hold:
\begin{enumerate}[(A)]
    \item{$H_tf \in \mathcal{S}(\mathbb{R}^N)$ for all $t>0$;}
    \item{$H_{t}(H_{s}f)=H_{t+s}f$ for all $t,s>0$;}
    \item{$\lim_{t \to 0}\|H_tf-f\|_{L^{\infty}}=0$.}
\end{enumerate}
\end{theorem}

Let us denote
\begin{equation}\label{eq:mathcal_G}
     \mathcal G_t(\mathbf x,\mathbf y)=\Big(\max (w(B(\mathbf x,\sqrt{t})),w(B(\mathbf y, \sqrt{t})))\Big)^{-1}\exp\Big(-\frac {d(\mathbf x,\mathbf y)^2}{t}\Big).
 \end{equation}

  We shall need the following estimates for $h_t(\mathbf x,\mathbf y)$ - their two step proof, which is based on R\"osler's formula for the Dunkl  translations of  radial functions (see~\cite{Roesler2003}), can be found in~\cite[Theorem 4.1]{ADzH} and~\cite[Theorem 3.1]{DzH1}.

  \begin{theorem}\label{teo:heat_new}   For all $m \in \mathbb{N}_0$ and  for all multi-indices $\boldsymbol{\beta},\boldsymbol{\beta}' \in \mathbb{N}_0^{N}$ there are constants $C, c>0$ such that for all $\mathbf{x},\mathbf{y} \in \mathbb{R}^N$ and $t>0$ we have
  \begin{equation}\label{eq:heat2} |\partial_t^m \partial_{\mathbf x}^{\boldsymbol{\beta}}\partial_{\mathbf y}^{\boldsymbol{\beta}'} h_t(\mathbf{x},\mathbf{y})|
  \leq C t^{-m-\frac{|\boldsymbol{\beta}|}{2}-\frac{|\boldsymbol{\beta}'|}{2}} \Big(1+\frac{\| \mathbf x-\mathbf y\|}{\sqrt{t}}\Big)^{-2} \mathcal G_{t\slash c} (\mathbf x,\mathbf y).
  \end{equation}
  Moreover, if $\mathbf{y}' \in \mathbb{R}^{N}$ and $\|\mathbf y-\mathbf y'\|\leq \sqrt{t}$, then
  \begin{equation}\label{eq:heat3}
  \begin{split}|\partial_t^m  h_t(\mathbf{x},\mathbf{y}) & -
 \partial_t^m  h_t(\mathbf{x},\mathbf{y'})|  \leq C t^{-m} \frac{\|\mathbf y-\mathbf y'\|}{\sqrt{t}}\Big(1+\frac{\| \mathbf x-\mathbf y\|}{\sqrt{t}}\Big)^{-2} \mathcal G_{t\slash c} (\mathbf x,\mathbf y).
  \end{split}\end{equation}
\end{theorem}

We will also need the following formula (see~\cite{DzH1}).

\begin{equation}\label{eq:T_j}
 T_{j,\mathbf x}h_t(\mathbf{x},\mathbf{y})=\frac{y_j-x_j}{2t}h_t(\mathbf{x},\mathbf{y}),
 \end{equation}



The following lemma is a standard estimate which holds on the spaces of homogeneous type.

\begin{lemma}\label{lem:homogeneous}
    Let $\delta>0$. There is a constant $C>0$ such that for all $\mathbf{x} \in \mathbb{R}^N$ and $r>0$ we have:
\begin{enumerate}[(A)]
    \item \label{numitem:outside_orbit} $\int_{\mathcal{O}(B(\mathbf{x},r))^{c}}\left(\frac{d(\mathbf{x},\mathbf{y})}{r}\right)^{-\delta}w(B(\mathbf{x},d(\mathbf{x},\mathbf{y})))^{-1}\,dw(\mathbf{y}) \leq C$;
    \item \label{numitem:inside_orbit} $\int_{\mathcal{O}(B(\mathbf{x},r))}\left(\frac{d(\mathbf{x},\mathbf{y})}{r}\right)^{\delta}w(B(\mathbf{x},d(\mathbf{x},\mathbf{y})))^{-1}\,dw(\mathbf{y}) \leq C$.
\end{enumerate}
\end{lemma}

We record the commutation properties of the Dunkl operators with time differentiation and the heat semigroup. Since each $T_j$ acts solely on the spatial variable $\mathbf{x} \in \mathbb{R}^N$ while $\partial_t$ acts on the time variable $t > 0$, standard smoothness of the heat kernel $h_t(\mathbf{x},\mathbf{y})$ allows differentiation under the integral sign, yielding $\partial_t T_j = T_j \partial_t$. Furthermore, because $T_j$ commutes with the Dunkl Laplacian $\Delta_k = \sum_{i=1}^N T_i^2$, it commutes with the generated heat semigroup $H_t = e^{t\Delta_k}$. Consequently, for all $t > 0$ and $j \in \{1, \dots, N\}$, we have
\begin{equation}\label{eq:dunkl_time_commute}
    \partial_t T_j H_t f = T_j \partial_t H_t f = T_j \Delta_k H_t f.
\end{equation}

\subsection{Dunkl-Schr\"odinger operator and semigroup}\label{sec:Schorodinger}
Let $V \geq 0$ be a measurable function such that $V \in L^2_{\rm{loc}}(dw)$. We consider the following operator on the Hilbert space $L^2(dw)$:
\begin{equation}\label{eq:Schrodinger_operator}
    \mathcal{L}=-\Delta_k+V
\end{equation}
with the domain
\begin{align*}
    \mathcal{D}(\mathcal{L})=\{f \in L^2(dw)\,:\,\|\xi\|^2\mathcal{F}f(\xi) \in L^2(dw(\xi)) \text{ and }V(\mathbf{x})f(\mathbf{x}) \in L^2(dw(\mathbf{x}))\}
\end{align*}
(see~\cite{AH}). We call this operator the \textit{Dunkl-Schr\"odinger operator}. Let us define the quadratic form
\begin{equation}
    \mathbf{Q}(f,g)=\sum_{j=1}^{N}\int_{\mathbb{R}^N}T_jf(\mathbf{x})\overline{T_jg(\mathbf{x})}\,dw(\mathbf{x})+\int_{\mathbb{R}^N}V(\mathbf{x})f(\mathbf{x})\overline{g(\mathbf{x})}\,dw(\mathbf{x})
\end{equation}
with the domain
\begin{align*}
    \mathcal{D}(\mathbf{Q})=\left\{f \in L^2(dw)\;:\; \left(\sum_{j=1}^{N}|T_jf|^2\right)^{1/2},V^{1/2}f \in L^2(dw)\right\} .
\end{align*}
The quadratic form is densely defined and closed (see~\cite[Lemma 4.1]{AH}), so there exists a unique positive self-adjoint operator $L$ such that
\begin{align*}
    \langle Lf,f\rangle=\mathbf{Q}(f,f) \text{ for all }f \in \mathcal{D}(L), 
\end{align*}
moreover,
\begin{align*}
    \mathcal{D}(L^{1/2})=\mathcal{D}(\mathbf{Q}) \text{ and }\mathbf{Q}(f,f)=\|L^{1/2}f\|_{L^2(dw)},
\end{align*}
where $L^{1/2}$ is a unique self-adjoint operator such that $(L^{1/2})^2=L$. It was proved in~\cite[Theorem 4.6]{AH}, that $\mathcal{L}$ is essentially self-adjoint on $C^{\infty}_{c}(\mathbb{R}^N)$ and $L$ is its closure. Consequently, $L$ generates the semigroup of self-adjoint contractions on $L^2(dw)$. The semigroup has the form (see~\cite[Theorem 4.8]{AH})
\begin{equation}\label{eq:K_semigroup}
    K_tf(\mathbf{x})=\int_{\mathbb{R}^N}k_t(\mathbf{x},\mathbf{y})f(\mathbf{y})\,dw(\mathbf{y}),
\end{equation}
where $k_t(\mathbf{x},\mathbf{y})$ is the integral kernel which satisfies
\begin{equation}\label{eq:kernels_compare}
   0 \leq k_t(\mathbf{x},\mathbf{y}) \leq h_t(\mathbf{x},\mathbf{y}). 
\end{equation}

\section{Reverse Hölder class, The auxiliary function \texorpdfstring{$m(\mathbf{x})$}{m(x)}}\label{sec:reverse}
In this part, we assume that $q > \max(1,\frac{\mathbf{N}}{2})$ and $V \geq 0$ belongs to the reverse H\"older class ${\rm{RH}}^{q}(dw)$, that is, there is a constant $C_{\text{RH}}>0$ such that
\begin{equation}\label{eq:reverse_Holder}
    \left(\frac{1}{w(B)}\int_{B}V(\mathbf{x})^q\,dw(\mathbf{x})\right)^{1/q} \leq C_{\text{RH}}\frac{1}{w(B)}\int_{B}V(\mathbf{x})\,dw(\mathbf{x}) \text{ for every ball }B.
\end{equation}
For any Lebesque measurable set $A$ we define
\begin{equation}\label{eq:mu}
    \mu(A)=\int_{A}V(\mathbf{x})\,dw(\mathbf{x}).
\end{equation}

 The proofs of the results in this section are standard and they are based on the~\cite[Chapter 7]{Grafakos}.

\begin{lemma}\label{lem:mu_doubling}
The measure $\mu$ defined in~\eqref{eq:mu} is doubling, i.e. there is a constant $C_{\mu}>0$ such that for all $\mathbf{x} \in \mathbb{R}$ and $r>0$ we have
\begin{equation*}
    \mu(B(\mathbf{x},2r)) \leq C_{\mu}\mu(B(\mathbf{x},r)).
\end{equation*}
\end{lemma}

Here and subsequently, we write
\begin{equation}
    \gamma=2-\frac{\mathbf{N}}{q}.
\end{equation}
The reverse H\"older inequality~\eqref{eq:reverse_Holder} has the following consequence (see~\cite[Lemma 1.2]{Shen}), which will be used in the next section many times. 

\begin{lemma}\label{lem:Rr}
Assume that $V \in {\rm{RH}}^{q}(dw)$, where $q>\max(1,\frac{\mathbf{N}}{2})$, and $V \geq 0$. There is a constant $C \geq 1$ such that for all $\mathbf{x} \in \mathbb{R}^N$ and $0<r_1<r_2<\infty$ we have
\begin{align*}
    \frac{r_1^2}{w(B(\mathbf{x},r_1))}\int_{B(\mathbf{x},r_1)}V(\mathbf{y})\,dw(\mathbf{y}) \leq C \left(\frac{r_1}{r_2}\right)^{\gamma} \frac{r_2^2}{w(B(\mathbf{x},r_2))}\int_{B(\mathbf{x},r_2)}V(\mathbf{y})\,dw(\mathbf{y}). 
\end{align*}
\end{lemma}

\subsection{Definition and growth properties of \texorpdfstring{$m(\mathbf{x})$}{m(x)}}
For $\mathbf{x} \in \mathbb{R}^N$ we define 
\begin{equation}\label{eq:m}
    \frac{1}{m(\mathbf{x})}=\sup\left\{r>0\,:\; \frac{r^2}{w(B(\mathbf{x},r))}\int_{B(\mathbf{x},r)}V(\mathbf{y})\,dw(\mathbf{y}) \leq 1 \right\}
\end{equation}
(see~\cite[Definition 1.3]{Shen}). Thanks to Lemma~\ref{lem:Rr}, for all $\mathbf{x} \in \mathbb{R}^N$ (and $V \not\equiv 0$) we have
\begin{equation}\label{eq:well-define}
    \lim_{r \to 0} \frac{r^2}{w(B(\mathbf{x},r))}\int_{B(\mathbf{x},r)}V(\mathbf{y})\,dw(\mathbf{y})=0,\, \ \ \lim_{r \to +\infty} \frac{r^2}{w(B(\mathbf{x},r))}\int_{B(\mathbf{x},r)}V(\mathbf{y})\,dw(\mathbf{y})=+\infty,
\end{equation}
so the function $m$ is well-defined and $0<m(\mathbf{x})<\infty$. The next lemma is an adaptation of~\cite[Lemma 1.4]{Shen}.

\begin{lemma}\label{lem:m_growth}
Assume that $V \in {\rm{RH}}^{q}(dw)$, where $q>\max(1,\frac{\mathbf{N}}{2})$, and $V \geq 0$. There are constants $C,\kappa>0$ such that for all $\mathbf{x},\mathbf{y} \in \mathbb{R}^N$ we have
\begin{equation}\label{eq:Shen_A}
    C^{-1}m(\mathbf{y}) \leq m(\mathbf{x}) \leq Cm(\mathbf{y}) \text { if }\|\mathbf{x}-\mathbf{y}\|<m(\mathbf{x})^{-1},
\end{equation}
\begin{equation}\label{eq:Shen_B}
    m(\mathbf{y}) \leq Cm(\mathbf{x})(1+m(\mathbf{x})\|\mathbf{x}-\mathbf{y}\|)^{\kappa},
\end{equation}
\begin{equation}\label{eq:Shen_C}
    m(\mathbf{y}) \geq C^{-1}m(\mathbf{x})(1+m(\mathbf{x})\|\mathbf{x}-\mathbf{y}\|)^{-\frac{\kappa}{1+\kappa}}.
\end{equation}
\end{lemma}

\section{Inhomogeneous weighted Lipschitz spaces and regularity}\label{sec:lipschitz}

In this section, we study the inhomogeneous weighted Lipschitz spaces $\widetilde{\Lambda}_k^{\beta/2}$ associated with the Dunkl heat semigroup $H_t$ (see Definition~\ref{def:lipschitz_weighted}) and study their fundamental properties. First, we establish the well-definedness, spatial smoothness, and initial-value convergence for $H_t f$ when $f \in \widetilde{\Lambda}_k^{\beta/2}$. Next, we analyze the spatial regularity of functions in these spaces, showing that classical partial derivatives $\partial_{x_j} f$ exist and map $\widetilde{\Lambda}_k^{\beta/2}$ boundedly into the lower-order space $\widetilde{\Lambda}_k^{(\beta-1)/2}$.

\subsection{Basic properties}

\begin{lemma}\label{lem:heat_semigroup_well_defined}
Let $\beta>0$, $j \in \mathbb{N}_0$, and $t>0$. If $f(\cdot)(1+\|\cdot\|)^{-\beta} \in L^{\infty}(\mathbb{R}^N)$, then $\partial_t^j H_t f$ is well-defined and satisfies
\begin{equation}\label{eq:q_bound}
    |\partial_t^j H_t f(\mathbf{x})| \leq C \|f(\cdot)(1+\|\cdot\|)^{-\beta} \|_{L^{\infty}} t^{-j} (1+\|\mathbf{x}\| + \sqrt{t})^{\beta} \quad \text{for all } \mathbf{x} \in \mathbb{R}^N.
\end{equation}
\end{lemma}

\begin{proof}
Set $M = \|f(\cdot)(1+\|\cdot\|)^{-\beta} \|_{L^{\infty}}$. Using $(1+\|\mathbf{y}\|)^\beta \le C_\beta \left( (1+\|\mathbf{x}\|)^\beta + d(\mathbf{x},\mathbf{y})^\beta \right)$, we have $|f(\mathbf{y})| \le C_\beta M \left( (1+\|\mathbf{x}\|)^\beta + d(\mathbf{x},\mathbf{y})^\beta \right)$. Combining this with the Gaussian estimates for $\partial_t^j h_t(\mathbf{x},\mathbf{y})$ (see Theorem~\ref{teo:heat_new}), we get
\begin{align*}
    |\partial_t^j H_t f(\mathbf{x})| 
    &\le \int_{\mathbb{R}^N} |\partial_t^j h_t(\mathbf{x},\mathbf{y})| |f(\mathbf{y})| \, dw(\mathbf{y}) \\
    &\le \frac{C M}{t^j w(B(\mathbf{x},\sqrt{t}))} \int_{\mathbb{R}^N} e^{-c d(\mathbf{x},\mathbf{y})^2/t} \left( (1+\|\mathbf{x}\|)^\beta + d(\mathbf{x},\mathbf{y})^\beta \right) dw(\mathbf{y}).
\end{align*}
Writing $d(\mathbf{x},\mathbf{y})^\beta = t^{\beta/2} (d(\mathbf{x},\mathbf{y})/\sqrt{t})^\beta$, absorbing the polynomial term into the exponential decay, and applying Lemma~\ref{lem:homogeneous}, the integral is bounded by $C w(B(\mathbf{x},\sqrt{t})) \left( (1+\|\mathbf{x}\|)^\beta + t^{\beta/2} \right)$. Consequently,
\[
    |\partial_t^j H_t f(\mathbf{x})| \le C M t^{-j} \left( (1+\|\mathbf{x}\|)^\beta + t^{\beta/2} \right) \le C M t^{-j} (1+\|\mathbf{x}\| + \sqrt{t})^\beta,
\]
which proves both the absolute convergence of the integral and the bound \eqref{eq:q_bound}.
\end{proof}

\begin{lemma}\label{lem:converge}
    Let $\beta>0$. Assume that $f(\cdot)(1+\|\cdot\|)^{-\beta} \in L^{\infty}(dw)$. Then
    \[ \lim_{t \to 0^{+}} H_tf(\mathbf{x}) = f(\mathbf{x}) \quad \text{for almost every } \mathbf{x} \in \mathbb{R}^N. \]
\end{lemma}

\begin{proof}
    Let $R > 0$ be arbitrary. It is enough to show that the convergence holds for almost every $\mathbf{x} \in B(0,R)$. 
    We split the function $f$: let $f_1 = f\chi_{B(0,2R)}$ and $f_2 = f\chi_{\mathbb{R}^N \setminus B(0,2R)}$. Thus, $f = f_1 + f_2$.
    
    Since $f(\cdot)(1+\|\cdot\|)^{-\beta} \in L^{\infty}(dw)$, the function $f$ is locally bounded. Hence $f_1 \in L^{\infty}(dw)$. By Corollary 5.4 and Remark 5.5 of~\cite{ADzH} we have
    \[ \lim_{t \to 0^{+}} H_tf_1(\mathbf{x}) = f_1(\mathbf{x}) \quad \text{for almost every } \mathbf{x} \in \mathbb{R}^N. \]
    Since $f_1(\mathbf{x}) = f(\mathbf{x})$ for all $\mathbf{x} \in B(0,R)$, it follows that $\lim_{t \to 0^{+}} H_tf_1(\mathbf{x}) = f(\mathbf{x})$ for a.e. $\mathbf{x} \in B(0,R)$.
    
    It remains to show that for every $\mathbf{x} \in B(0,R)$, we have $\lim_{t \to 0^{+}} H_tf_2(\mathbf{x}) = 0$. 
    Let $\mathbf{x} \in B(0,R)$ and $\mathbf{y} \in \mathbb{R}^N \setminus B(0,2R)$. Then $\|\mathbf{x}\| < R$ and $\|\mathbf{y}\| \geq 2R$. For any $\sigma \in G$,  we have $\|\sigma(\mathbf{x})\| = \|\mathbf{x}\| < R$. By the triangle inequality,
    \[ \|\sigma(\mathbf{x})-\mathbf{y}\| \geq \|\mathbf{y}\| - \|\sigma(\mathbf{x})\| > \|\mathbf{y}\| - R \geq \frac{1}{2}\|\mathbf{y}\|. \]
    Taking the minimum over $\sigma \in G$, we get $d(\mathbf{x},\mathbf{y}) \geq \frac{1}{2}\|\mathbf{y}\|$. Also, clearly $d(\mathbf{x},\mathbf{y}) \geq 2R - R = R$.
    
    Let $t \in (0,1]$. By Theorem~\ref{teo:heat_new}, we have
    \begin{align*}
        |H_tf_2(\mathbf{x})| &\leq \int_{\|\mathbf{y}\| \geq 2R} h_t(\mathbf{x},\mathbf{y}) |f(\mathbf{y})|\,dw(\mathbf{y}) \\
        &\leq C \int_{\|\mathbf{y}\| \geq 2R} \frac{1}{w(B(\mathbf{x},\sqrt{t/c}))} \exp\Big(-c\frac{d(\mathbf{x},\mathbf{y})^2}{t}\Big)|f(\mathbf{y})|\,dw(\mathbf{y}).
    \end{align*}
    By our assumption, $|f(\mathbf{y})| \leq C_f(1+\|\mathbf{y}\|)^{\beta}$ for some $C_f>0$. We split the exponential term into three parts:
    \[ \exp\Big(-c\frac{d(\mathbf{x},\mathbf{y})^2}{t}\Big) = \exp\Big(-\frac{c}{2}\frac{d(\mathbf{x},\mathbf{y})^2}{t}\Big) \exp\Big(-\frac{c}{4}\frac{d(\mathbf{x},\mathbf{y})^2}{t}\Big) \exp\Big(-\frac{c}{4}\frac{d(\mathbf{x},\mathbf{y})^2}{t}\Big). \]
    For the first part, since $d(\mathbf{x},\mathbf{y}) \geq R$, we have $\exp(-\frac{c}{2}d(\mathbf{x},\mathbf{y})^2/t) \leq \exp(-\frac{c}{2}R^2/t)$. 
    For the second part, using $d(\mathbf{x},\mathbf{y}) \geq \frac{1}{2}\|\mathbf{y}\|$ and $t \leq 1$, we can absorb the polynomial growth:
    \[ (1+\|\mathbf{y}\|)^{\beta} \leq (1+2d(\mathbf{x},\mathbf{y}))^{\beta} \leq \widetilde{C} \Big(1+\frac{d(\mathbf{x},\mathbf{y})}{\sqrt{t}}\Big)^{\beta}. \]
    Since polynomials are suppressed by exponentials, $\widetilde{C}(1+d(\mathbf{x},\mathbf{y})/\sqrt{t})^{\beta} \exp(-\frac{c}{4}d(\mathbf{x},\mathbf{y})^2/t) \leq C$.
    The third part remains inside the integral, which evaluates to a bound independent of $t$:
    \[ \int_{\mathbb{R}^N} \frac{1}{w(B(\mathbf{x},\sqrt{t/c}))} \exp\Big(-\frac{c}{4}\frac{d(\mathbf{x},\mathbf{y})^2}{t}\Big)\,dw(\mathbf{y}) \leq C. \]
    
    Combining these estimates gives
    \[ |H_tf_2(\mathbf{x})| \leq C \exp\Big(-\frac{cR^2}{2t}\Big). \]
    Since $R>0$ is fixed, the right-hand side converges to $0$ as $t \to 0^{+}$.
    Therefore, for almost every $\mathbf{x} \in B(0,R)$,
    \[ \lim_{t \to 0^{+}} H_tf(\mathbf{x}) = \lim_{t \to 0^{+}} H_tf_1(\mathbf{x}) + \lim_{t \to 0^{+}} H_tf_2(\mathbf{x}) = f(\mathbf{x}) + 0 = f(\mathbf{x}). \]
    Since $R > 0$ was chosen arbitrarily, the convergence holds almost everywhere on $\mathbb{R}^N$.
\end{proof}

\begin{lemma}\label{lem:h_t_smooth}
    Let $\beta>0$. Assume that $f(\cdot)(1+\|\cdot\|)^{-\beta} \in L^{\infty}(dw)$. Then $H_tf \in C^{\infty}((0,\infty)\times \mathbb{R}^N)$ as a function of variables $(t,\mathbf{x})$. Moreover, for any $m \in \mathbb{N}_0$ and any multi-index $\boldsymbol{\alpha} \in \mathbb{N}_0^N$, we have
    \[
        \partial_t^m \partial_{\mathbf{x}}^{\boldsymbol{\alpha}} H_tf(t,\mathbf{x}) = \int_{\mathbb{R}^N} \partial_t^m \partial_{\mathbf{x}}^{\boldsymbol{\alpha}} h_t(\mathbf{x},\mathbf{y}) f(\mathbf{y})\,dw(\mathbf{y}).
    \]
\end{lemma}

\begin{proof}
    To prove that $(t,\mathbf{x}) \mapsto H_tf(\mathbf{x})$ belongs to $C^{\infty}((0,\infty)\times \mathbb{R}^N)$, it suffices to show that for any $m \in \mathbb{N}_0$ and any multi-index $\boldsymbol{\alpha} \in \mathbb{N}_0^N$, the derivative operator $\partial_t^m \partial_{\mathbf{x}}^{\boldsymbol{\alpha}}$ can be passed under the integral sign. It is justified if for every compact set $K \subset (0,\infty)\times \mathbb{R}^N$, there exists an integrable majorant $F_K(\mathbf{y}) \in L^1(dw)$ such that
    \[
        |\partial_t^m \partial_{\mathbf{x}}^{\boldsymbol{\alpha}} h_t(\mathbf{x},\mathbf{y}) f(\mathbf{y})| \leq F_K(\mathbf{y}) \quad \text{for all } (t,\mathbf{x}) \in K, \mathbf{y} \in \mathbb{R}^N.
    \]

    Fix an arbitrary compact set $K = [t_1, t_2] \times \overline{B}(0,R) \subset (0,\infty)\times \mathbb{R}^N$, where $0 < t_1 < t_2 < \infty$ and $R > 0$. 
    By Theorem~\ref{teo:heat_new} (setting $\boldsymbol{\beta} = \boldsymbol{\alpha}$ and $\boldsymbol{\beta}' = 0$), there exist constants $C, c > 0$ such that for all $(t,\mathbf{x}) \in K$ and $\mathbf{y} \in \mathbb{R}^N$,
    \begin{equation}\label{eq:diff_bound}
        |\partial_t^m \partial_{\mathbf{x}}^{\boldsymbol{\alpha}} h_t(\mathbf{x},\mathbf{y})| \leq C t^{-m-|\boldsymbol{\alpha}|/2} \mathcal{G}_{t/c}(\mathbf{x},\mathbf{y}) \leq C t_1^{-m-|\boldsymbol{\alpha}|/2} \frac{1}{w(B(\mathbf{x},\sqrt{t/c}))} \exp\Big(-c\frac{d(\mathbf{x},\mathbf{y})^2}{t}\Big).
    \end{equation}

    Since $t \in [t_1, t_2]$, we have $\sqrt{t/c} \geq \sqrt{t_1/c}$. By the doubling property and growth estimate~\eqref{eq:growth}, for all $\mathbf{x} \in \overline{B}(0,R)$ and $t \in [t_1, t_2]$,
    \[
        w(B(\mathbf{x},\sqrt{t/c})) \geq C^{-1} \Big(\frac{t_1}{t_2}\Big)^{\mathbf{N}/2} w(B(\mathbf{x},\sqrt{t_2/c})) \geq C_K > 0,
    \]
    where $C_K = C^{-1} (t_1/t_2)^{\mathbf{N}/2} \inf_{\mathbf{x} \in \overline{B}(0,R)} w(B(\mathbf{x},\sqrt{t_2/c})) > 0$.
    Furthermore, for $t \in [t_1, t_2]$, we have $e^{-c d(\mathbf{x},\mathbf{y})^2/t} \leq e^{-c d(\mathbf{x},\mathbf{y})^2/t_2}$.

    For the orbit distance $d(\mathbf{x},\mathbf{y}) = \min_{\sigma \in G} \|\sigma(\mathbf{x}) - \mathbf{y}\|$, when $\mathbf{x} \in \overline{B}(0,R)$ and $\|\mathbf{y}\| \geq 2R$, we have
    \[
        d(\mathbf{x},\mathbf{y}) \geq \|\mathbf{y}\| - \|\mathbf{x}\| \geq \|\mathbf{y}\| - R \geq \frac{1}{2}\|\mathbf{y}\|.
    \]
    Therefore, for all $(t,\mathbf{x}) \in K$ and $\|\mathbf{y}\| \geq 2R$,
    \[
        \exp\Big(-c\frac{d(\mathbf{x},\mathbf{y})^2}{t_2}\Big) \leq \exp\Big(-\frac{c}{4t_2}\|\mathbf{y}\|^2\Big).
    \]

    Combining this with the assumption $|f(\mathbf{y})| \leq C_f (1+\|\mathbf{y}\|)^{\beta}$, we construct the majorant $F_K(\mathbf{y})$ on $\mathbb{R}^N$:
    \[
        F_K(\mathbf{y}) = 
        \begin{cases}
            \displaystyle C t_1^{-m-|\boldsymbol{\alpha}|/2} C_K^{-1} C_f (1+\|\mathbf{y}\|)^{\beta} \exp\Big(-\frac{c}{4t_2}\|\mathbf{y}\|^2\Big), & \text{for } \|\mathbf{y}\| \geq 2R,\\[10pt]
            \displaystyle C t_1^{-m-|\boldsymbol{\alpha}|/2} C_K^{-1} C_f (1+2R)^{\beta}, & \text{for } \|\mathbf{y}\| < 2R.
        \end{cases}
    \]

    The function $F_K(\mathbf{y})$ is bounded on the compact ball $\overline{B}(0,2R)$, while for $\|\mathbf{y}\| \geq 2R$, the Gaussian decay $\exp(-c\|\mathbf{y}\|^2/(4t_2))$ dominates the polynomial growth $(1+\|\mathbf{y}\|)^{\beta}$. Since $dw$ has polynomial growth by~\eqref{eq:balls_asymp}, it follows that $F_K \in L^1(dw)$.

    Thus, by the Lebesgue Dominated Convergence Theorem, $\partial_t^m \partial_{\mathbf{x}}^{\boldsymbol{\alpha}} H_tf(t,\mathbf{x})$ exists, is continuous on $(0,\infty) \times \mathbb{R}^N$, and is equal to the integral of the differentiated heat kernel against $f$. Since this holds for arbitrary $m \in \mathbb{N}_0$ and $\boldsymbol{\alpha} \in \mathbb{N}_0^N$, we conclude that $H_tf \in C^{\infty}((0,\infty)\times \mathbb{R}^N)$.
\end{proof}

\subsection{Spatial derivatives and mapping properties}

Having established the foundational properties of the heat semigroup, we now analyze the spatial regularity of functions in $\widetilde{\Lambda}_k^{\beta/2}$. We first derive crucial $L^\infty$-bounds for mixed spatial and temporal derivatives of $H_t f$. Utilizing these estimates, we prove the existence of continuous classical partial derivatives $\partial_{x_j} f$ for $1 < \beta < 2$ and demonstrate that differentiation maps $\widetilde{\Lambda}_k^{\beta/2}$ boundedly into the lower-order space $\widetilde{\Lambda}_k^{(\beta-1)/2}$.

\begin{lemma}\label{lem:semigroup_derivative_estimate}
    Let $\beta>0$ and let $n = \lfloor \beta/2 \rfloor + 1$ be the smallest integer strictly greater than $\beta/2$. Suppose $f \in \widetilde{\Lambda}_k^{\beta/2}$. Then for any $j \in \{1, \dots, N\}$, $j_1 \in \mathbb{N}_0$, and integer $j_2 \geq n$, there exists a constant $C > 0$ such that for all $s > 0$,
    \begin{equation}\label{eq:semigroup_derivative_bound}
        \left\| \partial_{x_j}^{j_1} \partial_s^{j_2} H_s f \right\|_{L^\infty} \leq C \|f\|_{\widetilde{\Lambda}_k^{\beta/2}} s^{\beta/2 - j_2 - j_1/2}.
    \end{equation}
\end{lemma}

\begin{proof}
    By the definition of the space $\widetilde{\Lambda}_k^{\beta/2}$, we have
    \begin{equation}\label{eq:m_derivative_def_bound}
        \left\| \partial_r^n H_r f \right\|_{L^\infty} \leq \|f\|_{\widetilde{\Lambda}_k^{\beta/2}} r^{\beta/2 - n} \quad \text{for all } r > 0.
    \end{equation}
    Using the semigroup property $H_s = H_{s/2} H_{s/2}$, we decompose the operator $H_s$ for any $s > 0$ as
    \[
        \partial_s^{j_2} H_s f = \partial_\tau^{j_2 - n} H_\tau \left( \partial_r^n H_r f \right)\Big|_{\tau = s/2, \, r = s/2}.
    \]
    Applying the spatial derivative $\partial_{x_j}^{j_1}$, we obtain
    \[
        \partial_{x_j}^{j_1} \partial_s^{j_2} H_s f(\mathbf{x}) = \int_{\mathbb{R}^N} \partial_{x_j}^{j_1} \partial_\tau^{j_2 - n} h_{s/2}(\mathbf{x}, \mathbf{y}) \left( \partial_r^n H_r f(\mathbf{y})\Big|_{r=s/2} \right) dw(\mathbf{y}).
    \]
    By Theorem~\ref{teo:heat_new}, there exists a constant $C_1 > 0$ such that for all $\tau > 0$,
    \[
        \int_{\mathbb{R}^N} \left| \partial_{x_j}^{j_1} \partial_\tau^{j_2 - n} h_\tau(\mathbf{x}, \mathbf{y}) \right| dw(\mathbf{y}) \leq C_1 \tau^{-(j_2 - n) - j_1/2}.
    \]
    Consequently, for any bounded function $g \in L^\infty(\mathbb{R}^N)$,
    \[
        \left\| \partial_{x_j}^{j_1} \partial_\tau^{j_2 - n} H_\tau g \right\|_{L^\infty} \leq C_1 \tau^{-(j_2 - n) - j_1/2} \|g\|_{L^\infty}.
    \]
    Setting $\tau = s/2$ and $g = \partial_r^n H_r f\big|_{r=s/2}$, and applying estimate~\eqref{eq:m_derivative_def_bound}, we deduce
    \begin{align*}
        \left\| \partial_{x_j}^{j_1} \partial_s^{j_2} H_s f \right\|_{L^\infty} 
        &\leq C_1 \left(\frac{s}{2}\right)^{-(j_2 - n) - j_1/2} \left\| \partial_r^n H_r f\Big|_{r=s/2} \right\|_{L^\infty} \\
        &\leq C_1 \left(\frac{s}{2}\right)^{-(j_2 - n) - j_1/2} \|f\|_{\widetilde{\Lambda}_k^{\beta/2}} \left(\frac{s}{2}\right)^{\beta/2 - n} = C \|f\|_{\widetilde{\Lambda}_k^{\beta/2}} s^{\beta/2 - j_2 - j_1/2}.
    \end{align*}
\end{proof}

\begin{lemma}\label{lem:semigroup_dunkl_derivative_estimate}
Let $0 < \beta < 2$ and integer $j_2 \geq 1$. If $f \in \widetilde{\Lambda}_k^{\beta/2}$, then for any $j \in \{1, \dots, N\}$ there exists $C > 0$ such that for all $s > 0$,
\begin{equation}\label{eq:semigroup_dunkl_derivative_bound}
    \left\| T_j \partial_s^{j_2} H_s f \right\|_{L^\infty} \leq C \|f\|_{\widetilde{\Lambda}_k^{\beta/2}} s^{\beta/2 - j_2 - 1/2}.
\end{equation}
\end{lemma}

\begin{proof}
We decompose $H_s = H_{s/2} H_{s/2}$ to write $\partial_s^{j_2} H_s f = \partial_\tau^{j_2 - 1} H_\tau \left( \partial_r H_r f \right)\big|_{\tau = r = s/2}$. Applying $T_j$ yields
\[
    T_j \partial_s^{j_2} H_s f(\mathbf{x}) = \int_{\mathbb{R}^N} T_{j,\mathbf{x}} \partial_\tau^{j_2 - 1} h_{s/2}(\mathbf{x}, \mathbf{y}) \left( \partial_r H_r f(\mathbf{y})\Big|_{r=s/2} \right) dw(\mathbf{y}).
\]
Commuting $T_{j,\mathbf{x}}$ with $\partial_\tau$ and applying \eqref{eq:T_j}, we have $T_{j,\mathbf{x}} h_\tau(\mathbf{x}, \mathbf{y}) = \frac{y_j - x_j}{2\tau} h_\tau(\mathbf{x}, \mathbf{y})$. Differentiating with respect to $\tau$ and using estimate \eqref{eq:heat2}, the spatial factor is absorbed:
\[
    |y_j - x_j| \left(1 + \frac{\|\mathbf{x}-\mathbf{y}\|}{\sqrt{\tau}}\right)^{-2} \leq \frac{1}{4} \sqrt{\tau}.
\]
This yields $\left| T_{j,\mathbf{x}} \partial_\tau^{j_2 - 1} h_\tau(\mathbf{x}, \mathbf{y}) \right| \leq C \tau^{-j_2 + 1/2} \mathcal{G}_{\tau/c}(\mathbf{x}, \mathbf{y})$. Integrating over $\mathbb{R}^N$ via Lemma~\ref{lem:homogeneous} gives
\[
    \left\| T_j \partial_\tau^{j_2 - 1} H_\tau g \right\|_{L^\infty} \leq C \tau^{-j_2 + 1/2} \|g\|_{L^\infty}.
\]
Setting $\tau = r = s/2$ and $g = \partial_r H_r f\big|_{r=s/2}$ (where $\|g\|_{L^\infty} \leq \|f\|_{\widetilde{\Lambda}_k^{\beta/2}} (s/2)^{\beta/2 - 1}$) completes the proof.
\end{proof}

\begin{lemma}\label{lem:derivative_existence}
    Let $1 < \beta < 2$ and assume that $f \in \widetilde{\Lambda}_k^{\beta/2}$. Then $f$ possesses a continuous representative (still denoted by $f$) such that for every $j \in \{1, \dots, N\}$, the classical partial derivative $\partial_{x_j}f(\mathbf{x})$ exists for all $\mathbf{x} \in \mathbb{R}^N$. Moreover, $\partial_{x_j}H_t f$ converges uniformly on $\mathbb{R}^N$ to $\partial_{x_j}f$ as $t \to 0^{+}$.
\end{lemma}

\begin{proof}
    Since $1 < \beta < 2$, we have $1/2 < \beta/2 < 1$, so $n = 1$ is the smallest integer strictly greater than $\beta/2$. By the definition of $\widetilde{\Lambda}_k^{\beta/2}$, $\|\partial_s H_s f\|_{L^\infty} \leq \|f\|_{\widetilde{\Lambda}_k^{\beta/2}} s^{\beta/2 - 1}$. By the Fundamental Theorem of Calculus, for any $0 < t < 1$ and $\mathbf{x} \in \mathbb{R}^N$,
    \begin{equation}\label{eq:ftc_Htf_en}
        H_t f(\mathbf{x}) = H_1 f(\mathbf{x}) - \int_t^1 \partial_s H_s f(\mathbf{x}) \, ds.
    \end{equation}
    Since $\beta/2 - 1 > -1$, the integral $\int_0^1 s^{\beta/2-1}\,ds$ converges. Thus, $H_t f$ converges uniformly on $\mathbb{R}^N$ as $t \to 0^{+}$ to a continuous function. Since $H_t f \to f$ almost everywhere, we identify $f$ with this continuous limit. Fix $j \in \{1, \dots, N\}$. Differentiating~\eqref{eq:ftc_Htf_en} with respect to $x_j$ for $t > 0$ yields
    \begin{equation}\label{eq:dx_Htf_en}
        \partial_{x_j} H_t f(\mathbf{x}) = \partial_{x_j} H_1 f(\mathbf{x}) - \int_t^1 \partial_{x_j} \partial_s H_s f(\mathbf{x}) \, ds.
    \end{equation}

    Applying Lemma~\ref{lem:semigroup_derivative_estimate} with $j_1 = 1$ and $j_2 = 1 = n$, we get
    \begin{equation}\label{eq:mixed_deriv_bound_applied}
        \|\partial_{x_j} \partial_s H_s f\|_{L^\infty} \leq C \|f\|_{\widetilde{\Lambda}_k^{\beta/2}} s^{\beta/2 - 1 - 1/2} = C \|f\|_{\widetilde{\Lambda}_k^{\beta/2}} s^{\beta/2 - 3/2}.
    \end{equation}
    Because $\beta > 1$, the exponent satisfies $\beta/2 - 3/2 > -1$. Hence, $s \mapsto s^{\beta/2 - 3/2}$ is integrable on $(0,1]$. The improper integral
    \[
        \int_0^1 \partial_{x_j} \partial_s H_s f(\mathbf{x}) \, ds
    \]
    converges absolutely and uniformly on $\mathbb{R}^N$.

    Standard differentiation theorems for uniform convergence imply that $f$ is continuously differentiable with respect to $x_j$, with
    \[
        \partial_{x_j} f(\mathbf{x}) = \partial_{x_j} H_1 f(\mathbf{x}) - \int_0^1 \partial_{x_j} \partial_s H_s f(\mathbf{x}) \, ds.
    \]
    Furthermore, for any $t \in (0, 1)$, we have
    \[
        \partial_{x_j} f(\mathbf{x}) - \partial_{x_j} H_t f(\mathbf{x}) = - \int_0^t \partial_{x_j} \partial_s H_s f(\mathbf{x}) \, ds.
    \]
    Taking the supremum over $\mathbf{x} \in \mathbb{R}^N$, we obtain
    \[
        \|\partial_{x_j} f - \partial_{x_j} H_t f\|_{L^\infty} \leq \int_0^t \|\partial_{x_j} \partial_s H_s f\|_{L^\infty} \, ds \leq C \|f\|_{\widetilde{\Lambda}_k^{\beta/2}} \frac{t^{\beta/2 - 1/2}}{\beta/2 - 1/2}.
    \]
    Since $\beta/2 - 1/2 > 0$, the right-hand side tends to $0$ as $t \to 0^{+}$. This confirms that $\partial_{x_j} H_t f$ converges uniformly on $\mathbb{R}^N$ to $\partial_{x_j}f$.
\end{proof}

\begin{corollary}\label{cor:dunkl_derivative_existence}
Let $1 < \beta < 2$ and assume that $f \in \widetilde{\Lambda}_k^{\beta/2}$. Then for every $j \in \{1, \dots, N\}$, the Dunkl derivative 
\[
    T_j f(\mathbf{x}) = \partial_j f(\mathbf{x}) + \sum_{\alpha \in R} \frac{k(\alpha)}{2} \langle \alpha, e_j \rangle D_\alpha f(\mathbf{x})
\]
is well-defined and continuous on $\mathbb{R}^N$, where $D_\alpha f(\mathbf{x}) = \int_0^1 \langle \nabla f(A_t \mathbf{x}), \alpha \rangle \, dt$ with $A_t = I - 2t \|\alpha\|^{-2} \alpha \alpha^T$. Moreover, $T_j H_t f$ converges uniformly on $\mathbb{R}^N$ to $T_j f$ as $t \to 0^{+}$.
\end{corollary}

\begin{proof}
By Lemma~\ref{lem:derivative_existence}, for each $j \in \{1, \dots, N\}$, the classical partial derivatives $\partial_j f$ exist, are continuous on $\mathbb{R}^N$, and $\partial_j H_t f \to \partial_j f$ uniformly on $\mathbb{R}^N$ as $t \to 0^{+}$. 

Since $A_t$ is a bounded operator with $\|A_t\| \leq 1$ for all $t \in [0,1]$ and $\nabla f$ is continuous, the integral representation of $D_\alpha f$ defines a continuous function on $\mathbb{R}^N$. Thus, $T_j f$ exists as a continuous function on $\mathbb{R}^N$ for every $j \in \{1, \dots, N\}$.

To verify the uniform convergence, notice that for any $t > 0$,
\[
    T_j H_t f(\mathbf{x}) - T_j f(\mathbf{x}) = \left(\partial_j H_t f(\mathbf{x}) - \partial_j f(\mathbf{x})\right) + \sum_{\alpha \in R} \frac{k(\alpha)}{2} \langle \alpha, e_j \rangle \int_0^1 \langle \nabla H_t f(A_\tau \mathbf{x}) - \nabla f(A_\tau \mathbf{x}), \alpha \rangle \, d\tau.
\]
Taking the supremum over $\mathbf{x} \in \mathbb{R}^N$, we obtain
\[
    \|T_j H_t f - T_j f\|_{L^\infty} \leq \|\partial_j H_t f - \partial_j f\|_{L^\infty} + C \sum_{i=1}^N \|\partial_i H_t f - \partial_i f\|_{L^\infty}.
\]
Applying Lemma~\ref{lem:derivative_existence} to each partial derivative $\partial_i f$, we have $\|\partial_i H_t f - \partial_i f\|_{L^\infty} \to 0$ as $t \to 0^{+}$, which completes the proof.
\end{proof}

\begin{lemma}\label{lem:derivative_lipschitz}
    Let $1 < \beta < 2$ and assume that $f \in \widetilde{\Lambda}_k^{\beta/2}$. Then for every $j \in \{1, \dots, N\}$, we have
    \[
        T_j f(\cdot)(1+\|\cdot\|)^{-(\beta-1)} \in L^{\infty}(\mathbb{R}^N),
    \]
    and $T_j f$ belongs to the Dunkl-Lipschitz space $\widetilde{\Lambda}_k^{(\beta-1)/2}$. Moreover, there exists a constant $C > 0$, independent of $f$, such that
    \[
        \|T_j f\|_{\widetilde{\Lambda}_k^{(\beta-1)/2}} \leq C \|f\|_{\widetilde{\Lambda}_k^{\beta/2}}.
    \]
\end{lemma}

\begin{proof}
    Fix $j \in \{1, \dots, N\}$ and let $g = T_j f$. We want to show that $g \in \widetilde{\Lambda}_k^{(\beta-1)/2}$. Set $\beta' = (\beta-1)/2$. Since $1 < \beta < 2$, we have $0 < \beta' < 1/2$, so the smallest positive integer strictly greater than $\beta'$ is $n' = 1$. 

    According to Definition~\ref{def:lipschitz_weighted}, we must verify two conditions for $g$: the weighted $L^\infty$-norm bound and the supremum bound involving the heat semigroup.

    \smallskip
    \noindent\textit{Weighted $L^\infty$-norm bound:} 
    By Corollary~\ref{cor:dunkl_derivative_existence}, $T_j H_s f$ converges uniformly on $\mathbb{R}^N$ to $T_j f$ as $s \to 0^{+}$. By the Fundamental Theorem of Calculus in the time variable, for any fixed $\mathbf{x} \in \mathbb{R}^N$ and any $t > 0$, we can write
    \[
        T_j f(\mathbf{x}) = T_j H_t f(\mathbf{x}) - \int_0^t T_j \partial_s H_s f(\mathbf{x}) \, ds.
    \]
    We balance both terms by choosing $t = (1+\|\mathbf{x}\|)^2$. 

    For the integral term, we apply Lemma~\ref{lem:semigroup_dunkl_derivative_estimate} with $j_2 = 1 = n$, which gives $\|T_j \partial_s H_s f\|_{L^\infty} \leq C_1 \|f\|_{\widetilde{\Lambda}_k^{\beta/2}} s^{\beta/2 - 3/2}$. Integrating this bound yields
    \[
        \left| \int_0^t T_j \partial_s H_s f(\mathbf{x}) \, ds \right| \leq C_1 \|f\|_{\widetilde{\Lambda}_k^{\beta/2}} \int_0^t s^{\beta/2 - 3/2} \, ds = \frac{C_1}{\beta/2 - 1/2} \|f\|_{\widetilde{\Lambda}_k^{\beta/2}} t^{\beta/2 - 1/2}.
    \]
    Substituting $t = (1+\|\mathbf{x}\|)^2$, we obtain
    \[
        \left| \int_0^t T_j \partial_s H_s f(\mathbf{x}) \, ds \right| \leq C_2 \|f\|_{\widetilde{\Lambda}_k^{\beta/2}} (1+\|\mathbf{x}\|)^{\beta-1}.
    \]

    For the first term, using formula \eqref{eq:T_j} and Theorem~\ref{teo:heat_new}, we have $|T_{j,\mathbf{x}} h_t(\mathbf{x},\mathbf{y})| \leq C_3 t^{-1/2} \mathcal{G}_{t/c}(\mathbf{x},\mathbf{y})$. Together with $|f(\mathbf{y})| \leq \|f\|_{\widetilde{\Lambda}_k^{\beta/2}} (1+\|\mathbf{y}\|)^\beta$ and Lemma~\ref{lem:heat_semigroup_well_defined}, we get
    \begin{align*}
        |T_j H_t f(\mathbf{x})| &\leq \int_{\mathbb{R}^N} |T_{j,\mathbf{x}} h_t(\mathbf{x},\mathbf{y})| |f(\mathbf{y})| \, dw(\mathbf{y}) \\
        &\leq C_3 \|f\|_{\widetilde{\Lambda}_k^{\beta/2}} \int_{\mathbb{R}^N} t^{-1/2} \mathcal{G}_{t/c}(\mathbf{x},\mathbf{y}) (1+\|\mathbf{y}\|)^\beta \, dw(\mathbf{y}) \leq C_4 \|f\|_{\widetilde{\Lambda}_k^{\beta/2}} t^{-1/2} (1+\|\mathbf{x}\| + \sqrt{t})^\beta.
    \end{align*}
    Substituting $t = (1+\|\mathbf{x}\|)^2$, so that $\sqrt{t} = 1+\|\mathbf{x}\|$ and $t^{-1/2} = (1+\|\mathbf{x}\|)^{-1}$, we find
    \[
        |T_j H_t f(\mathbf{x})| \leq C_4 \|f\|_{\widetilde{\Lambda}_k^{\beta/2}} (1+\|\mathbf{x}\|)^{-1} \big(2(1+\|\mathbf{x}\|)\big)^\beta = C_5 \|f\|_{\widetilde{\Lambda}_k^{\beta/2}} (1+\|\mathbf{x}\|)^{\beta-1}.
    \]
    Adding both estimates together, we conclude that for all $\mathbf{x} \in \mathbb{R}^N$,
    \[
        |T_j f(\mathbf{x})| \leq (C_2 + C_5) \|f\|_{\widetilde{\Lambda}_k^{\beta/2}} (1+\|\mathbf{x}\|)^{\beta-1}.
    \]
    Therefore, $\|g(\cdot)(1+\|\cdot\|)^{-(\beta-1)}\|_{L^\infty} \leq C' \|f\|_{\widetilde{\Lambda}_k^{\beta/2}} < \infty$.

    \smallskip
    \noindent\textit{Heat semigroup seminorm bound:} 
    Using the integral representation of $T_j f$ evaluated at $\tau = t$:
    \[
        T_j f = T_j H_t f - \int_0^t T_j \partial_s H_s f \, ds.
    \]
    Applying the linear operator $\partial_t H_t$ to both sides yields
    \[
        \partial_t H_t (T_j f) = \partial_t H_t (T_j H_t f) - \int_0^t \partial_t H_t (T_j \partial_s H_s f) \, ds.
    \]
    We estimate each term on the right-hand side using the fact that $T_j$ commutes with $\partial_t$ and $H_t$, alongside Lemma~\ref{lem:semigroup_dunkl_derivative_estimate}. 

    For the first term, using the semigroup property $H_t H_t = H_{2t}$, we rewrite the action as
    \[
        \partial_t H_t (T_j H_t f) = \partial_t T_j H_{2t} f = 2 \left. T_j \partial_s H_s f \right|_{s=2t}.
    \]
    Applying Lemma~\ref{lem:semigroup_dunkl_derivative_estimate} with $j_2 = 1 = n$, we obtain
    \[
        \|\partial_t H_t (T_j H_t f)\|_{L^\infty} = 2 \left\| \left. T_j \partial_s H_s f \right|_{s=2t} \right\|_{L^\infty} \leq C_1 \|f\|_{\widetilde{\Lambda}_k^{\beta/2}} (2t)^{\beta/2 - 3/2} = C_2 \|f\|_{\widetilde{\Lambda}_k^{\beta/2}} t^{\beta/2 - 3/2}.
    \]

    For the integral term, commuting the operators gives for each $s \in (0, t)$:
    \[
        \partial_t H_t (T_j \partial_s H_s f) = T_j \partial_r^2 H_r f \Big|_{r = t + s}.
    \]
    Applying Lemma~\ref{lem:semigroup_dunkl_derivative_estimate} with $j_2 = 2 \geq n = 1$, we get
    \[
        \left\| T_j \partial_r^2 H_r f \Big|_{r = t + s} \right\|_{L^\infty} \leq C_3 \|f\|_{\widetilde{\Lambda}_k^{\beta/2}} (t + s)^{\beta/2 - 2 - 1/2} = C_3 \|f\|_{\widetilde{\Lambda}_k^{\beta/2}} (t + s)^{\beta/2 - 5/2}.
    \]
    Integrating this bound with respect to $s$ over $[0, t]$, we find
    \begin{align*}
        \left\| \int_0^t \partial_t H_t (T_j \partial_s H_s f) \, ds \right\|_{L^\infty} 
        &\leq \int_0^t \left\| T_j \partial_r^2 H_r f \Big|_{r = t + s} \right\|_{L^\infty} \, ds \leq C_3 \|f\|_{\widetilde{\Lambda}_k^{\beta/2}} \int_0^t (t + s)^{\beta/2 - 5/2} \, ds \\
        &= \frac{C_3}{\beta/2 - 3/2} \|f\|_{\widetilde{\Lambda}_k^{\beta/2}} \left( (2t)^{\beta/2 - 3/2} - t^{\beta/2 - 3/2} \right) = C_4 \|f\|_{\widetilde{\Lambda}_k^{\beta/2}} t^{\beta/2 - 3/2}.
    \end{align*}

    Combining the estimates for both terms, we conclude that
    \[
        \|\partial_t H_t g\|_{L^\infty} \leq (C_2 + C_4) \|f\|_{\widetilde{\Lambda}_k^{\beta/2}} t^{\beta/2 - 3/2}.
    \]
    Recalling that $n' = 1$ and $\beta' = (\beta-1)/2$, we have $n' - \beta' = 3/2 - \beta/2$. Multiplying the above inequality by $t^{n' - \beta'}$ yields
    \[
        t^{1 - (\beta-1)/2} \|\partial_t H_t(T_j f)\|_{L^\infty} \leq (C_2 + C_4) \|f\|_{\widetilde{\Lambda}_k^{\beta/2}} t^{3/2 - \beta/2} t^{\beta/2 - 3/2} = C_5 \|f\|_{\widetilde{\Lambda}_k^{\beta/2}}.
    \]
    Taking the supremum over all $t > 0$, we arrive at
    \[
        \sup_{t>0} t^{1 - (\beta-1)/2} \left\| \partial_t H_t(T_j f) \right\|_{L^\infty} \leq C_5 \|f\|_{\widetilde{\Lambda}_k^{\beta/2}} < \infty,
    \]
    which completes the proof.
\end{proof}

\section{Equivalence of Lipschitz Spaces}\label{sec:equiv}

The main goal of this section is to establish the complete equivalence between the heat semigroup Lipschitz spaces and their classical counterparts defined via Hölder and Zygmund conditions (Theorem~\ref{teo:equivalence}). 

To establish Theorem~\ref{teo:equivalence}, we split the proof into two distinct cases depending on the smoothness parameter $\beta$.

\subsection{Case \texorpdfstring{$0<\beta<1$}{0<beta<1}}

We begin our analysis with the range $0 < \beta < 1$. In this setting, the arguments are relatively standard and rely on heat kernel estimates (Theorem~\ref{teo:heat_new}). For the reader's convenience and to keep the exposition self-contained, we present the detailed proofs for this case below. Recall that the space $\widetilde{\Lambda}_k^{\beta/2,H}$ is defined in Definition~\ref{def:lipschitz_weighted_holder_1}

\begin{lemma}\label{lem:inclusion_H_to_standard}
    Let $0 < \beta < 1$. If $f \in \widetilde{\Lambda}_k^{\beta/2,H}$, then $f \in \widetilde{\Lambda}_k^{\beta/2}$. Moreover, there exists a constant $C>0$, independent of $f$, such that
    \[
        \|f\|_{\widetilde{\Lambda}_k^{\beta/2}} \leq C \|f\|_{\widetilde{\Lambda}_k^{\beta/2,H}}.
    \]
\end{lemma}

\begin{proof}
    Let $f \in \widetilde{\Lambda}_k^{\beta/2,H}$. By definition, we already have the weighted $L^\infty$-norm bound:
    \[
        \| f(\cdot)(1+\|\cdot\|)^{-\beta}\|_{L^\infty} \leq \|f\|_{\widetilde{\Lambda}_k^{\beta/2,H}}.
    \]
    This is precisely the first component of the norm for $\widetilde{\Lambda}_k^{\beta/2}$. It remains to establish the heat semigroup seminorm bound. Since $0 < \beta < 1$, we have $0 < \beta/2 < 1/2$, so the smallest positive integer strictly greater than $\beta/2$ is $n=1$. Thus, we need to bound $t^{1-\beta/2}\|\partial_t H_t f\|_{L^\infty}$.

    Recall that $\int_{\mathbb{R}^N} h_t(\mathbf{x},\mathbf{y})\,dw(\mathbf{y}) = 1$ for all $t>0$. Differentiating this identity with respect to $t$, we obtain
    $
        \int_{\mathbb{R}^N} \partial_t h_t(\mathbf{x},\mathbf{y})\,dw(\mathbf{y}) = 0
    $.
    Using this, we can write the time derivative of the heat semigroup applied to $f$ as
    \begin{align*}
        \partial_t H_t f(\mathbf{x}) &= \int_{\mathbb{R}^N} \partial_t h_t(\mathbf{x},\mathbf{y}) f(\mathbf{y}) \, dw(\mathbf{y}) = \int_{\mathbb{R}^N} \partial_t h_t(\mathbf{x},\mathbf{y}) (f(\mathbf{y}) - f(\mathbf{x})) \, dw(\mathbf{y}).
    \end{align*}
    Since $f \in \widetilde{\Lambda}_k^{\beta/2,H}$, we have the Lipschitz condition $|f(\mathbf{y}) - f(\mathbf{x})| \leq \|f\|_{\widetilde{\Lambda}_k^{\beta/2,H}} \|\mathbf{x}-\mathbf{y}\|^\beta$. Applying this along with Theorem~\ref{teo:heat_new} for $m=1$ and $\beta=\beta'=\boldsymbol{0}$, we obtain
    \begin{align*}
        |\partial_t H_t f(\mathbf{x})| &\leq \int_{\mathbb{R}^N} |\partial_t h_t(\mathbf{x},\mathbf{y})| |f(\mathbf{y}) - f(\mathbf{x})| \, dw(\mathbf{y}) \leq \|f\|_{\widetilde{\Lambda}_k^{\beta/2,H}} \int_{\mathbb{R}^N} |\partial_t h_t(\mathbf{x},\mathbf{y})| \|\mathbf{x}-\mathbf{y}\|^\beta \, dw(\mathbf{y}) \\
        &\leq C_1 \|f\|_{\widetilde{\Lambda}_k^{\beta/2,H}} \int_{\mathbb{R}^N} t^{-1} \Big(1+\frac{\|\mathbf{x}-\mathbf{y}\|}{\sqrt{t}}\Big)^{-2} \mathcal{G}_{t/c}(\mathbf{x},\mathbf{y}) \|\mathbf{x}-\mathbf{y}\|^\beta \, dw(\mathbf{y}).
    \end{align*}
    We rewrite the distance term as $\|\mathbf{x}-\mathbf{y}\|^\beta = t^{\beta/2} \Big( \frac{\|\mathbf{x}-\mathbf{y}\|}{\sqrt{t}} \Big)^\beta$. This yields
    \[
        |\partial_t H_t f(\mathbf{x})| \leq C_1 \|f\|_{\widetilde{\Lambda}_k^{\beta/2,H}} t^{\beta/2 - 1} \int_{\mathbb{R}^N} \Big(1+\frac{\|\mathbf{x}-\mathbf{y}\|}{\sqrt{t}}\Big)^{-2} \Big( \frac{\|\mathbf{x}-\mathbf{y}\|}{\sqrt{t}} \Big)^\beta \mathcal{G}_{t/c}(\mathbf{x},\mathbf{y}) \, dw(\mathbf{y}).
    \]
    Since $0 < \beta < 1$, the function $r \mapsto (1+r)^{-2} r^\beta$ is bounded for all $r \geq 0$. Moreover, the integral of the Gaussian-type kernel $\mathcal{G}_{t/c}(\mathbf{x},\mathbf{y})$ over $\mathbb{R}^N$ with respect to the measure $w$ is uniformly bounded by a constant independent of $t$ and $\mathbf{x}$. Therefore,
    \begin{align*}
        |\partial_t H_t f(\mathbf{x})| &\leq C_2 \|f\|_{\widetilde{\Lambda}_k^{\beta/2,H}} t^{\beta/2 - 1} \int_{\mathbb{R}^N} \mathcal{G}_{t/c}(\mathbf{x},\mathbf{y}) \, dw(\mathbf{y}) \leq C_3 \|f\|_{\widetilde{\Lambda}_k^{\beta/2,H}} t^{\beta/2 - 1}.
    \end{align*}
    Multiplying both sides by $t^{1-\beta/2}$ and taking the supremum over all $\mathbf{x} \in \mathbb{R}^N$ and $t > 0$, we conclude that
    $
        \sup_{t>0} t^{1-\beta/2} \|\partial_t H_t f\|_{L^\infty} \leq C_3 \|f\|_{\widetilde{\Lambda}_k^{\beta/2,H}}
    $.
    Adding the weighted $L^\infty$-norm bound completes the proof.
\end{proof}

\begin{lemma}\label{lem:inclusion_standard_to_H}
    Let $0 < \beta < 1$. If $f \in \widetilde{\Lambda}_k^{\beta/2}$, then $f \in \widetilde{\Lambda}_k^{\beta/2,H}$. Moreover, there exists a constant $C>0$, independent of $f$, such that
    \[
        \|f\|_{\widetilde{\Lambda}_k^{\beta/2,H}} \leq C \|f\|_{\widetilde{\Lambda}_k^{\beta/2}}.
    \]
\end{lemma}

\begin{proof}
    Let $f \in \widetilde{\Lambda}_k^{\beta/2}$. Since the weighted $L^\infty$-norm bound is shared by both spaces, it trivially holds that
    \begin{equation}\label{eq:Lambda_f_weighted}
        \|f(\cdot)(1+\|\cdot\|)^{-\beta}\|_{L^\infty} \leq \|f\|_{\widetilde{\Lambda}_k^{\beta/2}}.
    \end{equation}
    
    We now turn to prove the Lipschitz regularity of $f$. Fix any $\mathbf{x}, \mathbf{x}' \in \mathbb{R}^N$ such that $\mathbf{x} \neq \mathbf{x}'$, and set $\delta = \|\mathbf{x} - \mathbf{x}'\|$. Using the Fundamental Theorem of Calculus:
    \begin{align}\label{eq:Lambda_split_heat}
        |f(\mathbf{x}) - f(\mathbf{x}')| &= \lim_{\varepsilon \to 0^+} |H_\varepsilon f(\mathbf{x}) - H_\varepsilon f(\mathbf{x}')| \nonumber \\
        &\leq \lim_{\varepsilon \to 0^+} \left| \int_\varepsilon^{\delta^2} \frac{d}{ds} \Big(H_s f(\mathbf{x}) - H_s f(\mathbf{x}')\Big) \, ds \right| + \left| H_{\delta^2} f(\mathbf{x}) - H_{\delta^2} f(\mathbf{x}') \right| \nonumber \\
        &\leq \int_0^{\delta^2} |\partial_s H_s f(\mathbf{x})| \, ds + \int_0^{\delta^2} |\partial_s H_s f(\mathbf{x}')| \, ds + \left| H_{\delta^2} f(\mathbf{x}) - H_{\delta^2} f(\mathbf{x}') \right|.
    \end{align}

    For the integral terms, the definition of the space $\widetilde{\Lambda}_k^{\beta/2}$ directly implies that $\|\partial_s H_s f\|_{L^\infty} \leq \|f\|_{\widetilde{\Lambda}_k^{\beta/2}} s^{\beta/2 - 1}$. Therefore,
    \begin{equation}\label{eq:Lambda_split_part1}
        \int_0^{\delta^2} |\partial_s H_s f(\mathbf{x})| \, ds \leq \|f\|_{\widetilde{\Lambda}_k^{\beta/2}} \int_0^{\delta^2} s^{\beta/2 - 1} \, ds = \frac{2}{\beta} \|f\|_{\widetilde{\Lambda}_k^{\beta/2}} (\delta^2)^{\beta/2} = \frac{2}{\beta} \|f\|_{\widetilde{\Lambda}_k^{\beta/2}} \delta^\beta.
    \end{equation}
    The exact same bound applies to the integral involving $\mathbf{x}'$.

    For the remaining difference, we use the Mean Value Theorem along with a gradient estimate. Since $0 < \beta < 1$, we have $m=1$. For any $r > 0$, Lemma~\ref{lem:semigroup_derivative_estimate} with $j_1=1$ and $j_2=1=n$ yields
    \[
        \|\nabla \partial_r H_r f\|_{L^\infty} \leq C_0 \|f\|_{\widetilde{\Lambda}_k^{\beta/2}} r^{\beta/2 - 3/2}.
    \]
    Since (by Theorem~\ref{teo:heat_new} and Lemma~\ref{lem:heat_semigroup_well_defined}) $\nabla H_T f \to 0$ pointwise as $T \to \infty$ and $\beta/2 - 3/2 < -1$, we can integrate this estimate from $s$ to $\infty$ to obtain
    \[
        \|\nabla H_s f\|_{L^\infty} \leq \int_s^\infty \|\nabla \partial_r H_r f\|_{L^\infty} \, dr \leq C_0 \|f\|_{\widetilde{\Lambda}_k^{\beta/2}} \int_s^\infty r^{\beta/2 - 3/2} \, dr = \frac{C_0}{1/2 - \beta/2} \|f\|_{\widetilde{\Lambda}_k^{\beta/2}} s^{\beta/2 - 1/2}.
    \]
    Evaluating this gradient bound at $s = \delta^2$, we get
    \[
        \|\nabla H_{\delta^2} f\|_{L^\infty} \leq C_1 \|f\|_{\widetilde{\Lambda}_k^{\beta/2}} (\delta^2)^{\beta/2 - 1/2} = C_1 \|f\|_{\widetilde{\Lambda}_k^{\beta/2}} \delta^{\beta - 1}.
    \]
    Consequently, by the Mean Value Theorem,
    \begin{equation}\label{eq:Lambda_split_part2}
        |H_{\delta^2} f(\mathbf{x}) - H_{\delta^2} f(\mathbf{x}')| \leq \|\mathbf{x} - \mathbf{x}'\| \|\nabla H_{\delta^2} f\|_{L^\infty} \leq \delta \cdot C_1 \|f\|_{\widetilde{\Lambda}_k^{\beta/2}} \delta^{\beta - 1} = C_1 \|f\|_{\widetilde{\Lambda}_k^{\beta/2}} \delta^\beta.
    \end{equation}

    Combining \eqref{eq:Lambda_split_part1} and \eqref{eq:Lambda_split_part2} into \eqref{eq:Lambda_split_heat}, we obtain that for all $\varepsilon \in (0, \delta^2)$:
    \[
        |H_\varepsilon f(\mathbf{x}) - H_\varepsilon f(\mathbf{x}')| \leq \left( C_1 + \frac{4}{\beta} \right) \|f\|_{\widetilde{\Lambda}_k^{\beta/2}} \delta^\beta = C \|f\|_{\widetilde{\Lambda}_k^{\beta/2}} \|\mathbf{x} - \mathbf{x}'\|^\beta.
    \]
    By our assumption \eqref{eq:Lambda_f_weighted}, the function $f$ satisfies the weighted bound $f(\cdot)(1+\|\cdot\|)^{-\beta} \in L^{\infty}(dw)$. Thus, Lemma~\ref{lem:converge} ensures that $H_\varepsilon f \to f$ pointwise almost everywhere as $\varepsilon \to 0^+$. Since the approximations $H_\varepsilon f$ satisfy the Hölder condition uniformly with a constant independent of $\varepsilon$, their almost everywhere limit $f$ admits a uniformly Hölder continuous representative. Identifying $f$ with this continuous representative, we can pass to the limit $\varepsilon \to 0^+$ for all $\mathbf{x}, \mathbf{x}' \in \mathbb{R}^N$ to obtain
    \[
        |f(\mathbf{x}) - f(\mathbf{x}')| \leq C \|f\|_{\widetilde{\Lambda}_k^{\beta/2}} \|\mathbf{x} - \mathbf{x}'\|^\beta.
    \]
    Dividing by $\|\mathbf{x} - \mathbf{x}'\|^\beta$ and taking the supremum over all $\mathbf{x} \neq \mathbf{x}'$ yields the required bound for the Lipschitz seminorm. Adding \eqref{eq:Lambda_f_weighted} completes the proof.
\end{proof}

\subsection{Case \texorpdfstring{$1 < \beta<2$}{1<beta<2}}

Unlike the previous regime, only one of the inclusions between the considered function spaces can be carried out using standard classical arguments. To establish the second inclusion, we rely on a specialized technique introduced in~\cite{DzHL}. Recall that the space $\widetilde{\Lambda}_k^{\beta/2,H}$ is defined in Definition~\ref{def:lipschitz_weighted_holder_2}.

\begin{remark}\label{rem:zygmund_equivalence}
    For $0 < \beta < 1$, the definition of the space $\widetilde{\Lambda}_k^{\beta/2,H}$ given via the first-order difference condition \eqref{eq:cond_L_Dunkl} is equivalent to the definition based on the Zygmund-type second symmetric difference condition \eqref{eq:cond_L_Dunkl_zygmund}. More precisely, for any measurable function $f$ satisfying $\| f(\cdot)(1+\|\cdot\|)^{-\beta}\|_{L^\infty} < \infty$, the first-difference semi-norm
    \[
        \sup_{\mathbf{x}\ne \mathbf{x}'}\frac{|f(\mathbf{x})-f(\mathbf{x}')|}{\|\mathbf{x}-\mathbf{x}'\|^\beta}
    \]
    and the second-difference semi-norm
    \[
        \sup_{\mathbf{x}\in\mathbb{R}^N} \sup_{\mathbf{y}\neq \mathbf{0}} \frac{|f(\mathbf{x}+\mathbf{y}) + f(\mathbf{x}-\mathbf{y}) - 2f(\mathbf{x})|}{\|\mathbf{y}\|^\beta}
    \]
    are equivalent, and thus yield equivalent norms for $\widetilde{\Lambda}_k^{\beta/2,H}$. This equivalence is established in \cite[Proposition 3.7]{DeLeonTorrea}.
\end{remark}

The following lemma establishes the first inclusion, and its proof closely imitates the classical approach.

\begin{lemma}\label{lem:inclusion_standard_to_H_zygmund}
    Let $1 < \beta < 2$. If $f \in \widetilde{\Lambda}_k^{\beta/2}$, then $f \in \widetilde{\Lambda}_k^{\beta/2,H}$. Moreover, there exists a constant $C>0$, independent of $f$, such that
    \[
        \|f\|_{\widetilde{\Lambda}_k^{\beta/2,H}} \leq C \|f\|_{\widetilde{\Lambda}_k^{\beta/2}}.
    \]
\end{lemma}

\begin{proof}
    Let $f \in \widetilde{\Lambda}_k^{\beta/2}$. The weighted $L^\infty$-norm bound trivially holds:
    \begin{equation}\label{eq:L_inf_part_zygmund}
        \|f(\cdot)(1+\|\cdot\|)^{-\beta}\|_{L^\infty} \leq \|f\|_{\widetilde{\Lambda}_k^{\beta/2}}.
    \end{equation}

    To prove the Zygmund regularity, fix $\mathbf{x} \in \mathbb{R}^N$ and $\mathbf{y} \in \mathbb{R}^N \setminus \{\mathbf{0}\}$, and set $\delta = \|\mathbf{y}\| > 0$. For any function $g$, we denote the second-order difference operator by
    \[
        \Delta_{\mathbf{y}}^2 g(\mathbf{x}) = g(\mathbf{x}+\mathbf{y}) + g(\mathbf{x}-\mathbf{y}) - 2g(\mathbf{x}).
    \]

    Since $f \in \widetilde{\Lambda}_k^{\beta/2}$, Lemma~\ref{lem:semigroup_derivative_estimate} (with $j_1=0$, $j_2=1=n$) yields $\|\partial_s H_s f\|_{L^\infty} \leq C \|f\|_{\widetilde{\Lambda}_k^{\beta/2}} s^{\beta/2 - 1}$. Because $\beta > 0$, the integral $\int_0^{\delta^2} \|\partial_s H_s f\|_{L^\infty} \, ds$ is finite, implying that $H_s f$ converges uniformly as $s \to 0^+$ to a continuous representative of $f$. Identifying $f$ with its continuous representative, we can write it pointwise as the integral of its time derivative evaluated up to the scale $s = \delta^2 = \|\mathbf{y}\|^2$:
    \[
        f(\mathbf{x}) = H_{\delta^2} f(\mathbf{x}) - \int_0^{\delta^2} \partial_s H_s f(\mathbf{x}) \, ds.
    \]

    Applying the linear difference operator $\Delta_{\mathbf{y}}^2$ to both sides yields
    \begin{equation}\label{eq:split_zygmund}
        |\Delta_{\mathbf{y}}^2 f(\mathbf{x})| \leq |\Delta_{\mathbf{y}}^2 (H_{\delta^2} f)(\mathbf{x})| + \int_0^{\delta^2} |\Delta_{\mathbf{y}}^2 (\partial_s H_s f)(\mathbf{x})| \, ds.
    \end{equation}

    For the integral term in \eqref{eq:split_zygmund}, we use the bound $|\Delta_{\mathbf{y}}^2 g(\mathbf{x})| \leq 4 \|g\|_{L^\infty}$. Since $\|\partial_s H_s f\|_{L^\infty} \leq \|f\|_{\widetilde{\Lambda}_k^{\beta/2}} s^{\beta/2 - 1}$, we have
    \begin{equation}\label{eq:part1_zygmund}
        \int_0^{\delta^2} |\Delta_{\mathbf{y}}^2 (\partial_s H_s f)(\mathbf{x})| \, ds \leq 4 \|f\|_{\widetilde{\Lambda}_k^{\beta/2}} \int_0^{\delta^2} s^{\beta/2 - 1} \, ds = \frac{8}{\beta} \|f\|_{\widetilde{\Lambda}_k^{\beta/2}} (\delta^2)^{\beta/2} = \frac{8}{\beta} \|f\|_{\widetilde{\Lambda}_k^{\beta/2}} \|\mathbf{y}\|^\beta.
    \end{equation}

    For the term $|\Delta_{\mathbf{y}}^2 (H_{\delta^2} f)(\mathbf{x})|$, we use the identity for the second difference of a $C^2$ function $F$:
    \[
        \Delta_{\mathbf{y}}^2 F(\mathbf{x}) = \int_0^1 \left( \int_{-t}^t \sum_{i,j=1}^N y_i y_j \frac{\partial^2 F}{\partial x_i \partial x_j}(\mathbf{x} + \tau \mathbf{y}) \, d\tau \right) dt.
    \]
    Taking the absolute value, we obtain
    \[
        |\Delta_{\mathbf{y}}^2 F(\mathbf{x})| \leq \|\mathbf{y}\|^2 \|\nabla^2 F\|_{L^\infty} \int_0^1 \int_{-t}^t d\tau \, dt = \|\mathbf{y}\|^2 \|\nabla^2 F\|_{L^\infty}.
    \]

    To estimate the classical second spatial derivatives of the semigroup, we use the decay at infinity and write $\nabla^2 H_s f = - \int_s^\infty \nabla^2 \partial_t H_t f \, dt$. Applying Lemma~\ref{lem:semigroup_derivative_estimate} with $j_1=2$ and $j_2=1=n$, we have $\|\nabla^2 \partial_t H_t f\|_{L^\infty} \leq C_0 \|f\|_{\widetilde{\Lambda}_k^{\beta/2}} t^{\beta/2 - 2}$. Integrating this bound for $\beta < 2$ yields:
    \[
        \|\nabla^2 H_s f\|_{L^\infty} \leq \int_s^\infty \|\nabla^2 \partial_t H_t f\|_{L^\infty} \, dt \leq C_0 \|f\|_{\widetilde{\Lambda}_k^{\beta/2}} \int_s^\infty t^{\beta/2 - 2} \, dt = C_1 \|f\|_{\widetilde{\Lambda}_k^{\beta/2}} s^{\beta/2 - 1},
    \]
    where $C_1 = \frac{C_0}{1 - \beta/2}$. Evaluating this estimate at $s = \delta^2 = \|\mathbf{y}\|^2$, we get
    \[
        \|\nabla^2 H_{\delta^2} f\|_{L^\infty} \leq C_1 \|f\|_{\widetilde{\Lambda}_k^{\beta/2}} \|\mathbf{y}\|^{\beta - 2}.
    \]

    Thus, applying this bound to $F = H_{\delta^2} f$, we obtain
    \begin{equation}\label{eq:part2_zygmund}
        |\Delta_{\mathbf{y}}^2 (H_{\delta^2} f)(\mathbf{x})| \leq \|\mathbf{y}\|^2 \left( C_1 \|f\|_{\widetilde{\Lambda}_k^{\beta/2}} \|\mathbf{y}\|^{\beta - 2} \right) = C_1 \|f\|_{\widetilde{\Lambda}_k^{\beta/2}} \|\mathbf{y}\|^\beta.
    \end{equation}

    Combining estimates \eqref{eq:part1_zygmund} and \eqref{eq:part2_zygmund} into \eqref{eq:split_zygmund}, we conclude that
    \[
        |\Delta_{\mathbf{y}}^2 f(\mathbf{x})| \leq \left( C_1 + \frac{8}{\beta} \right) \|f\|_{\widetilde{\Lambda}_k^{\beta/2}} \|\mathbf{y}\|^\beta.
    \]
    Dividing by $\|\mathbf{y}\|^\beta$ and taking the supremum over $\mathbf{x} \in \mathbb{R}^N$ and $\mathbf{y} \neq \mathbf{0}$, together with \eqref{eq:L_inf_part_zygmund}, completes the proof.
\end{proof}

We now turn to the proof of the reverse inclusion, which will be established through a sequence of lemmas. We start by proving this inclusion in the standard Euclidean setting ($k \equiv 0$). Aside from serving as a convenient point of reference, the result of this classical lemma will be explicitly used in our approach to the general Dunkl setting later on. Although this result is standard, we provide a brief proof below for the sake of completeness and to keep the exposition self-contained.

\begin{lemma}\label{lem:inclusion_H_to_standard_zygmund_classical}
    Let $1 < \beta < 2$ and let $h_t(\mathbf{x},\mathbf{y}) = h_t(\mathbf{x}-\mathbf{y})$ be the classical Gaussian heat kernel on $\mathbb{R}^N$. If $f \in \widetilde{\Lambda}_k^{\beta/2,H}$, then $f \in \widetilde{\Lambda}_k^{\beta/2}$ and there exists $C>0$ such that
    \[
        \|f\|_{\widetilde{\Lambda}_k^{\beta/2}} \leq C \|f\|_{\widetilde{\Lambda}_k^{\beta/2,H}}.
    \]
\end{lemma}

\begin{proof}
    The weighted $L^\infty$-norm bound holds by definition. For the semigroup seminorm ($n=1$), using the radial symmetry of $h_t$ and the identity $\int_{\mathbb{R}^N} \partial_t h_t(\mathbf{y}) \, d\mathbf{y} = 0$, we represent the time derivative via the Zygmund second difference:
    \[
        \partial_t H_t f(\mathbf{x}) = \frac{1}{2} \int_{\mathbb{R}^N} \partial_t h_t(\mathbf{y}) \Big( f(\mathbf{x}+\mathbf{y}) + f(\mathbf{x}-\mathbf{y}) - 2f(\mathbf{x}) \Big) \, d\mathbf{y}.
    \]
    Applying the Zygmund condition together with the kernel estimate $|\partial_t h_t(\mathbf{y})| \leq C_1 t^{-1 - N/2} \exp(-c \|\mathbf{y}\|^2 / t)$ and setting $\mathbf{z} = \mathbf{y} / \sqrt{t}$, we get:
    \[
        |\partial_t H_t f(\mathbf{x})| \leq \frac{C_1}{2} \|f\|_{\widetilde{\Lambda}_k^{\beta/2,H}} t^{\beta/2 - 1} \int_{\mathbb{R}^N} e^{-c\|\mathbf{z}\|^2} \|\mathbf{z}\|^\beta \, d\mathbf{z} = C_2 \|f\|_{\widetilde{\Lambda}_k^{\beta/2,H}} t^{\beta/2 - 1}.
    \]
    Multiplying by $t^{1-\beta/2}$ and taking the supremum over $\mathbf{x} \in \mathbb{R}^N$ and $t > 0$ yields the required seminorm bound.
\end{proof}

\begin{lemma}\label{lem:derivative_H_space}
    Let $1 < \beta < 2$. If $f \in \widetilde{\Lambda}_k^{\beta/2,H}$, then $f$ is continuously differentiable on $\mathbb{R}^N$ and for each $j \in \{1, \dots, N\}$, the partial derivative $\partial_j f$ belongs to $\widetilde{\Lambda}_k^{(\beta-1)/2,H}$. Moreover, there exists a constant $C>0$, independent of $f$, such that
    \[
        \|\partial_j f\|_{\widetilde{\Lambda}_k^{(\beta-1)/2,H}} \leq C \|f\|_{\widetilde{\Lambda}_k^{\beta/2,H}}.
    \]
\end{lemma}

\begin{proof}
    We begin by observing that the definition of the inhomogeneous weighted Lipschitz space $\widetilde{\Lambda}_k^{\delta,H}$ (both for $0 < \delta < 1/2$ and $1/2 < \delta < 1$) depends exclusively on the Euclidean norm and the classical difference operators. It does not involve the root system $R$ nor the multiplicity function $k$. Consequently, these spaces coincide exactly with their classical counterparts, which we denote by setting $k \equiv 0$. In particular, $\widetilde{\Lambda}_k^{\beta/2,H} = \widetilde{\Lambda}_0^{\beta/2,H}$ and they share the same norm.

    Let $f \in \widetilde{\Lambda}_k^{\beta/2,H} = \widetilde{\Lambda}_0^{\beta/2,H}$. By Lemma~\ref{lem:inclusion_H_to_standard_zygmund_classical}, which was established for the classical heat kernel, we know that $f$ belongs to the classical semigroup space $\widetilde{\Lambda}_0^{\beta/2}$, and
    \begin{equation}\label{eq:k0_bound1}
        \|f\|_{\widetilde{\Lambda}_0^{\beta/2}} \leq C_1 \|f\|_{\widetilde{\Lambda}_0^{\beta/2,H}} = C_1 \|f\|_{\widetilde{\Lambda}_k^{\beta/2,H}}.
    \end{equation}

    Next, we use the derivative properties of the space (case $k \equiv 0$ in Lemma~\ref{lem:derivative_lipschitz}). Since $1 < \beta < 2$ and $f \in \widetilde{\Lambda}_0^{\beta/2}$, $f$ is continuously differentiable, and its classical partial derivatives $\partial_j f$ belong to $\widetilde{\Lambda}_0^{(\beta-1)/2}$. Furthermore, the norm of the derivative is bounded by the norm of the function:
    \begin{equation}\label{eq:k0_bound2}
        \|\partial_j f\|_{\widetilde{\Lambda}_0^{(\beta-1)/2}} \leq C_2 \|f\|_{\widetilde{\Lambda}_0^{\beta/2}}.
    \end{equation}

    Now, since $0 < \beta - 1 < 1$, we are in the Hölder regime for the exponent $(\beta-1)/2$. We can thus apply the classical version of Lemma~\ref{lem:inclusion_standard_to_H} (i.e., for $k \equiv 0$) to deduce that $\partial_j f \in \widetilde{\Lambda}_0^{(\beta-1)/2, H}$, with the estimate
    \begin{equation}\label{eq:k0_bound3}
        \|\partial_j f\|_{\widetilde{\Lambda}_0^{(\beta-1)/2,H}} \leq C_3 \|\partial_j f\|_{\widetilde{\Lambda}_0^{(\beta-1)/2}}.
    \end{equation}

    Finally, we use the independence of the spaces from the multiplicity function $k$ once again to identify $\widetilde{\Lambda}_0^{(\beta-1)/2,H} = \widetilde{\Lambda}_k^{(\beta-1)/2,H}$. Combining the inequalities \eqref{eq:k0_bound1}, \eqref{eq:k0_bound2}, and \eqref{eq:k0_bound3}, we obtain
    \[
        \|\partial_j f\|_{\widetilde{\Lambda}_k^{(\beta-1)/2,H}} = \|\partial_j f\|_{\widetilde{\Lambda}_0^{(\beta-1)/2,H}} \leq C_3 \|\partial_j f\|_{\widetilde{\Lambda}_0^{(\beta-1)/2}} \leq C_2 C_3 \|f\|_{\widetilde{\Lambda}_0^{\beta/2}} \leq C_1 C_2 C_3 \|f\|_{\widetilde{\Lambda}_k^{\beta/2,H}}.
    \]
    Setting $C = C_1 C_2 C_3$ completes the proof.
\end{proof}

\begin{corollary}\label{cor:derivative_semigroup_space}
    Let $1 < \beta < 2$. If $f \in \widetilde{\Lambda}_k^{\beta/2,H}$, then $f$ is continuously differentiable on $\mathbb{R}^N$ and for each $j \in \{1, \dots, N\}$, the partial derivative $\partial_j f$ belongs to $\widetilde{\Lambda}_k^{(\beta-1)/2}$. Moreover, there exists a constant $C>0$, independent of $f$, such that
    \[
        \|\partial_j f\|_{\widetilde{\Lambda}_k^{(\beta-1)/2}} \leq C \|f\|_{\widetilde{\Lambda}_k^{\beta/2,H}}.
    \]
\end{corollary}

\begin{proof}
    Let $f \in \widetilde{\Lambda}_k^{\beta/2,H}$. By Lemma~\ref{lem:derivative_H_space}, we know that $f$ is continuously differentiable and its partial derivative satisfies $\partial_j f \in \widetilde{\Lambda}_k^{(\beta-1)/2,H}$, with the norm estimate:
    \begin{equation}\label{eq:cor_bound1}
        \|\partial_j f\|_{\widetilde{\Lambda}_k^{(\beta-1)/2,H}} \leq C_1 \|f\|_{\widetilde{\Lambda}_k^{\beta/2,H}}.
    \end{equation}
    
    Since $1 < \beta < 2$, the exponent $\beta' := \beta - 1$ satisfies $0 < \beta' < 1$. This allows us to apply Lemma~\ref{lem:inclusion_H_to_standard} (which holds for parameters strictly between $0$ and $1$) to the function $\partial_j f$. Consequently, $\partial_j f$ belongs to the Dunkl semigroup space $\widetilde{\Lambda}_k^{(\beta-1)/2}$, and we have the bound:
    \begin{equation}\label{eq:cor_bound2}
        \|\partial_j f\|_{\widetilde{\Lambda}_k^{(\beta-1)/2}} \leq C_2 \|\partial_j f\|_{\widetilde{\Lambda}_k^{(\beta-1)/2,H}}.
    \end{equation}
    
    Combining \eqref{eq:cor_bound1} and \eqref{eq:cor_bound2} directly yields
    $
        \|\partial_j f\|_{\widetilde{\Lambda}_k^{(\beta-1)/2}} \leq C_1 C_2 \|f\|_{\widetilde{\Lambda}_k^{\beta/2,H}}
    $.
    Setting $C = C_1 C_2$ finishes the proof.
\end{proof}

\begin{lemma}\label{lem:Tj_semigroup_space}
    Let $1 < \beta < 2$ and suppose $f \in C^1(\mathbb{R}^N)$. If $\partial_l f \in \widetilde{\Lambda}_k^{(\beta-1)/2}$ for all $l \in \{1, \dots, N\}$, then the Dunkl derivative $T_j f$ belongs to $\widetilde{\Lambda}_k^{(\beta-1)/2}$ for each $j \in \{1, \dots, N\}$. Moreover, there exists a constant $C>0$, independent of $f$, such that
    \[
        \|T_j f\|_{\widetilde{\Lambda}_k^{(\beta-1)/2}} \leq C \sum_{l=1}^N \|\partial_l f\|_{\widetilde{\Lambda}_k^{(\beta-1)/2}}.
    \]
\end{lemma}

\begin{proof}
    By the equivalence of the semigroup and Hölder spaces for exponents strictly between $0$ and $1$ (which applies here since $0 < \beta - 1 < 1$), the assumption implies that $\partial_l f \in \widetilde{\Lambda}_k^{(\beta-1)/2, H}$ for all $l$. Hence, there exists a constant $C > 0$ such that for all $\mathbf{x}, \mathbf{y} \in \mathbb{R}^N$ and all $l \in \{1, \dots, N\}$:
    \begin{equation}\label{eq:grad_bounds}
        |\partial_l f(\mathbf{x})| \leq C (1+\|\mathbf{x}\|)^{\beta-1}, \quad \text{and} \quad |\partial_l f(\mathbf{x}) - \partial_l f(\mathbf{y})| \leq 
        C\|\mathbf{x}-\mathbf{y}\|^{\beta-1}.
    \end{equation}
    Consequently, the gradient $\nabla f$ satisfies identical bounds (up to a dimensional constant).

    Recall that the Dunkl operator is given by
    \[
        T_j f(\mathbf{x}) = \partial_j f(\mathbf{x}) + \sum_{\alpha \in R} \frac{k(\alpha)}{2} \langle \alpha, e_j \rangle D_\alpha f(\mathbf{x}),
    \]
    where $D_\alpha f(\mathbf{x}) = \frac{f(\mathbf{x})-f(\sigma_\alpha(\mathbf{x}))}{\langle \alpha, \mathbf{x} \rangle}$. Since $\partial_j f \in \widetilde{\Lambda}_k^{(\beta-1)/2, H}$, it suffices to prove that each difference quotient $D_\alpha f$ also belongs to $\widetilde{\Lambda}_k^{(\beta-1)/2, H}$.  We can write $D_\alpha f$ as
    \begin{equation}\label{eq:integral_D_alpha}
        D_\alpha f(\mathbf{x}) = \int_0^1 \langle \nabla f(A_t \mathbf{x}), \alpha \rangle \, dt,
    \end{equation}
    where $A_t \mathbf{x} = \mathbf{x} - 2t\alpha \|\alpha\|^{-2}\langle \mathbf{x}, \alpha\rangle$. Observe that $A_t = I - 2t \|\alpha\|^{-2} \alpha \alpha^T$ is a symmetric matrix. Its eigenvalues are $1$ (on the orthogonal complement of $\alpha$) and $1-2t$ (in the direction of $\alpha$). Since $t \in [0,1]$, we have $|1-2t| \leq 1$, meaning the operator norm satisfies $\|A_t\| \leq 1$. Therefore, for any $\mathbf{x}, \mathbf{y} \in \mathbb{R}^N$:
    \[
        \|A_t \mathbf{x}\| \leq \|\mathbf{x}\| \quad \text{and} \quad \|A_t \mathbf{x} - A_t \mathbf{y}\| \leq \|\mathbf{x}-\mathbf{y}\|.
    \]

    Since the norm in the space $\widetilde{\Lambda}_k^{(\beta-1)/2, H}$ is defined via suprema (the weighted $L^\infty$-norm and the Hölder supremum), we can use Minkowski's inequality for integrals (i.e., taking the supremum inside the integral) to bound $D_\alpha f$. 

    First, for the weighted $L^\infty$-bound, we use \eqref{eq:grad_bounds} and the fact that $0 < \beta-1$:
    \begin{align*}
        |D_\alpha f(\mathbf{x})| &\leq \int_0^1 |\langle \nabla f(A_t \mathbf{x}), \alpha \rangle| \, dt \leq \|\alpha\| \int_0^1 \|\nabla f(A_t \mathbf{x})\| \, dt \leq C_1 \|\alpha\| \int_0^1 (1+\|A_t \mathbf{x}\|)^{\beta-1} \, dt \\
        &\leq C_1 \|\alpha\| \int_0^1 (1+\|\mathbf{x}\|)^{\beta-1} \, dt = C_1 \|\alpha\| (1+\|\mathbf{x}\|)^{\beta-1}.
    \end{align*}
    Thus, $\| D_\alpha f(\cdot) (1+\|\cdot\|)^{-(\beta-1)} \|_{L^\infty} < \infty$. Second, for the Hölder condition, we apply the difference bound from \eqref{eq:grad_bounds} directly inside the integral:
    \begin{align*}
        |D_\alpha f(\mathbf{x}) - D_\alpha f(\mathbf{y})| &\leq \int_0^1 |\langle \nabla f(A_t \mathbf{x}) - \nabla f(A_t \mathbf{y}), \alpha \rangle| \, dt \leq \|\alpha\| \int_0^1 \|\nabla f(A_t \mathbf{x}) - \nabla f(A_t \mathbf{y})\| \, dt \\
        &\leq C_2 \|\alpha\| \int_0^1 \|A_t \mathbf{x} - A_t \mathbf{y}\|^{\beta-1} \, dt \leq C_2 \|\alpha\| \int_0^1 \|\mathbf{x} - \mathbf{y}\|^{\beta-1} \, dt = C_2 \|\alpha\| \|\mathbf{x} - \mathbf{y}\|^{\beta-1}.
    \end{align*}
    Taking the supremum over $\mathbf{x} \neq \mathbf{y}$, we conclude that $D_\alpha f$ satisfies the Hölder condition.

    Consequently, $D_\alpha f \in \widetilde{\Lambda}_k^{(\beta-1)/2, H}$ for each root $\alpha \in R$. Since $T_j f$ is a finite linear combination of $\partial_j f$ and the terms $D_\alpha f$, it follows that $T_j f \in \widetilde{\Lambda}_k^{(\beta-1)/2, H}$. Applying the equivalence of spaces once more, we conclude that $T_j f \in \widetilde{\Lambda}_k^{(\beta-1)/2}$.
\end{proof}

\begin{lemma}\label{lem:second_Dunkl_derivative_heat}
    Let $1 < \beta < 2$ and $s > 0$. If $T_j f \in \widetilde{\Lambda}_k^{(\beta-1)/2}$ for all $j \in \{1, \dots, N\}$, then there exists a constant $C > 0$, independent of $f$ and $s$, such that for all $j \in \{1, \dots, N\}$:
    \[
        \|T_j^2 H_s f\|_{L^\infty} \leq C s^{\frac{\beta-1}{2} - \frac{1}{2}}.
    \]
\end{lemma}

\begin{proof}
    By the commutativity of the Dunkl operators with the heat semigroup, we have $T_j H_s f = H_s T_j f$. Thus, we can express the second Dunkl derivative as
    \[
        T_j^2 H_s f(\mathbf{x}) = T_j H_s (T_j f)(\mathbf{x}).
    \]
    Let $g = T_j f$. By our assumption and the equivalence of the spaces for exponents strictly between $0$ and $1$ (which applies since $0 < \beta - 1 < 1$), we have $g \in \widetilde{\Lambda}_k^{(\beta-1)/2, H}$. In particular, $g$ is uniformly Hölder continuous:
    \begin{equation}\label{eq:holder_g_euclidean}
        |g(\mathbf{y}) - g(\mathbf{x})| \leq \|g\|_{\widetilde{\Lambda}_k^{(\beta-1)/2, H}} \|\mathbf{y} - \mathbf{x}\|^{\beta-1}.
    \end{equation}

    Applying the Dunkl operator to the heat semigroup, we write
    $
        T_j H_s g(\mathbf{x}) = \int_{\mathbb{R}^N} T_{j,\mathbf{x}} h_s(\mathbf{x},\mathbf{y}) g(\mathbf{y}) \, dw(\mathbf{y}).
    $
    Since $\int_{\mathbb{R}^N} T_{j,\mathbf{x}} h_s(\mathbf{x},\mathbf{y}) \, dw(\mathbf{y}) = 0$, we subtract $g(\mathbf{x})$ inside the integral without altering its value:
    \[
        T_j H_s g(\mathbf{x}) = \int_{\mathbb{R}^N} T_{j,\mathbf{x}} h_s(\mathbf{x},\mathbf{y}) (g(\mathbf{y}) - g(\mathbf{x})) \, dw(\mathbf{y}).
    \]

    We now employ the fundamental identity for the Dunkl derivative of the heat kernel (see~\eqref{eq:T_j})
    
        $T_{j,\mathbf x}h_s(\mathbf{x},\mathbf{y}) = \frac{y_j-x_j}{2s}h_s(\mathbf{x},\mathbf{y})$. Substituting this identity and taking absolute values, we obtain
    \begin{align*}
        |T_j^2 H_s f(\mathbf{x})| &\leq \int_{\mathbb{R}^N} \frac{|y_j - x_j|}{2s} h_s(\mathbf{x},\mathbf{y}) |g(\mathbf{y}) - g(\mathbf{x})| \, dw(\mathbf{y}) \\
        &\leq \frac{\|g\|_{\widetilde{\Lambda}_k^{(\beta-1)/2, H}}}{2s} \int_{\mathbb{R}^N} \|\mathbf{y} - \mathbf{x}\| \|\mathbf{y} - \mathbf{x}\|^{\beta-1} h_s(\mathbf{x},\mathbf{y}) \, dw(\mathbf{y}) \\
        &= \frac{\|g\|_{\widetilde{\Lambda}_k^{(\beta-1)/2, H}}}{2s} \int_{\mathbb{R}^N} \|\mathbf{y} - \mathbf{x}\|^\beta h_s(\mathbf{x},\mathbf{y}) \, dw(\mathbf{y}).
    \end{align*}

To estimate this integral, we employ the explicit upper bound for the heat kernel provided by Theorem~\ref{teo:heat_new} (taking $m=0$ and $|\boldsymbol{\beta}| = |\boldsymbol{\beta}'| = 0$):
    \[
        h_s(\mathbf{x},\mathbf{y}) \leq C_1 \Big(1+\frac{\| \mathbf x-\mathbf y\|}{\sqrt{s}}\Big)^{-2} \mathcal G_{s\slash c} (\mathbf x,\mathbf y).
    \]
    Substituting this bound into our inequality yields
    \[
        |T_j^2 H_s f(\mathbf{x})| \leq \frac{C_1 \|g\|_{\widetilde{\Lambda}_k^{(\beta-1)/2, H}}}{2s} \int_{\mathbb{R}^N} \|\mathbf{y} - \mathbf{x}\|^\beta \Big(1+\frac{\| \mathbf x-\mathbf y\|}{\sqrt{s}}\Big)^{-2} \mathcal G_{s\slash c} (\mathbf x,\mathbf y) \, dw(\mathbf{y}).
    \]
    
    To isolate the time dependence, we extract the factor $s^{\beta/2}$ from the Euclidean terms:
    \[
        \|\mathbf{y} - \mathbf{x}\|^\beta \Big(1+\frac{\| \mathbf x-\mathbf y\|}{\sqrt{s}}\Big)^{-2} = s^{\beta/2} \left( \frac{\|\mathbf{y} - \mathbf{x}\|}{\sqrt{s}} \right)^\beta \Big(1+\frac{\| \mathbf x-\mathbf y\|}{\sqrt{s}}\Big)^{-2}.
    \]
    This is precisely where the assumption $\beta < 2$ plays a critical role. Because $\beta < 2$, the function $r \mapsto r^\beta (1+r)^{-2}$ is uniformly bounded on $[0, \infty)$ by a constant $C_\beta$. Consequently, the integrand is dominated by the Gaussian factor:
    \[
        |T_j^2 H_s f(\mathbf{x})| \leq C_1 C_\beta \|g\|_{\widetilde{\Lambda}_k^{(\beta-1)/2, H}} s^{\frac{\beta}{2}-1} \int_{\mathbb{R}^N} \mathcal G_{s\slash c} (\mathbf x,\mathbf y) \, dw(\mathbf{y}).
    \]
    
    Since the integral of $\mathcal G_{s\slash c}(\mathbf{x},\mathbf{y})$ with respect to $w(\mathbf{y})$ is uniformly bounded by a constant $C_2 > 0$ (independent of $\mathbf{x}$ and $s$, see Lemma~\ref{lem:homogeneous}), we arrive at
    $
        |T_j^2 H_s f(\mathbf{x})| \leq C s^{\frac{\beta}{2}-1} = C s^{\frac{\beta-1}{2} - \frac{1}{2}},
    $
    where $C = C_1 C_2 C_\beta \|g\|_{\widetilde{\Lambda}_k^{(\beta-1)/2, H}}$. Taking the supremum over all $\mathbf{x} \in \mathbb{R}^N$ completes the proof.
   
\end{proof}

\begin{lemma}\label{lem:f_in_semigroup_space}
    Let $1 < \beta < 2$ and suppose that $f$ satisfies $\|f(\cdot)(1+\|\cdot\|)^{-\beta}\|_{L^\infty} < \infty$. Assume further that for all $j \in \{1, \dots, N\}$ we have
    \[
        \|T_j^2 H_s f\|_{L^\infty} \leq C_0 s^{\frac{\beta-1}{2} - \frac{1}{2}}, \quad \text{for all } s > 0,
    \]
    where $C_0 > 0$ is a constant independent of $s$. Then $f \in \widetilde{\Lambda}_k^{\beta/2}$ and there exists $C>0$ such that
    \[
        \|f\|_{\widetilde{\Lambda}_k^{\beta/2}} \leq C \left( \|f(\cdot)(1+\|\cdot\|)^{-\beta}\|_{L^\infty} + C_0 \right).
    \]
\end{lemma}

\begin{proof}
    Since $1 < \beta < 2$, we have $1/2 < \beta/2 < 1$. According to the definition of the inhomogeneous weighted Lipschitz space $\widetilde{\Lambda}_k^{\beta/2}$, the parameter $n$ (the smallest positive integer strictly greater than $\beta/2$) is $n=1$. Thus, to show that $f \in \widetilde{\Lambda}_k^{\beta/2}$, it suffices to prove that
    \begin{equation}\label{eq:space_condition}
        \sup_{s>0} s^{1-\beta/2} \| \partial_s H_s f \|_{L^\infty} < \infty,
    \end{equation}
    since the weighted $L^\infty$-bound is satisfied by assumption. To establish \eqref{eq:space_condition}, we use the defining property of the Dunkl heat kernel $h_s(\mathbf{x},\mathbf{y})$, which solves the Dunkl heat equation:
    \[
        \partial_s h_s(\mathbf{x},\mathbf{y}) = \Delta_{k,\mathbf{x}} h_s(\mathbf{x},\mathbf{y}) = \sum_{j=1}^N T_{j,\mathbf{x}}^2 h_s(\mathbf{x},\mathbf{y}).
    \]
    Differentiating under the integral sign in $H_s f(\mathbf{x}) = \int_{\mathbb{R}^N} h_s(\mathbf{x},\mathbf{y}) f(\mathbf{y}) \, dw(\mathbf{y})$ with respect to $s$ yields
    $
        \partial_s H_s f(\mathbf{x}) = \sum_{j=1}^N T_j^2 H_s f(\mathbf{x})
    $.

    By the triangle inequality and our assumption on $T_j^2 H_s f$, we estimate:
    \[
        \|\partial_s H_s f\|_{L^\infty} \leq \sum_{j=1}^N \|T_j^2 H_s f\|_{L^\infty} \leq N C_0 s^{\frac{\beta-1}{2} - \frac{1}{2}} = N C_0 s^{\frac{\beta}{2} - 1}.
    \]

    Multiplying both sides by $s^{1-\beta/2}$ gives
    \[
        s^{1-\beta/2} \|\partial_s H_s f\|_{L^\infty} \leq N C_0 s^{1-\beta/2} s^{\frac{\beta}{2}-1} = N C_0.
    \]
    Taking the supremum over all $s > 0$, we conclude that
    $
        \sup_{s>0} s^{1-\beta/2} \| \partial_s H_s f \|_{L^\infty} \leq N C_0 < \infty
    $.
    This confirms that $f \in \widetilde{\Lambda}_k^{\beta/2}$ and completes the proof.
\end{proof}

\begin{proof}[Proof of Theorem~\ref{teo:equivalence}]
    This equivalence is a direct consequence of the mutual inclusions established in the preceding lemmas for both ranges of the parameter $\beta$.
    
    \textbf{Case $0 < \beta < 1$:} 
    Lemma~\ref{lem:inclusion_H_to_standard} yields the inclusion $\widetilde{\Lambda}_k^{\beta/2,H} \subset \widetilde{\Lambda}_k^{\beta/2}$ along with the corresponding norm estimate. The reverse inclusion, $\widetilde{\Lambda}_k^{\beta/2} \subset \widetilde{\Lambda}_k^{\beta/2,H}$, is guaranteed by Lemma~\ref{lem:inclusion_standard_to_H}.
    
    \textbf{Case $1 < \beta < 2$:} 
    For the first inclusion, $\widetilde{\Lambda}_k^{\beta/2} \subset \widetilde{\Lambda}_k^{\beta/2,H}$, the result is established in Lemma~\ref{lem:inclusion_standard_to_H_zygmund} based on the Zygmund condition. 
    
    To prove the reverse inclusion, $\widetilde{\Lambda}_k^{\beta/2,H} \subset \widetilde{\Lambda}_k^{\beta/2}$, let $f \in \widetilde{\Lambda}_k^{\beta/2,H}$. By Corollary~\ref{cor:derivative_semigroup_space}, we know that $f \in C^1(\mathbb{R}^N)$ and its classical partial derivatives satisfy $\partial_l f \in \widetilde{\Lambda}_k^{(\beta-1)/2}$ for all $l \in \{1, \dots, N\}$. 
    Applying Lemma~\ref{lem:Tj_semigroup_space}, this regularity transfers to the Dunkl derivatives, yielding $T_j f \in \widetilde{\Lambda}_k^{(\beta-1)/2}$ for all $j \in \{1, \dots, N\}$. 
    This allows us to invoke Lemma~\ref{lem:second_Dunkl_derivative_heat}, which provides the crucial time-decay estimate for the second Dunkl derivatives of the heat semigroup:
    \[
        \|T_j^2 H_s f\|_{L^\infty} \leq C s^{\frac{\beta-1}{2} - \frac{1}{2}}.
    \]
    Finally, since $f \in \widetilde{\Lambda}_k^{\beta/2,H}$ inherently satisfies the weighted $L^\infty$-bound ($\|f(\cdot)(1+\|\cdot\|)^{-\beta}\|_{L^\infty} < \infty$), we can apply Lemma~\ref{lem:f_in_semigroup_space} to conclude that $f \in \widetilde{\Lambda}_k^{\beta/2}$. Combining these steps provides the full equivalence of the norms and confirms that both definitions describe the exact same function space for the entire considered range of $\beta$.
\end{proof}

\begin{remark} \label{rem:beta_one}
The borderline case $\beta = 1$ is excluded from the statements of Theorems~\ref{teo:equivalence} and~\ref{thm:dunkl_schrodinger_equivalence}. Although the equivalence of the adapted spaces is expected to hold at this point, its rigorous verification is omitted here due to the high technical complexity. Indeed, a formal proof for $\beta = 1$ would require developing interpolation machinery (such as the Peetre $K$-method adapted to the Dunkl setting), which we choose to defer to a forthcoming paper.
\end{remark}

\section{Estimates for semigroup kernels and potential operators}\label{sec:aux}

\subsection{Integral estimates involving the potential}

This section is devoted to establishing key technical bounds for potential operators and integral kernels, which will be essential in the subsequent proofs.

\begin{lemma}\label{lem:integral_estimate_A}
Assume that $V \in {\rm{RH}}^{q}(dw)$, where $q>\max(1,\frac{\mathbf{N}}{2})$, and $V \geq 0$. There is a constant $C>0$ such that for all $\mathbf{x} \in \mathbb{R}^N$, for all $\sigma \in G$, and for all $t>0$ satisfying $\sqrt{t} \leq m(\sigma(\mathbf{x}))^{-1}$, we have
\begin{equation}
    \int_{A_\sigma(\mathbf{x})} V(\mathbf{y})\mathcal{G}_t(\mathbf{x},\mathbf{y})\,dw(\mathbf{y}) \leq \frac{C}{t} (\sqrt{t}m(\sigma(\mathbf{x})))^{\gamma},
\end{equation}
where $A_\sigma(\mathbf{x})=\{\mathbf{y} \in \mathbb{R}^N : d(\mathbf{x},\mathbf{y})=\|\sigma(\mathbf{x})-\mathbf{y}\|\}$ and $\gamma=2-\frac{\mathbf{N}}{q}$.
\end{lemma}

\begin{proof}
 For $\mathbf{y} \in A_{\sigma}$, the distance between the orbits simplifies to $d(\mathbf{x},\mathbf{y}) = \|\sigma(\mathbf{x})-\mathbf{y}\|$. Utilizing the $G$-invariance of the measure $w$, which implies $w(B(\mathbf{x},\sqrt{t})) = w(B(\sigma(\mathbf{x}),\sqrt{t}))$, we obtain
$$
\mathcal{G}_t(\mathbf{x},\mathbf{y}) \leq \frac{C}{w(B(\sigma(\mathbf{x}),\sqrt{t}))} \exp\left(-c\frac{\|\sigma(\mathbf{x})-\mathbf{y}\|^2}{t}\right).
$$

We decompose the set $A_{\sigma}$ into dyadic annuli centered at $\sigma(\mathbf{x})$. Let $A_{0,\sigma} = A_{\sigma} \cap B(\sigma(\mathbf{x}), \sqrt{t})$ and for $j \geq 1$, define
$$
A_{j,\sigma} = A_{\sigma} \cap \left\{\mathbf{y} \in \mathbb{R}^N : 2^{j-1}\sqrt{t} \leq \|\sigma(\mathbf{x})-\mathbf{y}\| < 2^j\sqrt{t}\right\}.
$$
Using this decomposition, we can estimate the integral as follows:
$$
\begin{aligned}
\int\limits_{A_{\sigma}} V(\mathbf{y})\mathcal{G}_t(\mathbf{x},\mathbf{y})\,dw(\mathbf{y}) &= \sum_{j=0}^{\infty} \int\limits_{A_{j,\sigma}} V(\mathbf{y})\mathcal{G}_t(\mathbf{x},\mathbf{y})\,dw(\mathbf{y}) \leq \frac{C}{w(B(\sigma(\mathbf{x}),\sqrt{t}))} \sum_{j=0}^{\infty} e^{-c 4^{j-1}} \int\limits_{A_{j,\sigma}} V(\mathbf{y})\,dw(\mathbf{y}).
\end{aligned}
$$

Recall the measure $\mu(E) = \int_E V(\mathbf{y})\,dw(\mathbf{y})$ defined in \eqref{eq:mu}. By Lemma~\ref{lem:mu_doubling}, $\mu$ is a doubling measure with a doubling constant $C_\mu > 0$. Thus, for each $j \geq 1$, the integral over the annulus $A_{j,\sigma}$ can be bounded by the integral over the corresponding ball:
$$
\int_{A_{j,\sigma}} V(\mathbf{y})\,dw(\mathbf{y}) \leq \mu(B(\sigma(\mathbf{x}), 2^j\sqrt{t})) \leq C_{\mu}^j \mu(B(\sigma(\mathbf{x}), \sqrt{t})).
$$
Inserting this into our series yields
$$
\begin{aligned}
\int_{A_{\sigma}} V(\mathbf{y})\mathcal{G}_t(\mathbf{x},\mathbf{y})\,dw(\mathbf{y}) &\leq \frac{C}{w(B(\sigma(\mathbf{x}),\sqrt{t}))} \sum_{j=0}^{\infty} e^{-c 4^{j-1}} C_{\mu}^j \int_{B(\sigma(\mathbf{x}), \sqrt{t})} V(\mathbf{y})\,dw(\mathbf{y}) \\
&\leq \frac{C'}{w(B(\sigma(\mathbf{x}),\sqrt{t}))} \int_{B(\sigma(\mathbf{x}), \sqrt{t})} V(\mathbf{y})\,dw(\mathbf{y}),
\end{aligned}
$$
where we have used the fact that the series $\sum_{j=0}^{\infty} e^{-c 4^{j-1}} C_{\mu}^j$ is convergent. To  refine the estimate, we rewrite the right-hand side by multiplying and dividing by $t$:
$$
\int_{A_{\sigma}} V(\mathbf{y})\mathcal{G}_t(\mathbf{x},\mathbf{y})\,dw(\mathbf{y}) \leq \frac{C'}{t} \left( \frac{t}{w(B(\sigma(\mathbf{x}),\sqrt{t}))} \int_{B(\sigma(\mathbf{x}), \sqrt{t})} V(\mathbf{y})\,dw(\mathbf{y}) \right).
$$

By our assumption $\sqrt{t} \leq m(\sigma(\mathbf{x}))^{-1}$, we can apply Lemma~\ref{lem:Rr} with $r_1 = \sqrt{t}$ and $r_2 = m(\sigma(\mathbf{x}))^{-1}$. This yields
$$
\begin{aligned}
\int_{A_{\sigma}} V(\mathbf{y})\mathcal{G}_t(\mathbf{x},\mathbf{y})\,dw(\mathbf{y}) &\leq \frac{C' C}{t} \left(\frac{\sqrt{t}}{m(\sigma(\mathbf{x}))^{-1}}\right)^{\gamma} \frac{m(\sigma(\mathbf{x}))^{-2}}{w(B(\sigma(\mathbf{x}),m(\sigma(\mathbf{x}))^{-1}))} \int_{B(\sigma(\mathbf{x}), m(\sigma(\mathbf{x}))^{-1})} V(\mathbf{y})\,dw(\mathbf{y}).
\end{aligned}
$$
By the definition of the function $m$ given in \eqref{eq:m}, the term representing the average of the potential $V$ at the critical radius $r_2 = m(\sigma(\mathbf{x}))^{-1}$ is bounded by $1$. Consequently, the estimate simplifies to
$$
\int_{A_{\sigma}} V(\mathbf{y})\mathcal{G}_t(\mathbf{x},\mathbf{y})\,dw(\mathbf{y}) \leq \frac{C''}{t} (\sqrt{t}m(\sigma(\mathbf{x})))^{\gamma},
$$
which concludes the proof.
\end{proof}

\begin{lemma}\label{lem:integral_estimate_Duhamel}

Assume that $V \in L^2_{\rm{loc}}(dw)$ and $V \geq 0$. For all $\mathbf{x} \in \mathbb{R}^N$ and $t>0$, we have

\begin{equation}\label{eq:Duhamel_bound}
    \int_{0}^{t} \int_{\mathbb{R}^N} V(\mathbf{y})k_s(\mathbf{x},\mathbf{y})\,dw(\mathbf{y})\,ds \leq 1.
\end{equation}

\end{lemma} 

\begin{proof}
By the Duhamel formula, for all $\mathbf{x},\mathbf{y} \in \mathbb{R}^N$ and $t>0$, we have
\begin{equation*}
    h_t(\mathbf{x},\mathbf{y}) - k_t(\mathbf{x},\mathbf{y}) = \int_{0}^{t} \int_{\mathbb{R}^N} k_{t-s}(\mathbf{x},\mathbf{z})V(\mathbf{z})h_s(\mathbf{z},\mathbf{y})\,dw(\mathbf{z})\,ds.
\end{equation*}
Integrating both sides with respect to $\mathbf{y}$ over $\mathbb{R}^N$ and applying Tonelli's theorem yields
\begin{equation*}
    \int_{\mathbb{R}^N} h_t(\mathbf{x},\mathbf{y})\,dw(\mathbf{y}) - \int_{\mathbb{R}^N} k_t(\mathbf{x},\mathbf{y})\,dw(\mathbf{y}) = \int_{0}^{t} \int_{\mathbb{R}^N} k_{t-s}(\mathbf{x},\mathbf{z})V(\mathbf{z}) \left( \int_{\mathbb{R}^N} h_s(\mathbf{z},\mathbf{y})\,dw(\mathbf{y}) \right) dw(\mathbf{z})\,ds.
\end{equation*}
Since $\int_{\mathbb{R}^N} h_r(\mathbf{u},\mathbf{y})\,dw(\mathbf{y}) = 1$ for all $\mathbf{u} \in \mathbb{R}^N, r>0$, and $k_t \geq 0$, we obtain
\begin{equation*}
    1 \geq \int_{0}^{t} \int_{\mathbb{R}^N} k_{t-s}(\mathbf{x},\mathbf{z})V(\mathbf{z})\,dw(\mathbf{z})\,ds.
\end{equation*}
Changing the variable of integration from $s$ to $t-s$ completes the proof.
\end{proof}

\begin{lemma}\label{lem:integral_estimate_1_over_t}
Assume that $V \in L^2_{\rm{loc}}(dw)$ and $V \geq 0$. There exists a constant $C > 0$ such that for all $\mathbf{x} \in \mathbb{R}^N$ and $t>0$, we have
\begin{equation}\label{eq:Schrodinger_1_over_t}
    \int_{\mathbb{R}^N} V(\mathbf{y})k_t(\mathbf{x},\mathbf{y})\,dw(\mathbf{y}) \leq \frac{C}{t}.
\end{equation}
\end{lemma}

\begin{proof}
By the Duhamel formula (see Lemma~\ref{lem:integral_estimate_Duhamel}), we have the identity
\begin{equation}\label{eq:duhamel_mass_identity}
    1 - \int_{\mathbb{R}^N} k_t(\mathbf{x},\mathbf{y})\,dw(\mathbf{y}) = \int_{0}^{t} \int_{\mathbb{R}^N} V(\mathbf{y})k_s(\mathbf{x},\mathbf{y})\,dw(\mathbf{y})\,ds.
\end{equation}
Differentiating both sides with respect to $t > 0$, we obtain
\begin{equation}\label{eq:derivative_identity}
    \int_{\mathbb{R}^N} V(\mathbf{y})k_t(\mathbf{x},\mathbf{y})\,dw(\mathbf{y}) = - \int_{\mathbb{R}^N} \partial_t k_t(\mathbf{x},\mathbf{y})\,dw(\mathbf{y}).
\end{equation}

Applying Theorem~\ref{thm:time_derivatives} with $m=1$, there exist constants $C_1, c_1 > 0$ such that
\begin{equation*}
    |\partial_t k_t(\mathbf{x},\mathbf{y})| \leq \frac{C_1}{t w(B(\mathbf{x},\sqrt{t}))} \exp\left(- c_1 \frac{d(\mathbf{x},\mathbf{y})^2}{t}\right).
\end{equation*}

Notice that the right-hand side is  bounded by $\frac{1}{t}\mathcal{G}_{t/c_1}(\mathbf{x},\mathbf{y})$ defined in \eqref{eq:mathcal_G}. Therefore,
\begin{equation*}
    \int_{\mathbb{R}^N} V(\mathbf{y})k_t(\mathbf{x},\mathbf{y})\,dw(\mathbf{y}) \leq \int_{\mathbb{R}^N} |\partial_t k_t(\mathbf{x},\mathbf{y})|\,dw(\mathbf{y}) \leq \frac{C_2}{t} \int_{\mathbb{R}^N} \mathcal{G}_{t/c_1}(\mathbf{x},\mathbf{y})\,dw(\mathbf{y}).
\end{equation*}

Since the function $\mathbf{y} \mapsto \mathcal{G}_{t/c_1}(\mathbf{x},\mathbf{y})$ integrates to a bounded constant uniformly in $\mathbf{x}$ and $t$ (due to the Gaussian decay on the space of homogeneous type), we conclude that
\begin{equation*}
    \int_{\mathbb{R}^N} V(\mathbf{y})k_t(\mathbf{x},\mathbf{y})\,dw(\mathbf{y}) \leq \frac{C}{t}.
\end{equation*}
\end{proof}

\begin{lemma}\label{lem:kernel_time_derivative_identity}
Assume that $V \in L^2_{\rm{loc}}(dw)$ and $V \geq 0$. For any $\mathbf{x}, \mathbf{y} \in \mathbb{R}^N$ and $0 < u < s$, we have
\begin{equation}\label{eq:kernel_time_derivative_identity}
    \partial_s k_s(\mathbf{x},\mathbf{y}) = \int_{\mathbb{R}^N} \partial_u k_u(\mathbf{x},\mathbf{z}) k_{s-u}(\mathbf{z},\mathbf{y})\,dw(\mathbf{z}).
\end{equation}
\end{lemma}

\begin{proof}
By the semigroup property, $K_s = K_u K_{s-u}$. Since the self-adjoint generator $-L$ commutes with $K_t$, differentiating with respect to $s$ yields $\partial_s K_s = -L K_s = (-L K_u) K_{s-u} = (\partial_u K_u) K_{s-u}$. Writing this identity in terms of integral kernels gives \eqref{eq:kernel_time_derivative_identity}.
\end{proof}

\begin{lemma}\label{lem:derivative_integral_estimate}
Assume that $V \in L^2_{\rm{loc}}(dw)$ and $V \geq 0$. There exists a constant $C > 0$ such that for all $\mathbf{x} \in \mathbb{R}^N$ and $t > 0$, we have
\begin{equation}\label{eq:Duhamel_derivative_bound}
    \int_{t/2}^{t} \int_{\mathbb{R}^N} V(\mathbf{y})|\partial_s k_s(\mathbf{x},\mathbf{y})|\,dw(\mathbf{y})\,ds \leq \frac{C}{t}.
\end{equation}
\end{lemma}

\begin{proof}
Fix $t > 0$ and set $u = t/4$. By Lemma~\ref{lem:kernel_time_derivative_identity}, we have $\partial_s k_s(\mathbf{x},\mathbf{y}) = \int_{\mathbb{R}^N} \partial_u k_u(\mathbf{x},\mathbf{z}) k_{s-u}(\mathbf{z},\mathbf{y})\,dw(\mathbf{z})$. Substituting this into the integral and applying Fubini's theorem, we obtain
\begin{align*}
    \int_{t/2}^{t} \int_{\mathbb{R}^N} V(\mathbf{y})|\partial_s k_s(\mathbf{x},\mathbf{y})|\,dw(\mathbf{y})\,ds 
    &\leq \int_{\mathbb{R}^N} |\partial_u k_u(\mathbf{x},\mathbf{z})| \left( \int_{t/4}^{3t/4} \int_{\mathbb{R}^N} V(\mathbf{y}) k_{r}(\mathbf{z},\mathbf{y})\,dw(\mathbf{y})\,dr \right) dw(\mathbf{z}),
\end{align*}
where we substituted $r = s - u$. By Lemma~\ref{lem:integral_estimate_Duhamel}, the inner double integral is bounded from above by $\int_{0}^{t} \int_{\mathbb{R}^N} V(\mathbf{y}) k_r(\mathbf{z},\mathbf{y})\,dw(\mathbf{y})\,dr \leq 1$. Consequently,
\begin{equation*}
    \int_{t/2}^{t} \int_{\mathbb{R}^N} V(\mathbf{y})|\partial_s k_s(\mathbf{x},\mathbf{y})|\,dw(\mathbf{y})\,ds \leq \int_{\mathbb{R}^N} |\partial_u k_u(\mathbf{x},\mathbf{z})|\,dw(\mathbf{z}).
\end{equation*}
Finally, by the Gaussian bounds from Theorem~\ref{thm:time_derivatives} (with $m=1$), the last integral is bounded by $C_2/u$ for some constant $C_2 > 0$. Since $u = t/4$, we obtain the desired bound with $C = 4C_2$.
\end{proof}

\section{Proof of Theorem \ref{thm:main_derivative_difference}}\label{sec:proof_1}

The primary objective of this section is to prove Theorem \ref{thm:main_derivative_difference}. In order to estimate the difference between the time derivatives of the Dunkl heat semigroup and the Dunkl--Schr\"odinger semigroup, we analyze the integral representation of their difference directly.  For a fixed $\mathbf{x} \in \mathbb{R}^N$, we partition the integration domain $\mathbb{R}^N$ into $|G|$ regions (up to sets of measure zero), corresponding to the elements of the reflection group $G$. Specifically, for each $\sigma \in G$, we define the set
\begin{equation}\label{eq:group_decomposition_sets}
    D_\sigma(\mathbf{x}) := \left\{ \mathbf{y} \in \mathbb{R}^N : d(\mathbf{x},\mathbf{y}) = \|\sigma(\mathbf{x}) - \mathbf{y}\| \right\}.
\end{equation}
We can express the difference of their time derivatives as
\begin{equation}\label{eq:semigroup_difference_integral}
    \partial_t H_t f(\mathbf{x}) - \partial_t K_t f(\mathbf{x}) = \int_{\mathbb{R}^N} \big( \partial_t h_t(\mathbf{x},\mathbf{y}) - \partial_t k_t(\mathbf{x},\mathbf{y}) \big) f(\mathbf{y}) \, dw(\mathbf{y}).
\end{equation}
Splitting the integral over $\mathbb{R}^N$ according to the partition $\{D_\sigma(\mathbf{x})\}_{\sigma \in G}$, we obtain
\begin{equation}\label{eq:semigroup_difference_split}
    |\partial_t H_t f(\mathbf{x}) - \partial_t K_t f(\mathbf{x})| \leq \sum_{\sigma \in G} \int_{D_\sigma(\mathbf{x})} \big| \partial_t h_t(\mathbf{x},\mathbf{y}) - \partial_t k_t(\mathbf{x},\mathbf{y}) \big| |f(\mathbf{y})| \, dw(\mathbf{y}).
\end{equation}
Our strategy will be to estimate the integral over each region $D_\sigma(\mathbf{x})$ separately. 


\begin{lemma}\label{lem:large_time_derivative_bound}
Assume that $V \in {\rm{RH}}^{q}(dw)$ for some $q>\max(1,\frac{\mathbf{N}}{2})$, and $V \geq 0$. Let $f$ be a measurable function such that $\|f m^\beta\|_{L^\infty(\mathbb{R}^N)} < \infty$ for some $0 < \beta < 2 - \frac{\mathbf{N}}{q}$. Let $\Gamma_t(\mathbf{x},\mathbf{y}) = \partial_t h_t(\mathbf{x},\mathbf{y}) - \partial_t k_t(\mathbf{x},\mathbf{y})$. For any fixed $\sigma \in G$, there exists a constant $C>0$ such that for all $\mathbf{x} \in \mathbb{R}^N$ and $t \geq m(\sigma(\mathbf{x}))^{-2}$, we have
\begin{equation}
    \int_{D_\sigma(\mathbf{x})} |\Gamma_t(\mathbf{x},\mathbf{y})| |f(\mathbf{y})| \, dw(\mathbf{y}) \leq C t^{-1+\beta/2} \|f m^\beta\|_{L^\infty(\mathbb{R}^N)}.
\end{equation}
\end{lemma}

\begin{proof}
By Theorem~\ref{thm:time_derivatives}, both $\partial_t h_t$ and $\partial_t k_t$ satisfy the Gaussian bound, hence
\begin{equation}\label{eq:gamma_bound_sigma}
    |\Gamma_t(\mathbf{x},\mathbf{y})| \leq \frac{C}{t w(B(\mathbf{x},\sqrt{t}))} \exp\left(- c \frac{d(\mathbf{x},\mathbf{y})^2}{t}\right).
\end{equation}
On the region $D_\sigma(\mathbf{x})$, we have $d(\mathbf{x},\mathbf{y}) = \|\sigma(\mathbf{x})-\mathbf{y}\|$. We decompose $D_\sigma(\mathbf{x})$ into a local part $I_1 = \{ \mathbf{y} \in D_\sigma(\mathbf{x}) : \|\sigma(\mathbf{x})-\mathbf{y}\| < m(\sigma(\mathbf{x}))^{-1} \}$ and a global part $I_2 = D_\sigma(\mathbf{x}) \setminus I_1$.

On $I_1$, Lemma~\ref{lem:m_growth}\eqref{eq:Shen_A} yields $m(\mathbf{y}) \geq C^{-1} m(\sigma(\mathbf{x}))$. Using $t \geq m(\sigma(\mathbf{x}))^{-2}$, we obtain
\begin{equation*}
    |f(\mathbf{y})| \leq \|f m^\beta\|_{L^\infty} m(\mathbf{y})^{-\beta} \leq C \|f m^\beta\|_{L^\infty} m(\sigma(\mathbf{x}))^{-\beta} \leq C \|f m^\beta\|_{L^\infty} t^{\beta/2}.
\end{equation*}
Combining this with \eqref{eq:gamma_bound_sigma} and extending the integration to $\mathbb{R}^N$ gives
\begin{equation*}
    \int_{I_1} |\Gamma_t(\mathbf{x},\mathbf{y})| |f(\mathbf{y})| \, dw(\mathbf{y}) \leq \frac{C \|f m^\beta\|_{L^\infty} t^{-1+\beta/2}}{w(B(\mathbf{x},\sqrt{t}))} \int_{\mathbb{R}^N} \exp\left(- c \frac{\|\sigma(\mathbf{x})-\mathbf{y}\|^2}{t}\right) dw(\mathbf{y}) \leq C t^{-1+\beta/2} \|f m^\beta\|_{L^\infty}.
\end{equation*}

On $I_2$, we apply Lemma~\ref{lem:m_growth}\eqref{eq:Shen_C} with $\lambda = \beta \frac{\kappa}{1+\kappa} < \beta$. Since $m(\sigma(\mathbf{x}))\|\sigma(\mathbf{x})-\mathbf{y}\| \geq 1$, we have
\begin{equation*}
    m(\mathbf{y})^{-\beta} \leq C m(\sigma(\mathbf{x}))^{-\beta} (1 + m(\sigma(\mathbf{x}))\|\sigma(\mathbf{x})-\mathbf{y}\|)^{\lambda} \leq C m(\sigma(\mathbf{x}))^{-\beta+\lambda} \|\sigma(\mathbf{x})-\mathbf{y}\|^\lambda.
\end{equation*}
Since $m(\sigma(\mathbf{x})) \geq t^{-1/2}$ and $-\beta+\lambda < 0$, it follows that $m(\sigma(\mathbf{x}))^{-\beta+\lambda} \leq t^{(\beta-\lambda)/2}$. Thus,
\begin{equation*}
    |f(\mathbf{y})| \leq C \|f m^\beta\|_{L^\infty} t^{\beta/2} \left( \frac{\|\sigma(\mathbf{x})-\mathbf{y}\|}{\sqrt{t}} \right)^\lambda.
\end{equation*}
Integrating over $I_2$ using \eqref{eq:gamma_bound_sigma} yields
\begin{equation*}
    \int_{I_2} |\Gamma_t(\mathbf{x},\mathbf{y})| |f(\mathbf{y})| \, dw(\mathbf{y}) \leq \frac{C \|f m^\beta\|_{L^\infty} t^{-1+\beta/2}}{w(B(\mathbf{x},\sqrt{t}))} \int_{\mathbb{R}^N} \left( \frac{\|\sigma(\mathbf{x})-\mathbf{y}\|}{\sqrt{t}} \right)^\lambda \exp\left(- c \frac{\|\sigma(\mathbf{x})-\mathbf{y}\|^2}{t}\right) dw(\mathbf{y}).
\end{equation*}
Since the function $u \mapsto u^\lambda \exp(-c u^2)$ is uniformly bounded on $[0, \infty)$, the integral is bounded by a constant independent of $\mathbf{x}$ and $t$. This establishes the bound for $I_2$ and completes the proof.
\end{proof}


\begin{lemma}\label{lem:far_region_estimate}
Assume that $V \in {\rm{RH}}^{q}(dw)$ with $q>\max(1,\frac{\mathbf{N}}{2})$, and $V \geq 0$. Let $f$ be a measurable function such that $\|f m^\beta\|_{L^\infty(\mathbb{R}^N)} < \infty$ for some $0 < \beta < 2 - \frac{\mathbf{N}}{q}$. Let $\Gamma_t(\mathbf{x},\mathbf{y}) = \partial_t h_t(\mathbf{x},\mathbf{y}) - \partial_t k_t(\mathbf{x},\mathbf{y})$. For any fixed $\sigma \in G$ and $\mathbf{x} \in \mathbb{R}^N$, let $C_{\sigma,\mathbf{x}} = \{ \mathbf{y} \in D_\sigma(\mathbf{x}) : \|\sigma(\mathbf{x})-\mathbf{y}\| \geq m(\sigma(\mathbf{x}))^{-1} \}$. Then for all $t \leq m(\sigma(\mathbf{x}))^{-2}$, we have
\begin{equation}\label{eq:far_region_bound}
    \int_{C_{\sigma,\mathbf{x}}} |\Gamma_t(\mathbf{x},\mathbf{y})| |f(\mathbf{y})| \, dw(\mathbf{y}) \leq C t^{-1+\beta/2} \|f m^\beta\|_{L^\infty(\mathbb{R}^N)}.
\end{equation}
\end{lemma}

\begin{proof}
By Theorem~\ref{thm:time_derivatives}, the kernel difference $\Gamma_t(\mathbf{x},\mathbf{y})$ satisfies the Gaussian bound
\begin{equation}\label{eq:gamma_bound_far}
    |\Gamma_t(\mathbf{x},\mathbf{y})| \leq \frac{C}{t w(B(\mathbf{x},\sqrt{t}))} \exp\left(- c \frac{d(\mathbf{x},\mathbf{y})^2}{t}\right).
\end{equation}
Recall that $d(\mathbf{x},\mathbf{y}) = \|\sigma(\mathbf{x})-\mathbf{y}\|$ on $D_\sigma(\mathbf{x})$. For any $\mathbf{y} \in \mathbb{R}^N$, we have $|f(\mathbf{y})| \leq \|f m^\beta\|_{L^\infty} m(\mathbf{y})^{-\beta}$. Applying Lemma~\ref{lem:m_growth}\eqref{eq:Shen_C} with $\lambda = \beta \frac{\kappa}{1+\kappa} < \beta$, we obtain
\begin{equation*}
    m(\mathbf{y})^{-\beta} \leq C m(\sigma(\mathbf{x}))^{-\beta} (1 + m(\sigma(\mathbf{x}))\|\sigma(\mathbf{x})-\mathbf{y}\|)^{\lambda}.
\end{equation*}
For $\mathbf{y} \in C_{\sigma,\mathbf{x}}$, the condition $m(\sigma(\mathbf{x}))\|\sigma(\mathbf{x})-\mathbf{y}\| \geq 1$ implies that the term $(1 + m(\sigma(\mathbf{x}))\|\sigma(\mathbf{x})-\mathbf{y}\|)^\lambda$ is bounded by $2^\lambda m(\sigma(\mathbf{x}))^\lambda \|\sigma(\mathbf{x})-\mathbf{y}\|^\lambda$. Hence,
\begin{equation*}
    |f(\mathbf{y})| \leq C \|f m^\beta\|_{L^\infty} m(\sigma(\mathbf{x}))^{-\beta+\lambda} \|\sigma(\mathbf{x})-\mathbf{y}\|^\lambda.
\end{equation*}
Since $-\beta+\lambda < 0$ and $m(\sigma(\mathbf{x})) \geq \|\sigma(\mathbf{x})-\mathbf{y}\|^{-1}$ on $C_{\sigma,\mathbf{x}}$, we deduce that $m(\sigma(\mathbf{x}))^{-\beta+\lambda} \leq \|\sigma(\mathbf{x})-\mathbf{y}\|^{\beta-\lambda}$. Multiplying this by $\|\sigma(\mathbf{x})-\mathbf{y}\|^\lambda$ gives
\begin{equation*}
    |f(\mathbf{y})| \leq C \|f m^\beta\|_{L^\infty} \|\sigma(\mathbf{x})-\mathbf{y}\|^\beta = C \|f m^\beta\|_{L^\infty} t^{\beta/2} \left( \frac{\|\sigma(\mathbf{x})-\mathbf{y}\|}{\sqrt{t}} \right)^\beta.
\end{equation*}
Substituting this bound and \eqref{eq:gamma_bound_far} into the integral over $C_{\sigma,\mathbf{x}}$, we get
\begin{equation*}
    \int_{C_{\sigma,\mathbf{x}}} |\Gamma_t(\mathbf{x},\mathbf{y})| |f(\mathbf{y})| \, dw(\mathbf{y}) \leq \frac{C \|f m^\beta\|_{L^\infty} t^{-1+\beta/2}}{w(B(\mathbf{x},\sqrt{t}))} \int_{\mathbb{R}^N} \left( \frac{\|\sigma(\mathbf{x})-\mathbf{y}\|}{\sqrt{t}} \right)^\beta \exp\left(- c \frac{\|\sigma(\mathbf{x})-\mathbf{y}\|^2}{t}\right) dw(\mathbf{y}).
\end{equation*}
Therefore, the integral is bounded by a uniform constant, which yields the desired estimate \eqref{eq:far_region_bound}.
\end{proof}


\begin{lemma}\label{lem:critical_radius_comparability}
Assume that $V \in {\rm{RH}}^{q}(dw)$ with $q>\max(1,\frac{\mathbf{N}}{2})$ and $V \geq 0$. For any fixed $\sigma \in G$ and $\mathbf{x} \in \mathbb{R}^N$, let $C_{\sigma,\mathbf{x}}^c = \{ \mathbf{y} \in D_\sigma(\mathbf{x}) : \|\sigma(\mathbf{x})-\mathbf{y}\| < m(\sigma(\mathbf{x}))^{-1} \}$. Then there exists a constant $C \geq 1$ such that for all $\mathbf{x} \in \mathbb{R}^N$ and $\mathbf{y} \in C_{\sigma,\mathbf{x}}^c$, we have
\begin{equation}\label{eq:m_comparability}
    C^{-1} \frac{1}{m(\sigma(\mathbf{x}))} \leq \frac{1}{m(\mathbf{y})} \leq C \frac{1}{m(\sigma(\mathbf{x}))}.
\end{equation}
In particular, $m(\mathbf{y}) \sim m(\sigma(\mathbf{x}))$ uniformly on $C_{\sigma,\mathbf{x}}^c$.
\end{lemma}

\begin{proof}
Let $\mathbf{y} \in C_{\sigma,\mathbf{x}}^c$. By definition, this implies the local Euclidean distance bound $\|\sigma(\mathbf{x})-\mathbf{y}\| < m(\sigma(\mathbf{x}))^{-1}$. Applying property \eqref{eq:Shen_A} of Lemma~\ref{lem:m_growth} directly to the points $\sigma(\mathbf{x})$ and $\mathbf{y}$ yields
\begin{equation*}
    C^{-1} m(\sigma(\mathbf{x})) \leq m(\mathbf{y}) \leq C m(\sigma(\mathbf{x})).
\end{equation*}
\end{proof}

\begin{lemma}\label{lem:kernel_duhamel_derivative}
Assume that $V \in L^2_{\rm{loc}}(dw)$ and $V \geq 0$. For all $\mathbf{x}, \mathbf{y} \in \mathbb{R}^N$ and $t>0$, the following identity holds:
\begin{equation}\label{eq:kernel_duhamel_derivative}
\begin{split}
    \partial_t \big( h_t(\mathbf{x},\mathbf{y}) - k_t(\mathbf{x},\mathbf{y}) \big) &= \int_0^{t/2} \int_{\mathbb{R}^N} \partial_t h_{t-s}(\mathbf{x},\mathbf{z}) V(\mathbf{z}) k_s(\mathbf{z},\mathbf{y})\,dw(\mathbf{z})\,ds \\
    &\quad + \int_{t/2}^{t} \int_{\mathbb{R}^N} h_{t-s}(\mathbf{x},\mathbf{z}) V(\mathbf{z}) \partial_s k_s(\mathbf{z},\mathbf{y})\,dw(\mathbf{z})\,ds \\
    &\quad + \int_{\mathbb{R}^N} h_{t/2}(\mathbf{x},\mathbf{z}) V(\mathbf{z}) k_{t/2}(\mathbf{z},\mathbf{y})\,dw(\mathbf{z}).
\end{split}
\end{equation}
\end{lemma}

\begin{proof}
By the standard Duhamel formula for the semigroups $H_t$ and $K_t$, we have the integral equation for their kernels:
\begin{equation*}
    h_t(\mathbf{x},\mathbf{y}) - k_t(\mathbf{x},\mathbf{y}) = \int_0^{t} \int_{\mathbb{R}^N} h_{t-s}(\mathbf{x},\mathbf{z}) V(\mathbf{z}) k_s(\mathbf{z},\mathbf{y}) \, dw(\mathbf{z}) \, ds.
\end{equation*}
To differentiate this expression with respect to $t$, we first split the time integral into two parts: $I_1$ over $[0, t/2]$ and $I_2$ over $[t/2, t]$.

\textit{Derivative of $I_1$:}
Using the Leibniz integral rule, differentiating $I_1$ with respect to the parameter $t$ produces a term from differentiating the integrand and a boundary term from the upper limit evaluated at $s = t/2$:
\begin{align*}
    \partial_t \left( \int_0^{t/2} \int_{\mathbb{R}^N} h_{t-s}(\mathbf{x},\mathbf{z}) V(\mathbf{z}) k_s(\mathbf{z},\mathbf{y}) \, dw(\mathbf{z}) \, ds \right) 
    &= \int_0^{t/2} \int_{\mathbb{R}^N} \partial_t h_{t-s}(\mathbf{x},\mathbf{z}) V(\mathbf{z}) k_s(\mathbf{z},\mathbf{y}) \, dw(\mathbf{z}) \, ds \\
    &\quad + \frac{1}{2} \int_{\mathbb{R}^N} h_{t/2}(\mathbf{x},\mathbf{z}) V(\mathbf{z}) k_{t/2}(\mathbf{z},\mathbf{y}) \, dw(\mathbf{z}).
\end{align*}

\textit{Derivative of $I_2$:}
For the second interval $[t/2, t]$, we first use the change of variables $\tau = t - s$ (so $d\tau = -ds$, and the limits $[t/2, t]$ become $[t/2, 0]$, which we reverse to $[0, t/2]$):
\begin{equation*}
    I_2 = \int_0^{t/2} \int_{\mathbb{R}^N} h_{\tau}(\mathbf{x},\mathbf{z}) V(\mathbf{z}) k_{t-\tau}(\mathbf{z},\mathbf{y}) \, dw(\mathbf{z}) \, d\tau.
\end{equation*}
Now we apply the Leibniz rule again to differentiate with respect to $t$:
\begin{align*}
    \partial_t I_2 &= \int_0^{t/2} \int_{\mathbb{R}^N} h_{\tau}(\mathbf{x},\mathbf{z}) V(\mathbf{z}) \partial_t k_{t-\tau}(\mathbf{z},\mathbf{y}) \, dw(\mathbf{z}) \, d\tau + \frac{1}{2} \int_{\mathbb{R}^N} h_{t/2}(\mathbf{x},\mathbf{z}) V(\mathbf{z}) k_{t/2}(\mathbf{z},\mathbf{y}) \, dw(\mathbf{z}).
\end{align*}
We now reverse the change of variables, setting $s = t - \tau$. The integration range $\tau \in [0, t/2]$ shifts back to $s \in [t/2, t]$. Furthermore, by the chain rule, $\partial_t k_{t-\tau}(\mathbf{z},\mathbf{y}) = \partial_s k_s(\mathbf{z},\mathbf{y})$. Thus,
\begin{equation*}
    \partial_t I_2 = \int_{t/2}^{t} \int_{\mathbb{R}^N} h_{t-s}(\mathbf{x},\mathbf{z}) V(\mathbf{z}) \partial_s k_s(\mathbf{z},\mathbf{y}) \, dw(\mathbf{z}) \, ds + \frac{1}{2} \int_{\mathbb{R}^N} h_{t/2}(\mathbf{x},\mathbf{z}) V(\mathbf{z}) k_{t/2}(\mathbf{z},\mathbf{y}) \, dw(\mathbf{z}).
\end{equation*}

Summing the derivatives of $I_1$ and $I_2$, the two boundary terms add up exactly to 
\[ \int_{\mathbb{R}^N} h_{t/2}(\mathbf{x},\mathbf{z}) V(\mathbf{z}) k_{t/2}(\mathbf{z},\mathbf{y}) \, dw(\mathbf{z}),\]
 matching the third term of our desired expression. 
\end{proof}


\begin{lemma}\label{lem:z_far_region_estimates_integrated}
Assume that $V \in {\rm{RH}}^{q}(dw)$ with $q>\max(1,\frac{\mathbf{N}}{2})$ and $V \geq 0$. Let $0 < \beta < 2 - \frac{\mathbf{N}}{q}$. For a fixed $\sigma \in G$ and $\mathbf{x} \in \mathbb{R}^N$, define the local region $C_{\sigma,\mathbf{x}}^c = \{ \mathbf{y} \in D_\sigma(\mathbf{x}) : \|\sigma(\mathbf{x})-\mathbf{y}\| < m(\sigma(\mathbf{x}))^{-1} \}$ and the far region $C_{\mathbf{y},\mathbf{z},1} = \{ \mathbf{z} \in \mathbb{R}^N : \|\mathbf{y}-\mathbf{z}\| \geq m(\mathbf{y})^{-1} \}$. Then for any $t>0$, we have:
\begin{align}
    \int_{C_{\sigma,\mathbf{x}}^c} m(\sigma(\mathbf{x}))^{-\beta} \int_0^{t/2} \int_{C_{\mathbf{y},\mathbf{z},1}} |\partial_t h_{t-s}(\mathbf{y},\mathbf{z})| V(\mathbf{z}) k_s(\mathbf{z},\mathbf{x})\,dw(\mathbf{z})\,ds \,dw(\mathbf{y}) &\leq C t^{-1+\beta/2}, \label{eq:D1_bound_int} \\
    \int_{C_{\sigma,\mathbf{x}}^c} m(\sigma(\mathbf{x}))^{-\beta} \int_{t/2}^{t} \int_{C_{\mathbf{y},\mathbf{z},1}} h_{t-s}(\mathbf{y},\mathbf{z}) V(\mathbf{z}) |\partial_s k_s(\mathbf{z},\mathbf{x})|\,dw(\mathbf{z})\,ds \,dw(\mathbf{y}) &\leq C t^{-1+\beta/2}, \label{eq:D2_bound_int} \\
    \int_{C_{\sigma,\mathbf{x}}^c} m(\sigma(\mathbf{x}))^{-\beta} \int_{C_{\mathbf{y},\mathbf{z},1}} h_{t/2}(\mathbf{y},\mathbf{z}) V(\mathbf{z}) k_{t/2}(\mathbf{z},\mathbf{x})\,dw(\mathbf{z}) \,dw(\mathbf{y}) &\leq C t^{-1+\beta/2}. \label{eq:D3_bound_int}
\end{align}
\end{lemma}

\begin{proof}
For any $\mathbf{y} \in C_{\sigma,\mathbf{x}}^c$, Lemma~\ref{lem:critical_radius_comparability} yields $m(\sigma(\mathbf{x})) \sim m(\mathbf{y})$, so $m(\sigma(\mathbf{x}))^{-\beta} \leq C m(\mathbf{y})^{-\beta}$. Furthermore, for $\mathbf{z} \in C_{\mathbf{y},\mathbf{z},1}$, the definition of the region implies $m(\mathbf{y}) \geq \|\mathbf{y}-\mathbf{z}\|^{-1}$. Combining these gives the pointwise weight bound in terms of the purely Euclidean distance between $\mathbf{y}$ and $\mathbf{z}$:
\begin{equation*}
    m(\sigma(\mathbf{x}))^{-\beta} \leq C m(\mathbf{y})^{-\beta} \leq C \|\mathbf{y}-\mathbf{z}\|^\beta.
\end{equation*}

To handle the polynomial growth of the factor $\|\mathbf{y}-\mathbf{z}\|^\beta$, we apply Theorem~\ref{teo:heat_new}. For any $\tau > 0$, the upper bounds for the Dunkl heat kernel $h_\tau(\mathbf{y},\mathbf{z})$ and its time derivative inherently contain the Euclidean decay factor $\left(1 + \frac{\|\mathbf{y}-\mathbf{z}\|}{\sqrt{\tau}}\right)^{-2}$. Since $\beta < 2 - \mathbf{N}/q < 2$, the real function $u \mapsto \frac{u^\beta}{(1+u)^2}$ is uniformly bounded by a constant on $[0, \infty)$. Consequently,
\begin{equation}\label{eq:euclidean_absorption}
    \left(1 + \frac{\|\mathbf{y}-\mathbf{z}\|}{\sqrt{\tau}}\right)^{-2} \|\mathbf{y}-\mathbf{z}\|^\beta = \tau^{\beta/2} \left( \frac{\|\mathbf{y}-\mathbf{z}\|}{\sqrt{\tau}} \right)^\beta \left(1 + \frac{\|\mathbf{y}-\mathbf{z}\|}{\sqrt{\tau}}\right)^{-2} \leq C \tau^{\beta/2}.
\end{equation}
Therefore, Theorem~\ref{teo:heat_new} provides the following bounds incorporating the weight:
\begin{align}
    h_\tau(\mathbf{y},\mathbf{z}) m(\sigma(\mathbf{x}))^{-\beta} &\leq C \tau^{\beta/2} \mathcal{G}_{\tau/c}(\mathbf{y},\mathbf{z}), \label{eq:h_weight_bound} \\
    |\partial_\tau h_\tau(\mathbf{y},\mathbf{z})| m(\sigma(\mathbf{x}))^{-\beta} &\leq C \tau^{-1+\beta/2} \mathcal{G}_{\tau/c}(\mathbf{y},\mathbf{z}). \label{eq:dh_weight_bound}
\end{align}
Recall that for any $\tau > 0$ and $\mathbf{z} \in \mathbb{R}^N$, the Gaussian kernel integrates to a uniform constant: $\int_{\mathbb{R}^N} \mathcal{G}_{\tau/c}(\mathbf{y},\mathbf{z}) \, dw(\mathbf{y}) \leq C$. We now apply this to each term. For $s \in [0, t/2]$, we have $t-s \geq t/2$, meaning $t-s \sim t$. Applying \eqref{eq:dh_weight_bound} with $\tau = t-s$, we obtain:
\begin{equation*}
    |\partial_t h_{t-s}(\mathbf{y},\mathbf{z})| m(\sigma(\mathbf{x}))^{-\beta} \leq C (t-s)^{-1+\beta/2} \mathcal{G}_{(t-s)/c}(\mathbf{y},\mathbf{z}) \leq C t^{-1+\beta/2} \mathcal{G}_{t/c'}(\mathbf{y},\mathbf{z}).
\end{equation*}
Integrating this bound over $dw(\mathbf{y})$ on the local region $C_{\sigma,\mathbf{x}}^c$ yields:
\begin{equation*}
    \int_{C_{\sigma,\mathbf{x}}^c} |\partial_t h_{t-s}(\mathbf{y},\mathbf{z})| m(\sigma(\mathbf{x}))^{-\beta} \, dw(\mathbf{y}) \leq C t^{-1+\beta/2} \int_{\mathbb{R}^N} \mathcal{G}_{t/c'}(\mathbf{y},\mathbf{z}) \, dw(\mathbf{y}) \leq C t^{-1+\beta/2}.
\end{equation*}
Substituting this back into the left-hand side of \eqref{eq:D1_bound_int} leaves:
\begin{equation*}
    C t^{-1+\beta/2} \int_0^{t/2} \int_{\mathbb{R}^N} V(\mathbf{z}) k_s(\mathbf{z},\mathbf{x}) \, dw(\mathbf{z}) \, ds \leq C t^{-1+\beta/2},
\end{equation*}
where the remaining integral is bounded by $1$ according to Lemma~\ref{lem:integral_estimate_Duhamel}.

In order to estimate the second term of \eqref{eq:D2_bound_int}, $s \in [t/2, t]$, we apply \eqref{eq:h_weight_bound} with $\tau = t-s \leq t/2$:
\begin{equation*}
    h_{t-s}(\mathbf{y},\mathbf{z}) m(\sigma(\mathbf{x}))^{-\beta} \leq C (t-s)^{\beta/2} \mathcal{G}_{(t-s)/c}(\mathbf{y},\mathbf{z}) \leq C t^{\beta/2} \mathcal{G}_{(t-s)/c}(\mathbf{y},\mathbf{z}).
\end{equation*}
Integrating with respect to $dw(\mathbf{y})$ gives:
\begin{equation*}
    \int_{C_{\sigma,\mathbf{x}}^c} h_{t-s}(\mathbf{y},\mathbf{z}) m(\sigma(\mathbf{x}))^{-\beta} \, dw(\mathbf{y}) \leq C t^{\beta/2} \int_{\mathbb{R}^N} \mathcal{G}_{(t-s)/c}(\mathbf{y},\mathbf{z}) \, dw(\mathbf{y}) \leq C t^{\beta/2}.
\end{equation*}
Plugging this result into the left-hand side of \eqref{eq:D2_bound_int} yields:
\begin{equation*}
    C t^{\beta/2} \int_{t/2}^{t} \int_{\mathbb{R}^N} V(\mathbf{z}) |\partial_s k_s(\mathbf{z},\mathbf{x})| \, dw(\mathbf{z}) \, ds \leq C t^{\beta/2} \left( \frac{C}{t} \right) = C t^{-1+\beta/2},
\end{equation*}
where we used Lemma~\ref{lem:derivative_integral_estimate} to bound the remaining double integral by $C/t$.

Finally, for the third term \eqref{eq:D3_bound_int}, setting $\tau = t/2$ in \eqref{eq:h_weight_bound} gives:
\begin{equation*}
    h_{t/2}(\mathbf{y},\mathbf{z}) m(\sigma(\mathbf{x}))^{-\beta} \leq C t^{\beta/2} \mathcal{G}_{t/c'}(\mathbf{y},\mathbf{z}).
\end{equation*}
Integrating over $dw(\mathbf{y})$ yields a bound of $C t^{\beta/2}$. Substituting this into \eqref{eq:D3_bound_int} leads to:
\begin{equation*}
    C t^{\beta/2} \int_{\mathbb{R}^N} V(\mathbf{z}) k_{t/2}(\mathbf{z},\mathbf{x}) \, dw(\mathbf{z}) \leq C t^{\beta/2} \left( \frac{C}{t} \right) = C t^{-1+\beta/2},
\end{equation*}
where we used Lemma~\ref{lem:integral_estimate_1_over_t}. This concludes the proof for all three estimates.
\end{proof}


\begin{lemma}\label{lem:geometric_m_equivalence}
Assume that $V \in {\rm{RH}}^{q}(dw)$ with $q>\max(1,\frac{\mathbf{N}}{2})$ and $V \geq 0$. Let $\mathbf{x}, \mathbf{y}, \mathbf{z} \in \mathbb{R}^N$ and let $\sigma \in G$ be a fixed element. Assume that $\mathbf{y} \in C_{\sigma,\mathbf{x}}^c$ and $\mathbf{z} \in C_{\mathbf{y},\mathbf{z},1}^c$, where the regions are defined as $C_{\sigma,\mathbf{x}}^c = \{ \mathbf{y} \in \mathbb{R}^N : \|\sigma(\mathbf{x})-\mathbf{y}\| < m(\sigma(\mathbf{x}))^{-1} \}$ and $C_{\mathbf{y},\mathbf{z},1}^c = \{ \mathbf{z} \in \mathbb{R}^N : \|\mathbf{y}-\mathbf{z}\| < m(\mathbf{y})^{-1} \}$. 

Let $\sigma_1 \in G$ be an element such that $d(\mathbf{x},\mathbf{z}) = \|\sigma_1(\mathbf{x})-\mathbf{z}\|$. Then $m(\sigma(\mathbf{x})) \sim m(\mathbf{y}) \sim m(\mathbf{z}) \sim m(\sigma_1(\mathbf{x}))$. 

Consequently, there exists a constant $C>0$ such that if $t \leq m(\sigma(\mathbf{x}))^{-2}$, we have
\begin{equation}
    t \leq C m(\sigma_1(\mathbf{x}))^{-2}.
\end{equation}
\end{lemma}

\begin{proof}
By the definition of the local region $C_{\sigma,\mathbf{x}}^c$, we have the Euclidean distance bound $\|\sigma(\mathbf{x})-\mathbf{y}\| < m(\sigma(\mathbf{x}))^{-1}$. Applying Lemma~\ref{lem:m_growth} \eqref{eq:Shen_A} to the points $\sigma(\mathbf{x})$ and $\mathbf{y}$ yields $m(\sigma(\mathbf{x})) \sim m(\mathbf{y})$. Similarly, the assumption $\mathbf{z} \in C_{\mathbf{y},\mathbf{z},1}^c$ implies $\|\mathbf{y}-\mathbf{z}\| < m(\mathbf{y})^{-1}$. Applying Lemma~\ref{lem:m_growth} \eqref{eq:Shen_A} again to $\mathbf{y}$ and $\mathbf{z}$ gives $m(\mathbf{y}) \sim m(\mathbf{z})$. Thus, we obtain the first chain of equivalences:
\begin{equation}\label{eq:m_chain_1}
    m(\sigma(\mathbf{x})) \sim m(\mathbf{y}) \sim m(\mathbf{z}).
\end{equation}

Next, we bound the Euclidean distance between $\sigma(\mathbf{x})$ and $\mathbf{z}$ using the triangle inequality and the comparability of the critical radii established above:
\begin{equation*}
    \|\sigma(\mathbf{x})-\mathbf{z}\| \leq \|\sigma(\mathbf{x})-\mathbf{y}\| + \|\mathbf{y}-\mathbf{z}\| < \frac{1}{m(\sigma(\mathbf{x}))} + \frac{1}{m(\mathbf{y})} \leq \frac{C}{m(\mathbf{z})}.
\end{equation*}

Now we consider the element $\sigma_1 \in G$. By the definition of the orbit distance $d(\mathbf{x},\mathbf{z}) = \min_{g \in G} \|g(\mathbf{x})-\mathbf{z}\|$, and the fact that $\sigma_1$ realizes this minimum while $\sigma$ is just another element in the group $G$, we obtain:
\begin{equation*}
    \|\sigma_1(\mathbf{x})-\mathbf{z}\| = d(\mathbf{x},\mathbf{z}) \leq \|\sigma(\mathbf{x})-\mathbf{z}\| \leq \frac{C}{m(\mathbf{z})}.
\end{equation*}

Since the distance between $\sigma_1(\mathbf{x})$ and $\mathbf{z}$ is bounded by $C/m(\mathbf{z})$, we can invoke the polynomial growth property from Lemma~\ref{lem:m_growth} \eqref{eq:Shen_B} (or equivalent standard bounds) to deduce that $m(\sigma_1(\mathbf{x})) \sim m(\mathbf{z})$. Specifically, we have:
\begin{equation*}
    m(\sigma_1(\mathbf{x})) \leq C m(\mathbf{z}) (1+m(\mathbf{z})\|\sigma_1(\mathbf{x})-\mathbf{z}\|)^\kappa \leq C m(\mathbf{z}) (1+C)^\kappa \leq C' m(\mathbf{z}),
\end{equation*}
and the lower bound follows analogously from \eqref{eq:Shen_C}. Combining this with \eqref{eq:m_chain_1} gives the full chain:
\begin{equation*}
    m(\sigma(\mathbf{x})) \sim m(\mathbf{y}) \sim m(\mathbf{z}) \sim m(\sigma_1(\mathbf{x})).
\end{equation*}

Finally, assuming $t \leq m(\sigma(\mathbf{x}))^{-2}$, the established equivalence $m(\sigma(\mathbf{x})) \sim m(\sigma_1(\mathbf{x}))$ immediately implies that
$
    t \leq C m(\sigma_1(\mathbf{x}))^{-2}
$,
which completes the proof.
\end{proof}

\begin{lemma}\label{lem:z_near_region_estimates_integrated}
Assume that $V \in {\rm{RH}}^{q}(dw)$ with $q>\max(1,\frac{\mathbf{N}}{2})$, $V \geq 0$, and let $0 < \beta < 2 - \frac{\mathbf{N}}{q}$. For a fixed $\sigma \in G$ and $\mathbf{x} \in \mathbb{R}^N$, define the local region $C_{\sigma,\mathbf{x}}^c = \{ \mathbf{y} \in D_\sigma(\mathbf{x}) : \|\sigma(\mathbf{x})-\mathbf{y}\| < m(\sigma(\mathbf{x}))^{-1} \}$ and the local region for $\mathbf{z}$ as $C_{\mathbf{y},\mathbf{z},1}^c = \{ \mathbf{z} \in \mathbb{R}^N : \|\mathbf{y}-\mathbf{z}\| < m(\mathbf{y})^{-1} \}$. Then there exists a constant $C>0$ such that for any $0 < t \leq m(\sigma(\mathbf{x}))^{-2}$, we have:
\begin{align}
    \int_{C_{\sigma,\mathbf{x}}^c} m(\sigma(\mathbf{x}))^{-\beta} \int_0^{t/2} \int_{C_{\mathbf{y},\mathbf{z},1}^c} |\partial_t h_{t-s}(\mathbf{y},\mathbf{z})| V(\mathbf{z}) k_s(\mathbf{z},\mathbf{x})\,dw(\mathbf{z})\,ds \,dw(\mathbf{y}) &\leq C t^{-1+\beta/2}, \label{eq:D1_bound_near} \\
    \int_{C_{\sigma,\mathbf{x}}^c} m(\sigma(\mathbf{x}))^{-\beta} \int_{t/2}^{t} \int_{C_{\mathbf{y},\mathbf{z},1}^c} h_{t-s}(\mathbf{y},\mathbf{z}) V(\mathbf{z}) |\partial_s k_s(\mathbf{z},\mathbf{x})|\,dw(\mathbf{z})\,ds \,dw(\mathbf{y}) &\leq C t^{-1+\beta/2}, \label{eq:D2_bound_near} \\
    \int_{C_{\sigma,\mathbf{x}}^c} m(\sigma(\mathbf{x}))^{-\beta} \int_{C_{\mathbf{y},\mathbf{z},1}^c} h_{t/2}(\mathbf{y},\mathbf{z}) V(\mathbf{z}) k_{t/2}(\mathbf{z},\mathbf{x})\,dw(\mathbf{z}) \,dw(\mathbf{y}) &\leq C t^{-1+\beta/2}. \label{eq:D3_bound_near}
\end{align}
\end{lemma}

\begin{proof}
Fix $\sigma \in G$ and $\mathbf{x} \in \mathbb{R}^N$. We decompose the space of integration for $\mathbf{z}$ into sectors $A_{\sigma_1}(\mathbf{x})$ for $\sigma_1 \in G$, where
\[
A_{\sigma_1}(\mathbf{x}) = \{ \mathbf{z} \in \mathbb{R}^N : d(\mathbf{x},\mathbf{z}) = \|\sigma_1(\mathbf{x})-\mathbf{z}\| \}.
\]
For any $\mathbf{y} \in C_{\sigma,\mathbf{x}}^c$ and $\mathbf{z} \in C_{\mathbf{y},\mathbf{z},1}^c \cap A_{\sigma_1}(\mathbf{x})$, Lemma~\ref{lem:geometric_m_equivalence} implies the full equivalence of critical radii:
\[
m(\sigma(\mathbf{x})) \sim m(\mathbf{y}) \sim m(\mathbf{z}) \sim m(\sigma_1(\mathbf{x})).
\]
Consequently, $m(\sigma(\mathbf{x}))^{-\beta} \leq C m(\sigma_1(\mathbf{x}))^{-\beta}$. Moreover, the assumption $t \leq m(\sigma(\mathbf{x}))^{-2}$ guarantees $t \leq C m(\sigma_1(\mathbf{x}))^{-2}$, so for any $s \in (0,t]$ we have $\sqrt{s} \leq C m(\sigma_1(\mathbf{x}))^{-1}$. This satisfies the hypothesis of Lemma~\ref{lem:integral_estimate_A} on $A_{\sigma_1}(\mathbf{x})$. Combined with $\beta < \gamma$, we have:
\begin{equation}\label{eq:gamma_to_beta_reduction_corrected}
    (\sqrt{s}m(\sigma_1(\mathbf{x})))^{\gamma} \leq C (\sqrt{s}m(\sigma_1(\mathbf{x})))^{\beta}.
\end{equation}

We bound the integral over each sector $A_{\sigma_1}(\mathbf{x})$ using Theorem~\ref{thm:time_derivatives} and Fubini's theorem.

\textit{Estimate for \eqref{eq:D1_bound_near}:}
Using the bound $k_s(\mathbf{z},\mathbf{x}) \leq C \mathcal{G}_{cs}(\mathbf{x},\mathbf{z})$ from Theorem~\ref{thm:time_derivatives}, we write:
\begin{align*}
    &\int_{C_{\sigma,\mathbf{x}}^c} m(\sigma(\mathbf{x}))^{-\beta} \int_0^{t/2} \int_{C_{\mathbf{y},\mathbf{z},1}^c \cap A_{\sigma_1}(\mathbf{x})} |\partial_t h_{t-s}(\mathbf{y},\mathbf{z})| V(\mathbf{z}) k_s(\mathbf{z},\mathbf{x})\,dw(\mathbf{z})\,ds \,dw(\mathbf{y}) \\
    &\leq C m(\sigma_1(\mathbf{x}))^{-\beta} \int_0^{t/2} \left( \int_{\mathbb{R}^N} |\partial_t h_{t-s}(\mathbf{y},\mathbf{z})| \,dw(\mathbf{y}) \right) \int_{A_{\sigma_1}(\mathbf{x})} V(\mathbf{z}) \mathcal{G}_{cs}(\mathbf{x},\mathbf{z})\,dw(\mathbf{z})\,ds.
\end{align*}
Since $\int_{\mathbb{R}^N} |\partial_t h_{t-s}(\mathbf{y},\mathbf{z})| \,dw(\mathbf{y}) \leq \frac{C}{t-s} \leq \frac{C}{t}$ for $s \in [0,t/2]$, applying Lemma~\ref{lem:integral_estimate_A} to the inner integral yields $\frac{C}{s}(\sqrt{s}m(\sigma_1(\mathbf{x})))^{\gamma}$. By \eqref{eq:gamma_to_beta_reduction_corrected}, we obtain:
\begin{align*}
    \frac{C}{t} m(\sigma_1(\mathbf{x}))^{-\beta} \int_0^{t/2} \frac{1}{s} (\sqrt{s}m(\sigma_1(\mathbf{x})))^{\beta} \, ds = \frac{C}{t} \int_0^{t/2} s^{-1+\beta/2} \, ds = C t^{-1+\beta/2}.
\end{align*}

\textit{Estimate of \eqref{eq:D2_bound_near}:}
By Theorem~\ref{thm:time_derivatives} we have $|\partial_s k_s(\mathbf{z},\mathbf{x})| \leq \frac{C}{s} \mathcal{G}_{cs}(\mathbf{x},\mathbf{z})$, so:
\begin{align*}
    &\int_{C_{\sigma,\mathbf{x}}^c} m(\sigma(\mathbf{x}))^{-\beta} \int_{t/2}^{t} \int_{C_{\mathbf{y},\mathbf{z},1}^c \cap A_{\sigma_1}(\mathbf{x})} h_{t-s}(\mathbf{y},\mathbf{z}) V(\mathbf{z}) |\partial_s k_s(\mathbf{z},\mathbf{x})|\,dw(\mathbf{z})\,ds \,dw(\mathbf{y}) \\
    &\leq C m(\sigma_1(\mathbf{x}))^{-\beta} \int_{t/2}^{t} \frac{1}{s} (1) \left( \int_{A_{\sigma_1}(\mathbf{x})} V(\mathbf{z}) \mathcal{G}_{cs}(\mathbf{x},\mathbf{z})\,dw(\mathbf{z}) \right) ds.
\end{align*}
Applying Lemma~\ref{lem:integral_estimate_A} and \eqref{eq:gamma_to_beta_reduction_corrected}, the integral is bounded by:
\begin{align*}
    C m(\sigma_1(\mathbf{x}))^{-\beta} \int_{t/2}^{t} \frac{1}{s^2} (\sqrt{s}m(\sigma_1(\mathbf{x})))^{\beta} \, ds = C \int_{t/2}^{t} s^{-2+\beta/2} \, ds \leq C t^{-1+\beta/2}.
\end{align*}

\textit{Estimate of \eqref{eq:D3_bound_near}:}
Using $k_{t/2}(\mathbf{z},\mathbf{x}) \leq C \mathcal{G}_{ct}(\mathbf{x},\mathbf{z})$ and $\int_{\mathbb{R}^N} h_{t/2}(\mathbf{y},\mathbf{z}) \,dw(\mathbf{y}) = 1$, Lemma~\ref{lem:integral_estimate_A} evaluated at $s=t/2$ together with \eqref{eq:gamma_to_beta_reduction_corrected} yields:
\begin{align*}
    &\int_{C_{\sigma,\mathbf{x}}^c} m(\sigma(\mathbf{x}))^{-\beta} \int_{C_{\mathbf{y},\mathbf{z},1}^c \cap A_{\sigma_1}(\mathbf{x})} h_{t/2}(\mathbf{y},\mathbf{z}) V(\mathbf{z}) k_{t/2}(\mathbf{z},\mathbf{x})\,dw(\mathbf{z}) \,dw(\mathbf{y}) \\
    &\leq C m(\sigma_1(\mathbf{x}))^{-\beta} (1) \left( \frac{C}{t} (\sqrt{t}m(\sigma_1(\mathbf{x})))^{\gamma} \right) \leq \frac{C}{t} m(\sigma_1(\mathbf{x}))^{-\beta} (\sqrt{t}m(\sigma_1(\mathbf{x})))^{\beta} = C t^{-1+\beta/2}.
\end{align*}

Summing these bounds over all domains $A_{\sigma_1}(\mathbf{x})$ for the finite group $G$ completes the proof.
\end{proof}

\begin{proof}[Proof of Theorem~\ref{thm:main_derivative_difference}]
Let $\mathbf{x} \in \mathbb{R}^N$ and $t>0$. We decompose $\mathbb{R}^N$ into orbit sectors $D_\sigma(\mathbf{x}) = \{ \mathbf{y} \in \mathbb{R}^N : d(\mathbf{x},\mathbf{y}) = \|\sigma(\mathbf{x})-\mathbf{y}\| \}$ for $\sigma \in G$:
\begin{align*}
    |\partial_t H_t f(\mathbf{x}) - \partial_t K_t f(\mathbf{x})| &\leq \sum_{\sigma \in G} \int_{D_\sigma(\mathbf{x})} |\partial_t(h_t(\mathbf{x},\mathbf{y}) - k_t(\mathbf{x},\mathbf{y}))| |f(\mathbf{y})| \, dw(\mathbf{y}) =: \sum_{\sigma \in G} I_\sigma.
\end{align*}

Fix $\sigma \in G$. We estimate the term $I_\sigma$ by splitting the analysis into two cases.

\textbf{Case 1:} $t \geq m(\sigma(\mathbf{x}))^{-2}$.
Applying Lemma \ref{lem:large_time_derivative_bound} to the integral over $D_\sigma(\mathbf{x})$, we directly obtain:
\begin{equation*}
    I_\sigma \leq C t^{-1+\beta/2} \|f m^\beta\|_{L^\infty}.
\end{equation*}

\textbf{Case 2:} $t < m(\sigma(\mathbf{x}))^{-2}$.
We split $D_\sigma(\mathbf{x})$ into the far region $C_{\sigma,\mathbf{x}} = \{ \mathbf{y} \in D_\sigma(\mathbf{x}) : \|\sigma(\mathbf{x})-\mathbf{y}\| \geq m(\sigma(\mathbf{x}))^{-1} \}$ and the local region $C_{\sigma,\mathbf{x}}^c = \{ \mathbf{y} \in D_\sigma(\mathbf{x}) : \|\sigma(\mathbf{x})-\mathbf{y}\| < m(\sigma(\mathbf{x}))^{-1} \}$, writing $I_\sigma = I_{1,\sigma} + I_{2,\sigma}$, where:
\begin{align*}
    I_{1,\sigma} &= \int_{C_{\sigma,\mathbf{x}}} |\partial_t h_t(\mathbf{x},\mathbf{y}) - \partial_t k_t(\mathbf{x},\mathbf{y})| |f(\mathbf{y})| \, dw(\mathbf{y}), 
    I_{2,\sigma} = \int_{C_{\sigma,\mathbf{x}}^c} |\partial_t (h_t(\mathbf{x},\mathbf{y}) - k_t(\mathbf{x},\mathbf{y}))| |f(\mathbf{y})| \, dw(\mathbf{y}).
\end{align*}

For the integral $I_{1,\sigma}$ over the far region $C_{\sigma,\mathbf{x}}$, applying Lemma \ref{lem:far_region_estimate} yields:
\begin{equation*}
    I_{1,\sigma} \leq C t^{-1+\beta/2} \|f m^\beta\|_{L^\infty(\mathbb{R}^N)}.
\end{equation*}

To estimate the integral $I_{2,\sigma}$ over the local region $C_{\sigma,\mathbf{x}}^c$, we use the assumption $\|f m^\beta\|_{L^\infty(\mathbb{R}^N)} < \infty$, which gives $|f(\mathbf{y})| \leq \|f m^\beta\|_{L^\infty(\mathbb{R}^N)} m(\mathbf{y})^{-\beta}$. By Lemma \ref{lem:critical_radius_comparability}, for $\mathbf{y} \in C_{\sigma,\mathbf{x}}^c$, we have $m(\mathbf{y}) \sim m(\sigma(\mathbf{x}))$, giving:
\begin{equation*}
    |f(\mathbf{y})| \leq C \|f m^\beta\|_{L^\infty(\mathbb{R}^N)} m(\sigma(\mathbf{x}))^{-\beta}.
\end{equation*}

Applying the point-wise Duhamel identity from Lemma \ref{lem:kernel_duhamel_derivative} to the kernel derivative within $I_{2,\sigma}$, we introduce the majorant $I_{2,\sigma} \leq C \|f m^\beta\|_{L^\infty(\mathbb{R}^N)} (J_{1,\sigma} + J_{2,\sigma} + J_{3,\sigma})$, where:
\begin{align*}
    J_{1,\sigma} &= \int_{C_{\sigma,\mathbf{x}}^c} m(\sigma(\mathbf{x}))^{-\beta} \int_0^{t/2} \int_{\mathbb{R}^N} |\partial_t h_{t-s}(\mathbf{y},\mathbf{z})| V(\mathbf{z}) k_s(\mathbf{z},\mathbf{x})\,dw(\mathbf{z})\,ds \,dw(\mathbf{y}), \\
    J_{2,\sigma} &= \int_{C_{\sigma,\mathbf{x}}^c} m(\sigma(\mathbf{x}))^{-\beta} \int_{t/2}^{t} \int_{\mathbb{R}^N} h_{t-s}(\mathbf{y},\mathbf{z}) V(\mathbf{z}) |\partial_s k_s(\mathbf{z},\mathbf{x})|\,dw(\mathbf{z})\,ds \,dw(\mathbf{y}), \\
    J_{3,\sigma} &= \int_{C_{\sigma,\mathbf{x}}^c} m(\sigma(\mathbf{x}))^{-\beta} \int_{\mathbb{R}^N} h_{t/2}(\mathbf{y},\mathbf{z}) V(\mathbf{z}) k_{t/2}(\mathbf{z},\mathbf{x})\,dw(\mathbf{z}) \,dw(\mathbf{y}).
\end{align*}

For each term $J_{1,\sigma}, J_{2,\sigma}, J_{3,\sigma}$, we split the inner domain of integration with respect to $\mathbf{z} \in \mathbb{R}^N$ into $C_{\mathbf{y},\mathbf{z},1} = \{ \mathbf{z} \in \mathbb{R}^N : \|\mathbf{y}-\mathbf{z}\| \geq m(\mathbf{y})^{-1} \}$ and its complement $C_{\mathbf{y},\mathbf{z},1}^c = \{ \mathbf{z} \in \mathbb{R}^N : \|\mathbf{y}-\mathbf{z}\| < m(\mathbf{y})^{-1} \}$.

The components evaluated over the region $C_{\mathbf{y},\mathbf{z},1}$ are bounded by $C t^{-1+\beta/2}$ due to Lemma \ref{lem:z_far_region_estimates_integrated}.

For the components evaluated over the local region $C_{\mathbf{y},\mathbf{z},1}^c$, since $t < m(\sigma(\mathbf{x}))^{-2}$, Lemma \ref{lem:z_near_region_estimates_integrated} applies directly for this fixed $\sigma$, yielding $C t^{-1+\beta/2}$.

Combining these bounds gives $J_{1,\sigma} + J_{2,\sigma} + J_{3,\sigma} \leq C t^{-1+\beta/2}$, which implies:
\begin{equation*}
    I_{2,\sigma} \leq C t^{-1+\beta/2} \|f m^\beta\|_{L^\infty(\mathbb{R}^N)}.
\end{equation*}

Thus, in Case 2, we also have $I_\sigma = I_{1,\sigma} + I_{2,\sigma} \leq C t^{-1+\beta/2} \|f m^\beta\|_{L^\infty(\mathbb{R}^N)}$. Since for every $\sigma \in G$ the estimate $I_\sigma \leq C t^{-1+\beta/2} \|f m^\beta\|_{L^\infty(\mathbb{R}^N)}$ holds uniformly, summing over the finite group $G$ finishes the proof:
\begin{equation*}
    |\partial_t H_t f(\mathbf{x}) - \partial_t K_t f(\mathbf{x})| \leq \sum_{\sigma \in G} I_\sigma \leq C t^{-1+\beta/2} \|f m^\beta\|_{L^\infty(\mathbb{R}^N)}.
\end{equation*}
\end{proof}

\begin{proof}[Proof of Theorem~\ref{thm:dunkl_schrodinger_equivalence}]
Since $0 < \beta < 2$, we have $n = 1$. Set $[f]_{\mathrm{Zyg},\beta} := \sup_{\mathbf{z} \neq 0} |\mathbf{z}|^{-\beta} \| f(\cdot + \mathbf{z}) + f(\cdot - \mathbf{z}) - 2f(\cdot) \|_{L^\infty}$.

By Lemma~\ref{lem:m_growth}, $(1+\|\mathbf{x}\|)^{-\beta} \le C m(\mathbf{x})^\beta$, giving
\begin{equation}\label{eq:weight_bound}
    \| f(\cdot)(1+\|\cdot\|)^{-\beta}\|_{L^\infty} \le C \| f m^\beta \|_{L^\infty}.
\end{equation}
By Theorem~\ref{thm:main_derivative_difference}, we have
\begin{equation}\label{eq:semigroup_comp}
    \sup_{t>0} t^{1 - \beta/2} \left\| \partial_t K_t f - \partial_t H_t f \right\|_{L^\infty} \le C \|f m^\beta\|_{L^\infty}.
\end{equation}
Theorem~\ref{teo:equivalence} (and Remark~\ref{rem:zygmund_equivalence} for $0 < \beta < 1$) yields $\|f\|_{\widetilde{\Lambda}_k^{\beta/2}} \approx \| f(\cdot)(1+\|\cdot\|)^{-\beta}\|_{L^\infty} + [f]_{\mathrm{Zyg},\beta}$. 

Combining this equivalence with \eqref{eq:weight_bound} and applying the triangle inequality to \eqref{eq:semigroup_comp}, we obtain
\[
    \sup_{t>0} t^{1 - \beta/2} \|\partial_t K_t f\|_{L^\infty} \le \sup_{t>0} t^{1 - \beta/2} \|\partial_t H_t f\|_{L^\infty} + C \|f m^\beta\|_{L^\infty} \le C' \big( \|f m^\beta\|_{L^\infty} + [f]_{\mathrm{Zyg},\beta} \big),
\]
which proves $\|f\|_{\widetilde{\Lambda}^{\beta/2}_{\mathcal{L},k}} \le C \|f\|_{\Lambda^{\beta}_{\mathcal{L},k}}$.

The reverse inequality $\|f\|_{\Lambda^{\beta}_{\mathcal{L},k}} \le C \|f\|_{\widetilde{\Lambda}^{\beta/2}_{\mathcal{L},k}}$ follows symmetrically by bounding $[f]_{\mathrm{Zyg},\beta}$ via $\|f\|_{\widetilde{\Lambda}_k^{\beta/2}}$, replacing $\partial_t H_t f$ with $\partial_t K_t f$ using \eqref{eq:semigroup_comp}, and controlling the weight via \eqref{eq:weight_bound}.
\end{proof}

\section{Proof of Theorem~\ref{thm:time_derivatives}}\label{sec:appendix_}

\begin{lemma}\label{lem:davies_gaffney}
Let $U_1$ and $U_2$ be measurable subsets of $\mathbb{R}^N$ and let $d(U_1,U_2) = \inf_{\mathbf{x} \in U_1, \mathbf{y} \in U_2} d(\mathbf{x},\mathbf{y})$ denote the orbit distance between $U_1$ and $U_2$. Then there exist constants $C', c' > 0$ such that for all $t > 0$ and for all functions $f_1 \in L^2(U_1, dw)$, $f_2 \in L^2(U_2, dw)$, the Dunkl-Schr\"odinger semigroup $K_t = e^{-tL}$ satisfies the following Davies-Gaffney estimate:
\begin{equation}
    |\langle K_t f_1, f_2 \rangle| \leq C' \exp\left(- \frac{c'\,d(U_1,U_2)^2}{t} \right) \|f_1\|_{L^2(dw)} \|f_2\|_{L^2(dw)}.
\end{equation}
\end{lemma}

\begin{proof}
Let $t>0$. By the definition of the semigroup $K_t$ given in \eqref{eq:K_semigroup}, we can write the inner product as:
\begin{equation*}
    \langle K_t f_1, f_2 \rangle = \int_{U_1} \int_{U_2} k_t(\mathbf{x},\mathbf{y}) f_1(\mathbf{x}) \overline{f_2(\mathbf{y})} \, dw(\mathbf{y}) \, dw(\mathbf{x}).
\end{equation*}
By the domination property \eqref{eq:kernels_compare}, we have $0 \leq k_t(\mathbf{x},\mathbf{y}) \leq h_t(\mathbf{x},\mathbf{y})$. Taking the absolute value and applying the estimate \eqref{eq:heat2} from Theorem \ref{teo:heat_new}, we obtain:
\begin{align}
    | \langle K_t f_1, f_2 \rangle | &\leq \int_{U_1} \int_{U_2} h_t(\mathbf{x},\mathbf{y}) |f_1(\mathbf{x})| |f_2(\mathbf{y})| \, dw(\mathbf{y}) \, dw(\mathbf{x}) \nonumber \\
    &\leq C \int_{U_1} \int_{U_2} \mathcal{G}_{t/c}(\mathbf{x},\mathbf{y}) |f_1(\mathbf{x})| |f_2(\mathbf{y})| \, dw(\mathbf{y}) \, dw(\mathbf{x}), \label{eq:dg_proof_1}
\end{align}
where $\mathcal{G}_{t/c}$ is defined in \eqref{eq:mathcal_G}.

Recall that $\mathbf{x} \in U_1$ and $\mathbf{y} \in U_2$ implies $d(\mathbf{x},\mathbf{y}) \geq d(U_1,U_2)$. We can split the exponential factor in $\mathcal{G}_{t/c}(\mathbf{x},\mathbf{y})$ as follows:
\begin{align*}
    \exp\left(-\frac{c\,d(\mathbf{x},\mathbf{y})^2}{t}\right) &= \exp\left(-\frac{c\,d(\mathbf{x},\mathbf{y})^2}{2t}\right) \exp\left(-\frac{c\,d(\mathbf{x},\mathbf{y})^2}{2t}\right) \\
    &\leq \exp\left(-\frac{c\,d(U_1,U_2)^2}{2t}\right) \exp\left(-\frac{(c/2)\,d(\mathbf{x},\mathbf{y})^2}{t}\right).
\end{align*}
Using the doubling property of the measure $dw$ given in \eqref{eq:doubling} and \eqref{eq:growth}, the volume factor satisfies:
\begin{equation*}
    \Big(\max (w(B(\mathbf{x},\sqrt{t/c})),w(B(\mathbf{y}, \sqrt{t/c})))\Big)^{-1} \leq \tilde{C} \Big(\max (w(B(\mathbf{x},\sqrt{2t/c})),w(B(\mathbf{y}, \sqrt{2t/c})))\Big)^{-1}
\end{equation*}
for some constant $\tilde{C} > 0$. Therefore, we can bound $\mathcal{G}_{t/c}$ by:
\begin{equation}\label{eq:G_split}
    \mathcal{G}_{t/c}(\mathbf{x},\mathbf{y}) \leq \tilde{C} \exp\left(-\frac{c\,d(U_1,U_2)^2}{2t}\right) \mathcal{G}_{2t/c}(\mathbf{x},\mathbf{y}).
\end{equation}
Substituting \eqref{eq:G_split} into \eqref{eq:dg_proof_1}, we obtain:
\begin{equation}\label{eq:dg_proof_2}
    | \langle K_t f_1, f_2 \rangle | \leq C \tilde{C} \exp\left(-\frac{c\,d(U_1,U_2)^2}{2t}\right) \int_{U_1} \int_{U_2} \mathcal{G}_{2t/c}(\mathbf{x},\mathbf{y}) |f_1(\mathbf{x})| |f_2(\mathbf{y})| \, dw(\mathbf{y}) \, dw(\mathbf{x}).
\end{equation}

By Lemma \ref{lem:homogeneous}, there exists a constant $M > 0$ (independent of $t$ and $\mathbf{x}$) such that
\begin{equation*}
    \sup_{\mathbf{x} \in \mathbb{R}^N} \int_{\mathbb{R}^N} \mathcal{G}_{2t/c}(\mathbf{x},\mathbf{y}) \, dw(\mathbf{y}) \leq M \quad \text{and} \quad \sup_{\mathbf{y} \in \mathbb{R}^N} \int_{\mathbb{R}^N} \mathcal{G}_{2t/c}(\mathbf{x},\mathbf{y}) \, dw(\mathbf{x}) \leq M.
\end{equation*}
By Schur's test, the integral operator with kernel $\mathcal{G}_{2t/c}(\mathbf{x},\mathbf{y})$ is bounded on $L^2(dw)$ with operator norm at most $M$. Applying this to \eqref{eq:dg_proof_2} yields:
\begin{equation*}
    \int_{U_1} \int_{U_2} \mathcal{G}_{2t/c}(\mathbf{x},\mathbf{y}) |f_1(\mathbf{x})| |f_2(\mathbf{y})| \, dw(\mathbf{y}) \, dw(\mathbf{x}) \leq M \|f_1\|_{L^2(dw)} \|f_2\|_{L^2(dw)}.
\end{equation*}
Consequently, we conclude that
\begin{equation*}
    |\langle K_t f_1, f_2 \rangle| \leq C' \exp\left(- \frac{c'\,d(U_1,U_2)^2}{t} \right) \|f_1\|_{L^2(dw)} \|f_2\|_{L^2(dw)},
\end{equation*}
where $C' = C \tilde{C} M$ and $c' = c/2$. This completes the proof.
\end{proof}

\begin{lemma}\label{lem:on_diagonal_bound}
There exists a constant $C > 0$ such that for all $\mathbf{x} \in \mathbb{R}^N$ and $t > 0$, the Dunkl-Schr\"odinger heat kernel satisfies
\begin{equation}\label{eq:on_diagonal_bound}
    k_t(\mathbf{x},\mathbf{x}) \leq \frac{C}{w(B(\mathbf{x},\sqrt{t}))}.
\end{equation}
\end{lemma}

\begin{proof}
By the domination property \eqref{eq:kernels_compare}, we have $k_t(\mathbf{x},\mathbf{x}) \leq h_t(\mathbf{x},\mathbf{x})$. Applying Theorem \ref{teo:heat_new} with $m=0$ and $|\boldsymbol{\beta}| = |\boldsymbol{\beta}'| = 0$, and noting that $d(\mathbf{x},\mathbf{x})=0$, we obtain
$$k_t(\mathbf{x},\mathbf{x}) \leq C \mathcal{G}_{t/c}(\mathbf{x},\mathbf{x}) = \frac{C}{w(B(\mathbf{x},\sqrt{t/c}))}.$$
The desired estimate \eqref{eq:on_diagonal_bound} then follows directly from the doubling property of the measure $dw$ given in \eqref{eq:growth}.
\end{proof}

\begin{lemma}\label{lem:complex_time_bound}
Let $z \in \mathbb{C}_+$ with $t = \mathrm{Re}\, z > 0$ and $s = \mathrm{Im}\, z \in \mathbb{R}$. The complex-time heat kernel $k_z(\mathbf{x},\mathbf{y})$ satisfies the pointwise bound:
\begin{equation}\label{eq:complex_kernel_pointwise}
    |k_z(\mathbf{x},\mathbf{y})| \leq \frac{C}{\sqrt{w(B(\mathbf{x},\sqrt{\mathrm{Re}\, z}))}\sqrt{w(B(\mathbf{y},\sqrt{\mathrm{Re}\, z}))}} \quad \text{for all } \mathbf{x}, \mathbf{y} \in \mathbb{R}^N.
\end{equation}
\end{lemma}

\begin{proof}
Let $z = t + is \in \mathbb{C}_+$, where $t = \mathrm{Re}\, z > 0$ and $s = \mathrm{Im}\, z \in \mathbb{R}$. By the semigroup property, the complex-time heat operator factorizes as:
$$e^{-zL} = e^{-tL/2} e^{-isL} e^{-tL/2}.$$
Evaluating the action of $e^{-zL}$ at points $(\mathbf{x},\mathbf{y})$ corresponds to the $L^2(w)$-inner product of the real-time heat kernels:
$$k_{t+is}(\mathbf{x},\mathbf{y}) = \left\langle e^{-isL} k_{t/2}(\cdot, \mathbf{y}), k_{t/2}(\mathbf{x}, \cdot) \right\rangle_{L^2(w)}.$$
Applying the Cauchy-Schwarz inequality in $L^2(w)$ and using the fact that $L$ is self-adjoint non-negative (hence $e^{-isL}$ is a unitary operator on $L^2(w)$ with $\|e^{-isL}\|_{L^2(w) \to L^2(w)} = 1$), we obtain:
\begin{align*}
    |k_{t+is}(\mathbf{x},\mathbf{y})| &\leq \left\| e^{-isL} k_{t/2}(\cdot, \mathbf{y}) \right\|_{L^2(w)} \left\| k_{t/2}(\mathbf{x}, \cdot) \right\|_{L^2(w)} \\
    &= \left\| k_{t/2}(\cdot, \mathbf{y}) \right\|_{L^2(w)} \left\| k_{t/2}(\mathbf{x}, \cdot) \right\|_{L^2(w)}.
\end{align*}
By the self-adjointness and symmetry of the real-time heat kernel $k_{t/2}$, the $L^2(w)$-norm equals the diagonal kernel at time $t$:
$$\left\| k_{t/2}(\mathbf{x}, \cdot) \right\|_{L^2(w)}^2 = \int_{\mathbb{R}^N} |k_{t/2}(\mathbf{x},\mathbf{z})|^2 \, dw(\mathbf{z}) = k_t(\mathbf{x},\mathbf{x}).$$
Similarly, $\left\| k_{t/2}(\cdot, \mathbf{y}) \right\|_{L^2(w)}^2 = k_t(\mathbf{y},\mathbf{y})$. Combining these with the standard diagonal bound $k_t(\mathbf{x},\mathbf{x}) \leq \frac{C}{w(B(\mathbf{x},\sqrt{t}))}$, we conclude:
$$|k_{t+is}(\mathbf{x},\mathbf{y})| \leq \sqrt{k_t(\mathbf{x},\mathbf{x})} \sqrt{k_t(\mathbf{y},\mathbf{y})} \leq \frac{C}{\sqrt{w(B(\mathbf{x},\sqrt{t}))}\sqrt{w(B(\mathbf{y},\sqrt{t}))}}.$$
This establishes \eqref{eq:complex_kernel_pointwise}.
\end{proof}


\begin{theorem}[Phragm\'en-Lindel\"of]\label{thm:phragmen_lindelof_21}
Let $S$ be the open region in $\mathbb{C}$ bounded by two rays meeting at an angle $\pi/\alpha$, for some $\alpha > 1/2$. Suppose that $F$ is analytic on $S$, continuous on $\overline{S}$, and satisfies
\begin{equation}\label{eq:pl_growth}
    |F(z)| \leq C \exp\left(c |z|^\beta\right)
\end{equation}
for some $\beta \in [0, \alpha)$ and for all $z \in S$. Then the condition $|F(z)| \leq B$ on the two bounding rays implies $|F(z)| \leq B$ for all $z \in S$.
\end{theorem}





\begin{theorem}\label{thm:pointwise_complex_kernel_sector}
Let $L$ be a non-negative self-adjoint operator on $L^2(\mathbb{R}^N, dw)$, where $w$ satisfies the doubling condition \eqref{eq:growth}. Let $k_z(\mathbf{x},\mathbf{y})$ be the kernel of $e^{-zL}$. Fix $\theta \in (0, \pi/2)$ and define the sector $S_\theta = \{ z \in \mathbb{C}_+ : |\arg z| \le \theta \}$.

For any $\mathbf{x}, \mathbf{y} \in \mathbb{R}^N$ such that $d(\mathbf{x},\mathbf{y}) > 0$ and $z \in S_\theta$, there exist constants $C^*, c'' > 0$ (depending only on $\theta$ and the constants of the space) such that:
\begin{equation}\label{eq:pointwise_kernel_sector}
    |k_z(\mathbf{x}, \mathbf{y})| \leq \frac{C^*}{w(B(\mathbf{x}, d(\mathbf{x},\mathbf{y})))} \exp\left(- c''\,\mathrm{Re}\,\frac{d(\mathbf{x},\mathbf{y})^2}{z}\right).
\end{equation}
\end{theorem}

\begin{proof}
Fix $\mathbf{x}, \mathbf{y} \in \mathbb{R}^N$ with orbit distance $d = d(\mathbf{x},\mathbf{y}) > 0$. Define the normalized function $F: \mathbb{C}_+ \to \mathbb{C}$ by
$$ F(z) = k_z(\mathbf{x}, \mathbf{y}) \cdot w(B(\mathbf{x}, d)). $$

\textbf{Step 1: Global a priori bound on $\mathbb{C}_+$}

By Lemma~\ref{lem:complex_time_bound}, for any $z \in \mathbb{C}_+$ we have:
$$ |k_z(\mathbf{x},\mathbf{y})| \leq \frac{C}{\sqrt{w(B(\mathbf{x},\sqrt{\mathrm{Re}\, z}))}\sqrt{w(B(\mathbf{y},\sqrt{\mathrm{Re}\, z}))}}. $$
Let $\sigma_0 \in G$ be an element of the Weyl group such that $\|\sigma_0(\mathbf{x}) - \mathbf{y}\| = d$. Due to the $G$-invariance of the measure $dw$, we have $w(B(\mathbf{x},d)) = w(B(\sigma_0(\mathbf{x}),d))$. The Euclidean inclusion $B(\sigma_0(\mathbf{x}),d) \subset B(\mathbf{y},2d)$, combined with the doubling condition centered at $\mathbf{y}$ and at $\mathbf{x}$, yields:
\begin{equation}\label{eq:F_apriori_global_strict}
    |F(z)| \le C_1 \left(1 + \frac{d^2}{\mathrm{Re}\,z}\right)^{\mathbf{N}/2} \quad \text{for all } z \in \mathbb{C}_+,
\end{equation}
where $C_1 > 0$  depend only on the doubling constants of $w$.

\textbf{Step 2: Real-time Gaussian estimate and conformal mapping}

From the real-time heat kernel bound (derived from Lemma~\ref{lem:on_diagonal_bound} combined with Davies-Gaffney estimates from Lemma~\ref{lem:davies_gaffney}), there exists a constant $c > 0$ such that for any $t > 0$:
$$ |k_t(\mathbf{x},\mathbf{y})| \le \frac{C_0}{w(B(\mathbf{x},d))} \left(1 + \frac{d^2}{t}\right)^{\mathbf{N}/2} \exp\left(-\frac{c\,d^2}{t}\right). $$
Set $\gamma = c d^2$ and consider the region $\mathcal{C}_\gamma = \{ z \in \mathbb{C}_+ : \mathrm{Re}\left(\frac{\gamma}{z}\right) \ge 1 \}$. 
The inversion $\zeta = \frac{\gamma}{z}$ maps $\mathcal{C}_\gamma$ bijectively onto the half-plane $\Omega = \{ \zeta \in \mathbb{C} : \mathrm{Re}\,\zeta \ge 1 \}$. 
Define $u(\zeta) = F\left(\frac{\gamma}{\zeta}\right)$ for $\zeta \in \Omega$. 

On the real axis $\zeta = \xi \ge 1$ (corresponding to real times $t = \frac{\gamma}{\xi} = \frac{c d^2}{\xi}$), we observe that $\frac{d^2}{t} = \frac{\xi}{c}$. Thus:
$$ |u(\xi) e^\xi| = \left|F\left(\frac{c d^2}{\xi}\right)\right| e^\xi \le C_0 \left(1 + \frac{\xi}{c}\right)^{\mathbf{N}/2}. $$

\textbf{Step 3: Boundary estimates for the auxiliary function $v(\zeta)$}

Define the auxiliary function $v(\zeta)$ directly on $\Omega$ by:
$$ v(\zeta) = (2\zeta)^{-\mathbf{N}} u(\zeta) e^\zeta. $$
We estimate $|v(\zeta)|$ on the boundaries of the upper and lower quarter-planes $\Omega^\pm = \{ \zeta \in \mathbb{C} : \mathrm{Re}\,\zeta \ge 1, \pm \mathrm{Im}\,\zeta \ge 0 \}$:

\begin{enumerate}
    \item \textbf{On the horizontal ray $\{\zeta = \xi \in \mathbb{R} : \xi \ge 1\}$:}
    Using $1 + \frac{\xi}{c} \le \left(1 + \frac{1}{c}\right)\xi$ for $\xi \ge 1$:
    $$ |v(\xi)| = (2\xi)^{-\mathbf{N}} |u(\xi) e^\xi| \le 2^{-\mathbf{N}} \xi^{-\mathbf{N}} \cdot C_0 \left(1 + \frac{1}{c}\right)^{\mathbf{N}/2} \xi^{\mathbf{N}/2} \le 2^{-\mathbf{N}} \left(1 + \frac{1}{c}\right)^{\mathbf{N}/2} C_0 =: M_{\text{real}}. $$
    Notice that the remaining $\xi^{-\mathbf{N}/2} \le 1$, so $v$ is strictly bounded on the real ray.

    \item \textbf{On the vertical line $\mathrm{Re}\,\zeta = 1$ ($\zeta = 1+is, s \in \mathbb{R}$):}
    The points $z = \frac{\gamma}{1+is}$ have real part $\mathrm{Re}\,z = \frac{\gamma}{1+s^2} = \frac{c d^2}{1+s^2}$. By the global bound \eqref{eq:F_apriori_global_strict}:
    $$ |u(1+is)| = \left|F\left(\frac{\gamma}{1+is}\right)\right| \le C_1 \left(1 + \frac{1+s^2}{c}\right)^{\mathbf{N}/2} \le C_1 \left(1+\frac{1}{c}\right)^{\mathbf{N}/2} (1+s^2)^{\mathbf{N}/2} =: C_2 (1+s^2)^{\mathbf{N}/2}. $$
    Evaluating $|v(1+is)|$:
    $$ |v(1+is)| = |2(1+is)|^{-\mathbf{N}} |u(1+is)| |e^{1+is}| \le 2^{-\mathbf{N}} (1+s^2)^{-\mathbf{N}/2} \cdot C_2 (1+s^2)^{\mathbf{N}/2} \cdot e = e 2^{-\mathbf{N}} C_2 =: M_{\text{vert}}. $$
\end{enumerate}

\textbf{Step 4: Phragmén-Lindelöf on quarter-planes}

Set $M_0 = \max(M_{\text{real}}, M_{\text{vert}})$. Then $|v(\zeta)| \le M_0$ on the entire boundary $\partial \Omega^+$ and $\partial \Omega^-$. 

Inside $\Omega^+$, $v(\zeta)$ satisfies the exponential growth bound $|v(\zeta)| \le C e^{|\zeta|}$ (growth exponent $\beta = 1$). Since $\Omega^+$ is a sector of opening angle $\pi/2$, the critical Phragmén-Lindelöf exponent is $\alpha = \frac{\pi}{\pi/2} = 2$. 

Because $\beta = 1 < \alpha = 2$, the classical Phragmén-Lindelöf theorem applies independently on $\Omega^+$ and $\Omega^-$, yielding:
$$ |v(\zeta)| \le M_0 \quad \text{for all } \zeta \in \Omega. $$

\textbf{Step 5: Sector geometry and polynomial absorption}

Unfolding $v(\zeta)$ gives for all $z \in \mathcal{C}_\gamma$:
$$ |F(z)| \le M_0 \left|\frac{2\gamma}{z}\right|^{\mathbf{N}} \exp\left(-\mathrm{Re}\frac{\gamma}{z}\right). $$
Now restrict to $z \in S_\theta \cap \mathcal{C}_\gamma$. For $z \in S_\theta$, we have $|\arg z| \le \theta < \pi/2$, which implies $|\arg(\gamma/z)| = |\arg z| \le \theta$. In this sector, the modulus is controlled by the real part:
$$ \left|\frac{\gamma}{z}\right| = \frac{\mathrm{Re}(\gamma/z)}{\cos(\arg z)} \le \frac{1}{\cos\theta} \mathrm{Re}\left(\frac{\gamma}{z}\right). $$
Letting $t = \mathrm{Re}\left(\frac{\gamma}{z}\right) \ge 1$, we absorb the polynomial factor using half of the exponential decay:
$$ \left|\frac{2\gamma}{z}\right|^{\mathbf{N}} \exp\left(-\frac{1}{2}\mathrm{Re}\frac{\gamma}{z}\right) \le \sup_{t \ge 1} \left(\frac{2t}{\cos\theta}\right)^{\mathbf{N}} e^{-t/2} =: C_{\text{poly}}(\theta) < \infty. $$
Thus, for all $z \in S_\theta \cap \mathcal{C}_\gamma$:
$$ |F(z)| \le M_0 C_{\text{poly}}(\theta) \exp\left(-\frac{1}{2}\mathrm{Re}\frac{\gamma}{z}\right) = M_0 C_{\text{poly}}(\theta) \exp\left(-\frac{c}{2}\mathrm{Re}\frac{d^2}{z}\right). $$

For $z \in S_\theta \setminus \mathcal{C}_\gamma$, we have $\mathrm{Re}\left(\frac{\gamma}{z}\right) < 1$. Because $z \in S_\theta$, this geometric constraint implies $\frac{\gamma}{|z|} < \frac{1}{\cos\theta}$, hence $\frac{\gamma}{\mathrm{Re}\,z} \le \frac{\gamma}{|z|\cos\theta} < \frac{1}{\cos^2\theta}$, meaning that $\frac{d^2}{\mathrm{Re}\,z} < \frac{1}{c\cos^2\theta}$. On this set, $F(z)$ is uniformly bounded by \eqref{eq:F_apriori_global_strict}:
$$ |F(z)| \le C_1 \left(1 + \frac{1}{c\cos^2\theta}\right)^{\mathbf{N}/2} =: M_2. $$
Since $\mathrm{Re}\left(\frac{\gamma}{z}\right) < 1$, multiplying by $-1/2$ and exponentiating gives $e^{-1/2} < \exp\left(-\frac{1}{2}\mathrm{Re}\frac{\gamma}{z}\right)$. We use this to balance the constant $M_2$:
$$ |F(z)| \le M_2 \le M_2 e^{1/2} \exp\left(-\frac{1}{2}\mathrm{Re}\frac{\gamma}{z}\right) = M_2 e^{1/2} \exp\left(-\frac{c}{2}\mathrm{Re}\frac{d^2}{z}\right). $$

Setting $C^* = \max(M_0 C_{\text{poly}}(\theta), M_2 e^{1/2})$ and $c'' = c/2$, we obtain for all $z \in S_\theta$:
$$ |F(z)| \le C^* \exp\left(-c'' \mathrm{Re}\frac{d^2}{z}\right). $$
Dividing both sides by $w(B(\mathbf{x}, d))$ completes the proof.
\end{proof}

\begin{proof}[Proof of Theorem~\ref{thm:time_derivatives}]
Fix $t > 0$ and choose $\varepsilon \in (0, \sin\theta)$ such that the circle $\gamma_t = \{ z \in \mathbb{C} : |z - t| = \varepsilon t \}$ lies entirely within the sector $S_\theta$. 

For any $z \in \gamma_t$, we have $\mathrm{Re}\,z \approx t$ and $\mathrm{Re}(1/z) \ge \frac{c_0}{t}$ for $c_0 = \frac{1-\varepsilon}{(1+\varepsilon)^2}$. By the holomorphic property of $z \mapsto k_z(\mathbf{x},\mathbf{y})$ on $S_\theta$, Cauchy's integral formula for the $m$-th derivative yields:
\begin{equation*}
    \partial_t^m k_t(\mathbf{x},\mathbf{y}) = \frac{m!}{2\pi i} \int_{\gamma_t} \frac{k_z(\mathbf{x},\mathbf{y})}{(z - t)^{m+1}} \, dz.
\end{equation*}
Applying the complex-time kernel bounds from Theorem~\ref{thm:pointwise_complex_kernel_sector} (combined with standard volume estimates for $w(B(\mathbf{x},\sqrt{t}))$), we bound the integrand on $\gamma_t$ by
\begin{equation*}
    \sup_{z \in \gamma_t} |k_z(\mathbf{x},\mathbf{y})| \leq \frac{C}{w(B(\mathbf{x},\sqrt{t}))} \exp\left(- c'' c_0 \frac{d(\mathbf{x},\mathbf{y})^2}{t}\right).
\end{equation*}
Estimating the integral via the length of $\gamma_t$ ($2\pi \varepsilon t$) and the denominator $|z-t|^{m+1} = (\varepsilon t)^{m+1}$, we get:
\begin{align*}
    |\partial_t^m k_t(\mathbf{x},\mathbf{y})| &\leq \frac{m!}{2\pi} \frac{2\pi \varepsilon t}{(\varepsilon t)^{m+1}} \sup_{z \in \gamma_t} |k_z(\mathbf{x},\mathbf{y})| \\
    &\leq \frac{C_m}{t^m w(B(\mathbf{x},\sqrt{t}))} \exp\left(- c_m \frac{d(\mathbf{x},\mathbf{y})^2}{t}\right),
\end{align*}
where $C_m = \frac{m! C}{\varepsilon^m}$ and $c_m = c'' c_0$.
\end{proof}

\end{document}